\documentclass[12pt]{article}

\usepackage{amsmath, amssymb, amsthm, bm, graphicx, xcolor, physics, newtxtext, newtxmath, tikz, eso-pic, enumitem, float}
\usetikzlibrary {arrows.meta}

\usepackage{geometry}
\usepackage{hyperref}
\hypersetup{breaklinks, colorlinks=true, linkcolor=blue, urlcolor=cyan, citecolor=magenta}

\usepackage[lastpage, user]{zref}
\usepackage{fancyhdr}
\fancypagestyle{plain}{}

\usepackage[numbered, framed]{matlab-prettifier}
\newcommand*\CircleAroundChar[2][\small]{\tikz[baseline=(char.base)]{\node[shape=circle, draw, inner sep=1pt](char){#1#2};}}

\theoremstyle{plain}
\newtheorem{theorem}{Theorem}[section]

\theoremstyle{definition}
\newtheorem{definition}[theorem]{Definition}
\newtheorem{proposition}[theorem]{Proposition}
\newtheorem{remark}[theorem]{Remark}
\newtheorem{corollary}[theorem]{Corollary}
\newtheorem{lemma}[theorem]{Lemma}

\newtheorem{conjecture}[theorem]{Conjecture}

\renewenvironment{proof}{{\noindent \bf Proof.}}{\qed}

\usepackage[english]{babel}
\usepackage[autostyle, english = american]{csquotes}
\MakeOuterQuote{"}

\usepackage{eucal}
\let\mathscr\undefined
\usepackage[mathscr,scaled=1.15]{urwchancal}

\let\mathcalcopy\mathcal
\usepackage{ifthen}
\renewcommand{\mathcal}[1]
{\ifthenelse{\equal{#1}{E}}{\mathscr E}
{\ifthenelse{\equal{#1}{H}}{\mathscr H}
{\mathcalcopy #1}}}

\renewcommand{\leq}{\leqslant}
\renewcommand{\geq}{\geqslant}
\renewcommand{\subset}{\subseteq}

\usepackage{multirow,booktabs}
\usepackage{calc}
\usepackage[numbers]{natbib}

\newcommand{\veps}{\varepsilon}
\newcommand{\bbR}{\mathbb R}

\begin{document}

%


\title{Ohta\textendash Kawasaki energy for amphiphiles:\\asymptotics and phase-field simulations}
\author{Qiang Du\thanks{Department of Applied Physics and Applied Mathematics, and Data Science Institute, Columbia University, New York, NY 10027, USA. Email: qd2125@columbia.edu.}\qquad James M. Scott\thanks{Department of Applied Physics and Applied Mathematics, Columbia University, New York, NY 10027, USA. Email: jms2555@columbia.edu.}\qquad Zirui Xu\thanks{Department of Applied Physics and Applied Mathematics, Columbia University, New York, NY 10027, USA. Email: zx2250@columbia.edu. Corresponding author.}}
\date{}

\maketitle

\begin{abstract}
We study the minimizers of a degenerate case of the Ohta\textendash Kawasaki energy, defined as the sum of the perimeter and a Coulombic nonlocal term. We start by investigating radially symmetric candidates which give us insights into the asymptotic behaviors of energy minimizers in the large mass limit. In order to numerically study the problems that are analytically challenging, we propose a phase-field reformulation which is shown to Gamma-converge to the original sharp interface model. Our phase-field simulations and asymptotic results suggest that the energy minimizers exhibit behaviors similar to the self-assembly of amphiphiles, including the formation of lipid bilayer membranes.
\end{abstract}



\tableofcontents


\section{Introduction}

\subsection{Background and motivation}
\label{Background and motivation}

As soft condensed matter, amphiphiles are known to form various structures in aqueous environments. An amphiphilic molecule usually consists of a hydrophilic head and a hydrophobic tail connected by a covalent bond. Consequently, amphiphiles spontaneously arrange themselves in water in such a way that the hydrophobic tails are segregated from water, protected by the hydrophilic heads. Soft matter systems tend to self-assemble into lower-dimensional structures such as surfaces, curves and points, giving rise to sheet-like membranes, polymer networks and colloidal dispersions, respectively \cite[Pages 107 and 108]{schwarz2002bicontinuous}. One particularly important example is the bilayer membrane formed by lipids in water, which exhibits both rigidity and fluidity in that the membrane resists deformation while allowing rapid lateral diffusion of lipid molecules within each monolayer. The elasticity of the membrane is very different from those of solid materials such as aluminum foil and plastic film. The membrane is soft, which is a crucial property for biological cells and artificial liposomes. The typical energies required to bend a membrane are small enough for thermal fluctuations at room temperature to be important \cite[Pages 3 and 4]{brandt2011molecular}. Indeed, the bending elasticity of membranes is only a high-order effect, as will be discussed in Section \ref{subsec: Elastica functional and Helfrich energy}.

At the macroscopic level, the Helfrich energy introduced in 1973 proved to be a successful continuum model for describing the elasticity of the lipid bilayer membrane (see Appendix \ref{sec:Wilmore}). In this model, the membrane is treated as a two-dimensional surface of zero thickness, with its energy given by the surface integral of a quadratic function in the principle curvatures. However, the actual membrane is of a bilayer structure and nonzero thickness (usually a few nanometers). In order to gain a detailed knowledge at the microscopic level, an atomistic molecular dynamics simulation was carried out in 1992 \cite{damodaran1992structure}. Albeit accurate, such a simulation was time-consuming and thus restricted to a relatively small spatio-temporal scale (60 nm$^3\,\times\,$0.2 ns), rendering the physical processes of interest out of reach\footnote{In recent years, there have been significant improvements using machine learning and neural networks, and the state-of-the-art for $1.27\times10^8$ atoms is $0.8$ ns per day on 4560 nodes of the Summit supercomputer with 27360 GPUs \cite{jia2020pushing}. If they use an NVIDIA A100 GPU like we do in this paper, it would take them 25.6 years instead of 1 day. To put things into perspective, there are about $5\times10^6$ and $10^9$ lipid molecules in a 1 \textmu m$^2$ lipid bilayer and in the plasma membrane of a small animal cell, respectively, with a typical lipid molecule (phosphatidylcholine) consisting of 130 atoms. If we also take into account the surrounding solvent molecules, then the number of atoms would be raised to the power of 3/2 (assuming that the simulation box is a cube). Therefore, for our study, the atomistic molecular dynamics simulations seem out of reach, especially because of the relatively slow kinetics of soft matter systems.}. Since our primary interest is not in individual atoms but rather in the collective behaviors of large numbers of atoms, it is natural to group several neighboring atoms into a single bead, leading to the so-called coarse-graining methods, which reduce the degrees of freedom and accelerate the computation \cite{ingolfsson2014lipid,brandt2011molecular,marrink2019computational,sadeghi2020large}.

On an even coarser scale, a smooth density function is often used to represent the spatial distribution of each type of atoms or atom groups at the mesoscopic level. In 1986, Ohta and Kawasaki derived a density functional theory from statistical physics to explain the mesoscopic periodic patterns formed by diblock copolymers \cite{choksi2003derivation}. This theory was later generalized to triblock copolymers \cite{ren2003triblock,xu2022ternary} as well as mixtures of diblock copolymers and homopolymers \cite{choksi2005diblock}. The latter generalization was recently shown to be capable of modeling the lipid bilayer membrane along with its fusion processes \cite{han2020pathways}, and we study a special case of such a generalization in this paper.

We focus on the sharp interface limit (also known as the strong segregation limit). In our energy, we let $U$ and $V$ denote the regions occupied by the hydrophobic tails and hydrophilic heads of the lipids, respectively. The rest of the space is occupied by water. As their names suggest, the hydrophobic tails are insoluble in water, while the hydrophilic heads are assumed to be miscible with water. Therefore, the interfacial tension exists only on the interface of $U$, but not on the $V$-water interface. An additional Coulombic term accounts for the covalent bonding between the hydrophobic tails and hydrophilic heads.
This energy has been studied in the small mass regime in connection with the spherical micelle formed by amphiphilic surfactants in water \cite{bonacini2016ground}. In the large mass regime, although the question remains largely open, it is believed that the energy minimizers might resemble the lipid bilayer membrane \cite[Page 4]{bonacini2016ground}. 
In fact, a variant of this energy (which makes use of the 1-Wasserstein distance) has been proposed to model the lipid bilayer membrane at the mesoscopic level \cite{peletier2009partial,lussardi2014variational}.

For two subsets $U,V$ of $\bbR^n$ satisfying $U\cap V=\varnothing$, we study the minimization problem of the following energy
\begin{equation}
\label{sharp energy nonrelaxed}
E(U,V) = \text{Per}\;U +\gamma N(U,V),
\end{equation}
under the mass constraints $|U|=m$ and $|V|=\zeta m$, where $\gamma$, $m$ and $\zeta$ are positive constants. The local term $\text{Per}\;U$ is the standard perimeter of $U$ (which equals the surface area for smooth $U\subseteq\bbR^3$). The nonlocal term $N$ is defined as follows
\begin{equation*}
2N(U,V)=\int_U\int_UG(\vec x\!-\!\vec y)\dd{\vec y}\dd{\vec x}-\frac2\zeta\int_U\int_VG(\vec x\!-\!\vec y)\dd{\vec y}\dd{\vec x}+\frac1{\zeta^2}\int_V\int_VG(\vec x\!-\!\vec y)\dd{\vec y}\dd{\vec x},
\end{equation*}
where the Newtonian kernel $G(\vec x\!-\!\vec y)=\big(4\pi|\vec x\!-\!\vec y|\big)^{-1}$ if $n=3$, and $G(\vec x\!-\!\vec y)=-\ln|\vec x\!-\!\vec y|/(2\pi)$ if $n=2$. In this paper we are mainly concerned with the three-dimensional (3-D, i.e., $n=3$) case which is physically most relevant, and we will also consider the 2-D case ($n=2$) for comparison purposes. Note that the 1-D case ($n=1$) with $G(x\!-\!y)=-|x\!-\!y|/2$ and $\zeta=1$ has been solved in \cite[Section 3]{van2008copolymer}: any local minimizer consists of one or multiple non-overlapping bilayer(s), such as two bilayers $VUV0VUV$ (where 0 represents a layer of water of arbitrary thickness, and any $U$ layer is twice as thick as any $V$ layer), and the global minimizer selects the local minimizer whose $U$ layers are of thickness close to $\sqrt[3]{12/\gamma}$. We expect those 1-D results can be generalized to $\zeta\neq1$.

Intuitively, we can imagine that $U$ and $V$ uniformly carry equal amounts of positive and negative charges, respectively, so that the total electrostatic potential energy arising from the electrostatic interactions is given by $N(U,V)$. We define the associated electrostatic potential $\phi$ as
\begin{equation}
\label{eqn: electrostatic potential}
\phi(\vec x)=\int_UG(\vec x\!-\!\vec y)\dd{\vec y}-\frac{1}{\zeta}\int_VG(\vec x\!-\!\vec y)\dd{\vec y},\quad \vec x\in\bbR^n.
\end{equation}
According to \cite[Equation (2.6)]{bonacini2016ground}, by noticing $-\Delta\phi=\bm1_U-\bm1_V/\zeta$, we can rewrite $N$ as
\begin{equation}
\label{alternative expression of N(U,V)}
N(U,V)=\frac12\int_{\bbR^n}\phi(\vec x)\Big(\bm1_U\!-\!\frac{\bm1_V}\zeta\Big)(\vec x)\dd{\vec x}=-\frac12\int_{\bbR^n}\phi(\vec x)\Delta\phi(\vec x)\dd{\vec x}=\frac12\int_{\bbR^n}\big|\nabla\phi(\vec x)\big|^2\dd{\vec x},
\end{equation}
where $-\nabla\phi$ is the electrostatic field.

We mainly focus on the minimizers of the energy $E$ in the large mass regime $m\gg1$ with $\gamma$ fixed. Or equivalently, we can fix $m$ and let $\gamma\gg1$, thanks to the scaling properties of the local and nonlocal terms.

\subsection{Results in the existing literature}
\label{subsec: Results in existing literature}

Throughout this subsection we fix $\gamma=1$ and let $m$ vary. In Proposition \ref{existing qualitative results} we recall some existing results in 3-D ($n=3$). It is believed that these results can be generalized to other dimensions $n\geqslant 2$ \cite[Section 3.1]{bonacini2016ground}.  

\begin{proposition}
\label{existing qualitative results}
The following properties of $E$ are known in the literature for $\zeta=1$, and their proofs can be generalized to any $\zeta>0$:
\begin{enumerate}[label=\protect\CircleAroundChar{\arabic*}]
\item For any $m>0$, the global minimizer exists \cite[Theorem 3.1]{bonacini2016ground}.

\item The global minimizer satisfies the regularity property \cite[Theorem 3.2]{bonacini2016ground}, i.e., there is a representative of $(U,V)$ such that $U$ and $V$ are both open and bounded with finitely many connected components, and that $\partial U$ is of differentiability class $C^\infty$.

\item The global minimizer satisfies the screening property \cite[Theorem 2.1]{bonacini2016ground}, i.e., there is a representative of $(U,V)$ such that $\phi>0$ in $U\cup V$, and that $\phi=0$ in $\mathbb R^3\backslash(U\cup V)$.

\item There exists $m_0>0$ such that a spherical micelle (see Definition \ref{definition of micelle candidate} and Figure \ref{MorphologyCoefficient}-c) is the unique global minimizer for $m<m_0$, and that it is not a global minimizer for $m>m_0$ \cite[Theorem 3.5]{bonacini2016ground}.

\item The global minimum of $E$ has a lower bound and an upper bound. As $m\to\infty$, both bounds scale linearly with $m$ \cite[Section 4]{van2008copolymer}. 

\end{enumerate}
\end{proposition}

Next we give a brief review of existing studies for various $\zeta$.

\subsubsection*{The case of \texorpdfstring{$\zeta=1$}{zeta 1}}

In the large mass regime $m\gg1$, only some qualitative properties of the energy minimizers are known in the literature, e.g., there is a uniform bound on the mean curvature of $\partial U$ for the global minimizer for $m\geqslant1$ \cite[Page 9]{bonacini2016ground}. The exact global minimizer is unknown and is conjectured in several works (see \cite[Figure 2(c) and Page 9]{bonacini2016ground} and \cite[Bottom of Page 78]{van2008partial}) to resemble a planar bilayer membrane, cut off at large distance.

In the literature there is a variant of the energy \eqref{sharp energy nonrelaxed}, with the nonlocal term $N$ replaced by the 1-Wasserstein distance. For this variant, as $m\to\infty$, it is energetically preferable for $U$ and $V$ to form a closed bilayer membrane with an approximately uniform thickness \cite{peletier2009partial,lussardi2014variational}. In \cite[Bottom of Page 21]{van2008partial} van Gennip mentioned a failed attempt to generalize this result from such a variant to the problem \eqref{sharp energy nonrelaxed}.

\subsubsection*{The case of \texorpdfstring{$\zeta\ll1$}{small zeta}}

In the limit of $\zeta\to0$, the problem \eqref{sharp energy nonrelaxed} reduces to the so-called surface charge model \cite[Section 6]{bonacini2016optimal}. For the global minimizer, the negative charge is concentrated on the surface of $U$, and acts like a Faraday shield canceling out the positive charge carried by $U$, i.e., we have $\phi=0$ on $\bbR^3\backslash\overline U$.

As mentioned above, for $\zeta=1$ and $m\gg1$ the global minimizer should locally resemble the bilayer membrane shown in Figure \ref{MorphologyCoefficient}-a, where a layer of $V$ of approximately uniform thickness surrounds $U$. For $\zeta\leqslant1$ and $m\gg1$, it is natural to expect that the global minimizer takes on a similar bilayer membrane structure, with the thickness of the $V$ layer converging to 0 as $\zeta\to0$.

\subsubsection*{The case of \texorpdfstring{$\zeta\gg1$}{big zeta}}

In the limit of $\zeta\to\infty$, the mass of $V$ becomes infinitely large, and the negative charge carried by $V$ becomes infinitely dilute, therefore the problem \eqref{sharp energy nonrelaxed} reduces to the liquid drop model \cite{10.1063/5.0148456} where we only consider the positive charge carried by $U$. It is widely believed that the infimum of the energy in the liquid drop model is attained by a ball for $m\ll1$, and is approached by many equally large distant balls for $m\gg1$.

Therefore, for $\zeta\gg1$, it is natural to expect the global minimizer to be a micelle for $m\ll1$, and to be the union of many equally large non-overlapping micelles for $m\gg1$.

\subsection{Our contributions}
\label{subsec: contributions}
We now summarize our main contributions and then outline the overall structure of the rest of this paper.

As mentioned in Section \ref{subsec: Results in existing literature}, regarding the global minimizer for $\zeta=1$ and $m\gg1$, several works (e.g., \cite[Figure 2(c) and Page 9]{bonacini2016ground} and \cite[Bottom of Page 78]{van2008partial}) documented the belief that $U$ is approximately a large disk with a radius of order $\sqrt{m}$ and with a thickness of order 1, and that $V$ is of an approximately uniform thickness surrounding $U$, as can be seen from our numerical simulations in Figure \ref{simulation liposome local minimizer}. However, our numerical results indicate that such a disk-shaped membrane has higher energy than a liposome (see Definition \ref{liposome candidate defintion}) for $\zeta=1$ and $m\gg1$. In Figure \ref{simulation liposome local minimizer}, the disk-shaped membrane is slightly thicker near its rim, and thus the rim carries an energy penalty of order $\sqrt{m}$, proportional to the circumference. Therefore for the disk-shaped membrane, the energy-to-mass ratio $E/m$ converges with order $1/\sqrt m$, which is consistent with \cite[Theorem 8]{van2008copolymer}. Meanwhile, according to Corollary \ref{simple asymptotics of the optimal liposome candidate}, the convergence rate for the liposome is $1/m$. Note that the infimum of $E/m$ is attained in the limit of $m\to\infty$ \cite[Proposition 8.3]{bonacini2016ground}. Our work suggests that the liposome is the global minimizer for $\zeta=1$ and $m\gg1$, since a sphere minimizes the Helfrich energy in Proposition \ref{calculate undetermined coefficients of Helfrich energy}.

In the existing literature, the equal mass case $\zeta=1$ has been the focal point of most studies (e.g., \cite{bonacini2016optimal,bonacini2016ground,van2008partial,van2011h,van2009stability,van2008copolymer}, except for a 1-D study \cite{choksi2005diblock}). In this paper we consider the general cases $\zeta\in(0,\infty)$. As mentioned in the paragraph following \eqref{sharp energy nonrelaxed}, the 1-D results should be qualitatively the same for different $\zeta$. But as we will illustrate in Figure \ref{MorphologyCoefficient}, as $\zeta$ increases, the optimal morphology in 3-D should undergo transitions from the bilayer membrane to cylindrical micelle to spherical micelle. As mentioned in Section \ref{subsec: Results in existing literature}, in the large mass regime, it is natural to expect that for sufficiently large $\zeta$, the global minimizer consists of many spherical micelles which are scattered and non-overlapping, as depicted in Figure \ref{MorphologyCoefficient}-c. Meanwhile, for intermediately large $\zeta$, it is natural to expect the global minimizer to resemble a cylindrical micelle, see Figure \ref{MorphologyCoefficient}-b.

In Sections \ref{subsec:Liposome candidates} and \ref{subsec: Asymptotics of liposome candidates}, we use asymptotic analysis to study the energy minimizer among radially symmetric liposome candidates for $\zeta>0$ and $m\gg1$. An important finding is that the inner $V$ layer is slightly thicker but has slightly less mass compared to the outer $V$ layer. Interestingly, we notice that such slight differences also exist in a variant model (at least in 2-D), where the Coulombic nonlocal term $N$ is replaced by the 1-Wasserstein distance, as mentioned in Remark \ref{remark on asymptotics of liposome candidates}-\CircleAroundChar{3}. If we impose equal mass on the inner and outer $V$ layers at the expense of optimality, then the energy in our model will increase on the same order as the bending energy, as mentioned in Remark \ref{remark on asymptotics of liposome candidates}-\CircleAroundChar{2}.

For a general bilayer membrane which is not necessarily radially symmetric, up to a suitable rescaling, we can consider its Gamma-limit with vanishing thickness as $\gamma\to\infty$. Under the conjecture that its energy converges to the Helfrich energy, in Proposition \ref{calculate undetermined coefficients of Helfrich energy} we calculate the bending and Gaussian moduli using our asymptotic results for 2-D and 3-D radially symmetric liposome candidates (note that a 2-D circular liposome may be viewed as a 3-D cylindrical bilayer which resembles a very long tube \cite[Middle of Figure 1]{wang2021analytical}).
Our calculation reveals that for small $\zeta$, the Gaussian modulus is positive, and thus the optimal structure may resemble triply periodic minimal surfaces (TPMS) which exist ubiquitously in copolymer systems and biological specimens (see Remark \ref{triply periodic minimum surface remark}).

In Figure \ref{several qualitative changes} we summarize the conjectured candidates for 3-D global minimizers in the large mass regime, with the parameters $\zeta_i$ to be defined later. For $\zeta\!<\!\zeta_1$, we believe that the local structure of the global minimizer resembles a bilayer membrane, and that its global structure is a liposome ($\zeta_0\!<\!\zeta\!<\!\zeta_1$) or approximately a TPMS ($\zeta\!<\!\zeta_0$). For $\zeta_1\!<\!\zeta\!<\!\zeta_2$, its local structure resembles a cylindrical micelle, and its global structure resembles a circular ring. For $\zeta\!>\!\zeta_2$, the global minimizer consists of many scattered droplets of spherical micelles.
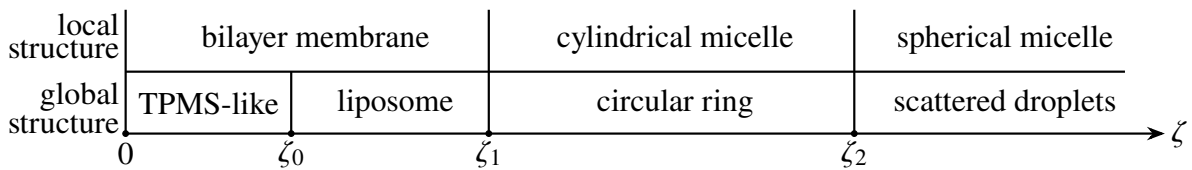
\begin{figure}[H]
\centering
\begin{tikzpicture}

\filldraw[black] (0,-3.1*1.4ex) circle (0.2ex);
\filldraw[black] (8.2843*1.4ex,-3.1*1.4ex) circle (0.2ex);
\filldraw[black] (18.1696*1.4ex,-3.1*1.4ex) circle (0.2ex);
\filldraw[black] (36.4572*1.4ex,-3.1*1.4ex) circle (0.2ex);

\node at (0ex,-3.1*1.4ex-1.4ex) {$0$};
\node at (8.2843*1.4ex,-3.1*1.4ex-1.4ex) {$\zeta_0$};
\node at (18.1696*1.4ex,-3.1*1.4ex-1.4ex) {$\zeta_1$};
\node at (36.4572*1.4ex,-3.1*1.4ex-1.4ex) {$\zeta_2$};
\node at (52*1.4ex+1ex,-3.1*1.4ex) {$\zeta$};

\draw [line width = 0.11*1.4ex](0,0) -- (50*1.4ex,0);
\draw [-{Stealth[length=0.9*1.4ex]},line width = 0.11*1.4ex](0,3.1*1.4ex) -- (0,-3.1*1.4ex) -- (52*1.4ex,-3.1*1.4ex);
\draw [line width = 0.11*1.4ex](8.2843*1.4ex,-3.1*1.4ex) -- (8.2843*1.4ex,0);
\draw [line width = 0.11*1.4ex](18.1696*1.4ex,-3.1*1.4ex) -- (18.1696*1.4ex,3.1*1.4ex);
\draw [line width = 0.11*1.4ex](36.4572*1.4ex,-3.1*1.4ex) -- (36.4572*1.4ex,3.1*1.4ex);

\node[align=right] at (-3.1*1.4ex,1.8*1.4ex) {local\\[-0.3*3.4ex]structure};
\node[align=right] at (-3.1*1.4ex,-1.8*1.4ex) {global\\[-0.2*3.4ex]structure};
\node at (9.5*1.4ex,1.6*1.4ex) {bilayer membrane};
\node at (27.5*1.4ex,1.6*1.4ex) {cylindrical micelle};
\node at (44*1.4ex,1.6*1.4ex) {spherical micelle};
\node at (4.15*1.4ex,-1.6*1.4ex) {TPMS-like};
\node at (13.5*1.4ex,-1.6*1.4ex) {liposome};
\node at (27.5*1.4ex,-1.6*1.4ex) {circular ring};
\node at (44*1.4ex,-1.6*1.4ex) {scattered droplets};

\end{tikzpicture}
\caption{Our conjectures about the 3-D global minimizer for $m\gg1$.}
\label{several qualitative changes}
\end{figure}

In Section \ref{section: Phase-field simulations}, we present numerical evidence for our conjectures. The numerical simulations are based on a phase-field reformulation of the sharp interface model \eqref{sharp energy nonrelaxed}. In order to justify such a reformulation, we prove a Gamma-convergence result in Section \ref{section: Phase-field reformulation}. The novelty of the proof is that our phase-field energy is degenerate: only one of the two order parameters is penalized by the Dirichlet energy, and the potential well has non-isolated minimizers. The rest of this paper is organized as follows. In Section \ref{sec:radialsym}, we restrict ourselves to the simplest radially symmetric case and derive some asymptotic results. In order to better present our asymptotic results, in Section \ref{section Rescaled energy functional} we rescale the energy and propose some conjectures. In Section \ref{section: Phase-field reformulation} we present our phase-field reformulation and prove its Gamma-convergence to the sharp interface model. In Section \ref{section: Phase-field simulations} we present some numerical simulations. In Section \ref{section discussion}, we conclude with remarks about future directions. In Appendix \ref{sec:Wilmore} we provide some background knowledge on the Helfrich and Willmore energies. In Appendix \ref{appendix Calculations of radially symmetric candidates}, we present the detailed calculations in the radially symmetric case. In Appendix \ref{appendix Asymptotics with 1-Wasserstein distance}, we recall the results in a variant model where the Coulombic nonlocal term is replaced by the 1-Wasserstein distance, in order to make it convenient for readers to draw comparisons.

\section{Radially symmetric candidates}\label{sec:radialsym}
In this section, we only consider radially symmetric candidates, thus reducing the variational problem to a finite-dimensional optimization problem which can be analyzed asymptotically. We obtain the asymptotic expansions of the minimum energy and the optimal layer thickness, thus laying the foundation for further studies of non-radially-symmetric cases. In Sections \ref{subsec:Liposome candidates} and \ref{subsec:Micelle candidates}, we consider two types of candidates, namely liposome and micelle, as depicted in Figure \ref{A portion of a radially symmetric candidate}, and then draw comparisons in Section \ref{subsec:Optimal candidates}.

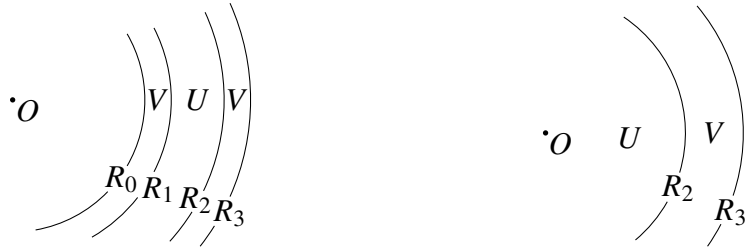
\begin{figure}[H]
\centering
\begin{tikzpicture}
\draw ([shift={(0,0)}]-30:50pt) arc (-30:30:50pt);
\draw ([shift={(0,0)}]-80:50pt) arc (-80:-44:50pt);
\draw ([shift={(0,0)}]-27:60pt) arc (-27:27:60pt);
\draw ([shift={(0,0)}]-60:60pt) arc (-60:-38:60pt);
\draw ([shift={(0,0)}]-25.5:80pt) arc (-25.5:24.5:80pt);
\draw ([shift={(0,0)}]-42:80pt) arc (-42:-32:80pt);
\draw ([shift={(0,0)}]-24:90pt) arc (-24:24:90pt);
\draw ([shift={(0,0)}]-38:90pt) arc (-38:-32:90pt);
\fill (0,0) circle[radius=1pt];
\node at (6pt,-4pt) {$O$};
\node at (55pt,0pt) {$V$};
\node at (70pt,0pt) {$U$};
\node at (85pt,0pt) {$V$};
\node at (41pt,-30pt) {$R_0$};
\node at (55pt,-34pt) {$R_1$};
\node at (69pt,-38pt) {$R_2$};
\node at (82pt,-43pt) {$R_3$};
\end{tikzpicture}
\hspace{100pt}
\begin{tikzpicture}
\draw ([shift={(0,0)}]-16.5:52.9pt) arc (-16.5:56:52.9pt);
\draw ([shift={(0,0)}]-50:52.9pt) arc (-50:-29.5:52.9pt);
\draw ([shift={(0,0)}]-18.5:74.8pt) arc (-18.5:40:74.8pt);
\draw ([shift={(0,0)}]-35:74.8pt) arc (-35:-27.5:74.8pt);
\fill (0,0) circle[radius=1pt];
\node at (6pt,-4pt) {$O$};
\node at (32pt,-2pt) {$U$};
\node at (64pt,0pt) {$V$};
\node at (50pt,-21pt) {$R_2$};
\node at (70pt,-30pt) {$R_3$};
\end{tikzpicture}
\caption{Left: a portion of a liposome candidate. Right: a portion of a micelle candidate. Both are radially symmetric.}
\label{A portion of a radially symmetric candidate}
\end{figure}

\subsection{Liposome candidates}
\label{subsec:Liposome candidates}

Liposome candidates are concentric rings or shells as shown on the left of Figure \ref{A portion of a radially symmetric candidate}. More specifically, we have the following definition.
\begin{definition}
\label{liposome candidate defintion}
We say $(U,V)$ is a liposome candidate, if $U,V\subset\bbR^n$ satisfy
\begin{enumerate}[label=\protect\CircleAroundChar{\arabic*}]
\item $U = B(R_2)\backslash B(R_1)$ and $V = \big(B(R_3)\backslash B(R_2)\big)\cup\big(B(R_1)\backslash B(R_0)\big)$ for some $0<R_0<R_1<R_2<R_3$, where $B(R_i)$ is an $n$-dimensional ball of radius $R_i$ centered at the origin $O$.
\item $R_i$ satisfies the following mass constraints
\begin{align}
&R_3^n-R_0^n=(\zeta\!+\!1)\,(R_2^n\!-\!R_1^n)\,,&\label{volume constraint radial candidates}\\
&R_2^n\!-\!R_1^n=
\left\{
\begin{aligned}
&m/\pi,&&n=2,\\
&3m/(4\pi),&&n=3.
\end{aligned}
\right.\label{2d 3d volume constraint radial candidates}
\end{align}
\end{enumerate}
\end{definition}

\begin{corollary}
\label{simple asymptotics of the optimal liposome candidate}
[of Theorem \ref{asymptotics of the optimal liposome candidate}] With $\zeta$ and $\gamma$ fixed, as $m\to\infty$, the minimizer of $E(U,V)$ among the liposome candidates satisfies
\begin{equation}
\label{differences of radii converge to constants}
(R_1\!-\!R_0\,,\,R_2\!-\!R_1\,,\,R_3\!-\!R_2) \to \sqrt[3]{3/(\gamma\zeta\!+\!\gamma)}\,(\zeta\,,\,2\,,\,\zeta),
\end{equation}
with the following asymptotics
\begin{equation*}
\frac{E(U,V)}m =
\left\{
\begin{aligned}
& \sqrt[3]{\gamma\frac{\zeta\!+\!1}{8/9}} + \frac{8 \pi ^2}{5}\frac{\zeta ^2\!+\!4 \zeta\!+\!1}{\gamma  (\zeta\!+\!1) m^2} + O\Big(\frac{1}{m^3}\Big),&&\text{for}\;n=2,\\
& \sqrt[3]{\gamma\frac{\zeta\!+\!1}{8/9}} + \frac{4 \pi}{15m}\frac{\zeta ^2\!+\!4 \zeta\!+\!16}{\big(\gamma(\zeta\!+\!1)/3\big)^{2/3}} + O\Big(\frac{1}{m^{3/2}}\Big),&&\text{for}\;n=3.\\
\end{aligned}
\right.
\end{equation*}
\end{corollary}

\begin{corollary}
\label{simple equal volume asymptotics proposition}
[of Proposition \ref{equal volume asymptotics proposition}]
Under the additional assumption that the inner and outer $V$ layers have the same mass \cite[Equation (5.16)]{van2008partial}, i.e., $R_3^n\!-\!R_2^n=R_1^n\!-\!R_0^n$, with $\zeta$ and $\gamma$ fixed, as $m\to\infty$, the minimizer of $E(U,V)$ among the liposome candidates still satisfies \eqref{differences of radii converge to constants}, with the following asymptotics
\begin{equation*}
\frac{E(U,V)}m =
\left\{
\begin{aligned}
& 
\sqrt[3]{\gamma\frac{\zeta\!+\!1}{8/9}}
+
24 \pi ^2\frac{ 2 \zeta ^2\!+\!8 \zeta\!+\!7}{5 \gamma  (\zeta\!+\!1) m^2}
+
O\Big(\frac{1}{m^3}\Big)
,&&\text{for}\;n=2,\\
& \sqrt[3]{\gamma\frac{\zeta\!+\!1}{8/9}} + \frac{4 \pi}{5 m}\frac{7 \zeta ^2\!+\!28 \zeta\!+\!32}{\big(\gamma(\zeta\!+\!1)/3\big)^{2/3}} + O\Big(\frac{1}{m^{3/2}}\Big),&&\text{for}\;n=3.
\end{aligned}
\right.
\end{equation*}
\end{corollary}

\begin{remark}$ $
\label{remark on nonrescaled liposome asymptotics}
\begin{enumerate}[label=\protect\CircleAroundChar{\arabic*}]
\item The additional assumption in Corollary \ref{simple equal volume asymptotics proposition} was used in \cite[Pages 83 and 84]{van2008partial}. Albeit not the optimal choice, it was thought to be only slightly non-optimal. Our calculations show that the leading-order term of the energy remains the same under this additional assumption, but that the next-order term (corresponding to the bending energy) becomes 6 to 21 times as large, depending on $n$ and $\zeta$. In other words, the bending energy can be decreased by dropping this additional assumption and allowing the inner and outer $V$ layers to have different masses.

\item From \eqref{differences of radii converge to constants} we know that $R_{i+1}\!-\!R_i$ converges to a constant as $m\to\infty$. This can serve as an illustration that the lipid bilayer has an intrinsically preferred thickness, which is consistent with \cite[Remark 4]{van2008copolymer}. In the large mass regime $m\gg1$, the bilayer should be of thickness of order $1$.
\end{enumerate}
\end{remark}

\subsection{Micelle candidates}
\label{subsec:Micelle candidates}

\begin{definition}
\label{definition of micelle candidate}
In Definition \ref{liposome candidate defintion}, if $R_0=R_1=0$, then $(U,V)$ is called a micelle candidate. (Note that for $n=2$, in some literature \cite{alama2022core} the terminology core-shell is used instead.)
\end{definition}

\begin{proposition}
\label{energy of micelle candidate}
The energy \eqref{sharp energy nonrelaxed} of a micelle candidate is $E(U,V) = \text{Per}\;U +\gamma N(U,V)$, where
\begin{equation*}
\text{Per}\;U=\left\{
\begin{aligned}
&2\pi \sqrt{m/\pi},&&n=2,\\
&4\pi \big(3m/(4\pi)\big)^{2/3},&&n=3,
\end{aligned}
\right.
\end{equation*}
and
\begin{equation}
N(U,V) = \left\{
\begin{aligned}
&\pi\zeta^{-2}(\zeta\!+\!1)\big((\zeta\!+\!1)\ln(\zeta\!+\!1)\!-\!\zeta\big)(m/\pi)^2/8,&&n=2,\\
&\pi\zeta ^{-2}(\zeta\!+\!1)\big(4\zeta\!+\!6\!-\!6(\zeta\!+\!1)^{2/3}\big)\big(3m/(4\pi)\big)^{5/3}/15,&&n=3.
\end{aligned}
\right.
\end{equation}
\end{proposition}

\begin{proof}
Proposition \ref{energy of micelle candidate} is a consequence of Proposition \ref{proposition: energy for liposome candidates} with $R_0=R_1=0$ and
\begin{equation*}
R_3^n/(\zeta\!+\!1)=R_2^n=
\left\{
\begin{aligned}
&m/\pi,&&n=2,\\
&3m/(4\pi),&&n=3.
\end{aligned}
\right.
\end{equation*}
\end{proof}

\begin{corollary}
\label{Minimal energy of micelle candidates}
For a micelle candidate, if $n=2$, then $E/m$ attains its minimum at
\begin{equation*}
m=4 \pi  \Big(\gamma  (\zeta\!+\!1) \big((\zeta\!+\!1) \ln(\zeta\!+\!1)-\zeta\big)/\zeta ^2\Big)^{-2/3},
\quad\text{and}\quad
\min_{m>0}\frac{E}{m} = \frac{3}{2} \sqrt[3]{\gamma \zeta ^{-2}(\zeta\!+\!1) \big((\zeta\!+\!1) \ln(\zeta\!+\!1)-\zeta\big)};
\end{equation*}
if $n=3$, then $E/m$ attains its minimum at
\begin{equation*}
m=20 \pi  \zeta ^2\big(\gamma(\zeta\!+\!1)\big)^{-1}\big/\big(2 \zeta\!+\!3\!-\!3 (\zeta\!+\!1)^{2/3}\big),
\quad\text{and}\quad
\min_{m>0}\frac{E}{m} = \frac{9}{2} \sqrt[3]{\gamma \zeta ^{-2}(\zeta\!+\!1) \big(2 \zeta\!+\!3\!-\!3 (\zeta\!+\!1)^{2/3}\big)/15}.
\end{equation*}
\end{corollary}

\begin{remark}$ $
\begin{enumerate}[label=\protect\CircleAroundChar{\arabic*}]
\item Due to the screening property (see Proposition \ref{existing qualitative results}-\CircleAroundChar{3}), for the disjoint union of non-overlapping components, the energy is simply the sum of the energy for each connected component. Therefore, in the large mass regime $m\gg1$, we can construct a candidate consisting of many non-overlapping micelles, each of which has a mass close to the optimal mass given in Corollary \ref{Minimal energy of micelle candidates} (similar to the construction mentioned in \cite[Page 96]{van2008copolymer}). Such a candidate will attain the optimal energy-to-mass ratio asymptotically as $m\to\infty$. In fact, our energy in Proposition \ref{energy of micelle candidate} is of a similar form to \cite[Equation (6.1)]{choksi2010small}. Using the same approach as \cite[Proof of Lemma 6.2]{choksi2010small}, we can prove that the optimal way to allocate the mass $m$ is as follows: let $m_i$ (assuming $m_i>0$) be the mass of each connected component, then all but one $m_i$ (say, $\{m_i\}_{i\neq1}$) must be equal, with $m_1=O(1)$. Therefore, in the large mass limit $m\to\infty$, the optimal way to allocate the mass yields asymptotically the same energy-to-mass ratio as Corollary \ref{Minimal energy of micelle candidates}.

\item According to \cite[Theorem 8]{van2008copolymer}, we can regard a 2-D micelle as the limiting case of an infinitely long 3-D cylindrical micelle, with asymptotically the same energy-to-mass ratio.

\item 
We have also considered the candidates with $R_0=0$ and $R_1>0$, and they always have higher energy-to-mass ratios than the liposome candidates in the large mass limit, according to our calculations. Therefore, they are omitted in our discussions.

\end{enumerate}
\end{remark}

\subsection{Optimal candidates}
\label{subsec:Optimal candidates}
We now compare the candidates that are considered in Sections \ref{subsec:Liposome candidates} and \ref{subsec:Micelle candidates}. By comparing the leading-order term of $E/m$ in Corollaries \ref{simple asymptotics of the optimal liposome candidate} and \ref{Minimal energy of micelle candidates}, we are led to the belief that in the large mass regime in 3-D, as $\zeta$ increases, the preferred morphology should be successively bilayer membrane, cylindrical micelle, and spherical micelle, as shown in Figure \ref{MorphologyCoefficient}. This picture is qualitatively consistent with the predictions by other theories \cite[Section 1.3.2.3]{ben1994statistical}.

\begin{conjecture}
\label{conjecture: infimum energy linear in m}
Let $n=3$. For any fixed $\gamma>0$ and $\zeta>\zeta_0$ (where $\zeta_0$ is defined in Remark \ref{triply periodic minimum surface remark}-\CircleAroundChar{2}, for reasons that will be mentioned therein), we have $\inf\limits_{U,V} E(U,V)/m\to c(\zeta)\,\sqrt[3]{\gamma}$ as $m\to\infty$, where
\begin{equation*}
c(\zeta)=
\left\{
\begin{aligned}
\sqrt[3]{9(\zeta\!+\!1)/8}\,,&&\text{if}\;0<\zeta\leqslant\zeta_1,\\
\frac{3}{2} \sqrt[3]{\zeta ^{-2}(\zeta\!+\!1) \big((\zeta\!+\!1) \ln(\zeta\!+\!1)-\zeta\big)}\,,&&\text{if}\;\zeta_1<\zeta\leqslant\zeta_2,\\
\frac{9}{2} \sqrt[3]{\zeta ^{-2}(\zeta\!+\!1) \big(2 \zeta\!+\!3\!-\!3 (\zeta\!+\!1)^{2/3}\big)/15}\,,&&\text{if}\;\zeta_2<\zeta<\infty.
\end{aligned}
\right.
\end{equation*}
Here $\zeta_1\approx1.81696$ is the unique nonzero root of $\zeta+\zeta^2/3 =(\zeta\!+\!1) \ln (\zeta\!+\!1)$, and $\zeta_2\approx3.64572$ is the unique nonzero root of $5\big((\zeta\!+\!1) \ln (\zeta\!+\!1)-\zeta\big)=9\big(2 \zeta - 3 (\zeta\!+\!1)^{2/3} + 3\big)$.
\end{conjecture}

\begin{remark}$ $
\begin{enumerate}[label=\protect\CircleAroundChar{\arabic*}]
\item The piecewise function $c(\zeta)$ is plotted on the left of Figure \ref{MorphologyCoefficient}, where $\zeta_1$ and $\zeta_2$ are also marked. As $\zeta$ increases, the minimum energy is attained on different branches indicated by different colors, and the preferred morphology should transition from bilayer membrane to cylindrical micelle to spherical micelle, which are illustrated by Figure \ref{MorphologyCoefficient}-a, b and c, respectively.

\item In \cite[Section 5.3]{van2008partial}, van Gennip also compared the energy among the bilayer membrane, cylindrical micelle, and spherical micelle candidates. However, van Gennip only considered the case of $\zeta=1$, and concluded that the bilayer membrane has the lowest energy among all the candidates. Our calculations reveal that cylindrical micelle and spherical micelle may have lower energy for $\zeta\neq1$, and that amphiphiles can self-assemble into not only sheet-like membranes, but also polymer networks and colloidal dispersions.
\end{enumerate}
\end{remark}
\begin{figure}[H]
\centering
\includegraphics[width=0.43\textwidth]{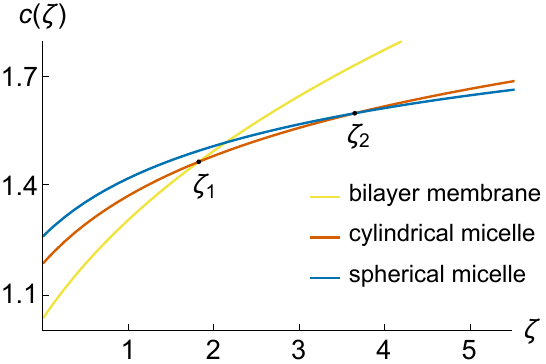}
\hspace{5ex}
\begin{minipage}[t][-5pt][b]{0.23\textwidth}
\includegraphics[scale=0.4]{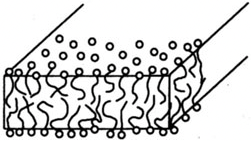}\hspace{-100pt}\raisebox{45pt}{a}\\
\raisebox{25pt}{\includegraphics[scale=0.4]{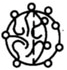}}
\includegraphics[scale=0.4]{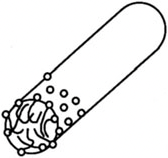}
\hspace{-99pt}\raisebox{16pt}{c}
\hspace{57pt}\raisebox{5pt}{b}
\end{minipage}
\caption{Left: Three branches of $c(\zeta)$. Right: a, bilayer membrane; b, cylindrical micelle; c, spherical micelle. The right image is reproduced from \cite[Figure 1.6]{ben1994statistical}.}
\label{MorphologyCoefficient}
\end{figure}

\section{Rescaled energy functional}
\label{section Rescaled energy functional}

In Section \ref{sec:radialsym}, we have briefly presented some asymptotic results as $m\to\infty$. In order to make it convenient and rigorous to translate our asymptotic results within the Gamma-convergence framework, in this section we rescale the energy functional so that the minimizer converges in Radon measure to lower-dimensional structures (surfaces, curves and points). In Section \ref{subsec: Asymptotics of liposome candidates}, we elaborate on the asymptotic results of the liposome candidates for the rescaled energy, and in Section \ref{subsec: Elastica functional and Helfrich energy}, we propose some conjectures about Gamma-convergence of the rescaled energy.

\begin{definition}
\label{rescaled energy}
Analogous to \cite{peletier2009partial}, we define a rescaled version of \eqref{sharp energy nonrelaxed}:
\begin{equation*}
F_\rho(u,v) =
\left\{
\begin{aligned}
&\rho^{1-d}\,\text{Per}\;U + \rho^{-2-d}N(U,V),&&\text{if}\;(u,v)\in K_\rho\,,\\
&\infty,&&\text{otherwise},
\end{aligned}
\right.
\end{equation*}
where $d$ can be chosen from $\{1,2,3\}$ as needed, $U=\text{supp}(u)$ and $V=\text{supp}(v)$. Moreover,
\begin{equation*}
K_\rho=\big\{(u,v)\in BV(\mathbb R^n;\{0,\rho^{-d}\})\times L^2(\mathbb R^n;\{0,\rho^{-d}\}):vu=0\;\text{a.e.},\;\mbox{$\int$}u=m,\;\mbox{$\int$}v=\zeta m\big\}.
\end{equation*}
\end{definition}

\begin{proposition}
\label{rescaling proposition}
Denote $U_\rho$ to be the dilation of $U$ with a scale factor of $1/\rho$, i.e., $\vec x\in U\Leftrightarrow U_\rho\ni \vec x/\rho$, and similarly for $V_\rho$. Then we have $|U_\rho|/\mbox{$\int$}u=\rho^{d-n}=|V_\rho|/\mbox{$\int$}v$ and $\rho^{d-n}F_\rho(u,v) = \text{Per}\;U_\rho + N(U_\rho\,,V_\rho)$, where $u=\bm1_U/\rho^d$ and $v=\bm1_V/\rho^d$.
\end{proposition}
\begin{proof}
Use the scaling properties of the perimeter and Newtonian potential.
\end{proof}

\begin{remark}
According to Proposition \ref{rescaling proposition}, $F_\rho$ is equivalent to $E$ up to a rescaling, and thus the results in Proposition \ref{existing qualitative results} also apply to $F_\rho$: for any $\zeta,m,\rho>0$ and for any $d$, the global minimizer of $F_\rho$ exists and satisfies the screening property.
\end{remark}

\subsection{Detailed asymptotics of liposome candidates}
\label{subsec: Asymptotics of liposome candidates}
Throughout this subsection, we let $d=1$. We consider liposome candidates given by $U = B(R_2)\backslash B(R_1)$ and $V = \big(B(R_3)\backslash B(R_2)\big)\cup\big(B(R_1)\backslash B(R_0)\big)$ for some $0<R_0<R_1<R_2<R_3$ satisfying $R_3^n-R_0^n=(\zeta\!+\!1)\,(R_2^n\!-\!R_1^n)$ and
\begin{equation*}
R_2^n\!-\!R_1^n=
\left\{
\begin{aligned}
&m\rho/\pi,&&n=2,\\
&3m\rho/(4\pi),&&n=3,
\end{aligned}
\right.
\end{equation*}
where $B(R_i)$ is an $n$-dimensional ball of radius $R_i$ centered at the origin $O$ (see also Definition \ref{liposome candidate defintion}).

\begin{corollary}
\label{rescaled liposome asymptotics}
[of Proposition \ref{rescaling proposition} and Theorem \ref{asymptotics of the optimal liposome candidate}.]
Let $d=1$, with $\zeta$ and $m$ fixed, as $\rho\to0$, the minimizer of $F_\rho$ among the liposome candidates has the following asymptotics if $n=2$,
\begin{equation*}
\begin{aligned}
&
\frac{F_\rho(u,v)}{m} = \sqrt[3]{\frac{\zeta\!+\!1}{8/9}} + 8 \pi ^2\rho^2\frac{\zeta ^2\!+\!4 \zeta\!+\!1}{5(\zeta\!+\!1) m^2} + O\big(\rho^3\big),\hspace{-25pt}
&
R_2\!-\!R_1 = \rho\sqrt[3]{\frac{24}{\zeta\!+\!1}}
+
O\big(\rho^3\big),\\
&
\big(R_3^2\!-\!R_2^2\big)-\big(R_1^2\!-\!R_0^2\big)=\frac{4 \rho^2\zeta  (\zeta\!+\!2)}{\sqrt[3]{3(\zeta\!+\!1)^2}}+
O\big(\rho^3\big),
\hspace{-25pt}
&
\frac{R_1\!+\!R_2}2 = \frac{m}{4\pi}\sqrt[3]{\frac{\zeta\!+\!1}{3}}
+
O\big(\rho^2\big),
\\
&R_1\!-\!R_0 = \rho\sqrt[3]{\frac{3\zeta^3}{\zeta\!+\!1}}
+
\frac{2 \pi\rho^2}{m/\zeta}\frac{\zeta\!+\!2}{\zeta\!+\!1}
+
O\big(\rho^3\big),\hspace{-25pt}
&R_3\!-\!R_2 = \rho\sqrt[3]{\frac{3\zeta^3}{\zeta\!+\!1}}
-
\frac{2 \pi\rho^2}{m/\zeta}\frac{\zeta\!+\!2}{\zeta\!+\!1}
+
O\big(\rho^3\big),
\end{aligned}
\end{equation*}
and has the following asymptotics if $n=3$,
\begin{equation*}
\begin{aligned}
&\frac{F_\rho(u,v)}{m} = \sqrt[3]{\frac{\zeta\!+\!1}{8/9}} + \frac{4 \pi\rho^2}{15m}\frac{\zeta ^2\!+\!4 \zeta\!+\!16}{\big((\zeta\!+\!1)/3\big)^{2/3}} + O\big(\rho^3\big),\hspace{-20pt}
&R_2\!-\!R_1 = \rho\sqrt[3]{\frac{24}{\zeta\!+\!1}}+
O\big(\rho^3\big),\\
&\big(R_3^3\!-\!R_2^3\big)-\big(R_1^3\!-\!R_0^3\big)=\frac{\rho^2\sqrt{m}\zeta(\zeta\!+\!2)}{\sqrt{\pi (\zeta\!+\!1)/6}}+O\big(\rho^3\big),
\hspace{-20pt}
&
\frac{R_1\!+\!R_2}2 =
\frac{\sqrt[6]{(\zeta\!+\!1)/3}}{2\sqrt{2 \pi/m} }
+
O\big(\rho^2\big),\\
&R_1\!-\!R_0 = \rho\sqrt[3]{\frac{3\zeta^3}{\zeta\!+\!1}}
+
\frac{\sqrt{8\pi} (\zeta\!+\!2) \zeta \rho^2}{\sqrt{m}\sqrt[6]{3(\zeta\!+\!1)^5}}
+
O\big(\rho^3\big),\hspace{-20pt}
&R_3\!-\!R_2 = \rho\sqrt[3]{\frac{3\zeta^3}{\zeta\!+\!1}}
-
\frac{\sqrt{8\pi} (\zeta\!+\!2) \zeta \rho^2}{\sqrt{m}\sqrt[6]{3(\zeta\!+\!1)^5}}
+
O\big(\rho^3\big).
\end{aligned}
\end{equation*}
\end{corollary}

\begin{corollary}
\label{rescaled liposome asymptotics under equal mass assumption}
[of Propositions \ref{rescaling proposition} and \ref{equal volume asymptotics proposition}.]
Under the additional assumption that the inner and outer $V$ layers have the same mass \cite[Equation (5.16)]{van2008partial}, i.e., $R_3^n\!-\!R_2^n=R_1^n\!-\!R_0^n$, let $d=1$, with $\zeta$ and $m$ fixed, as $\rho\to0$, the minimizer of $F_\rho$ among the liposome candidates has the following asymptotics if $n=2$,
\begin{equation*}
\begin{aligned}
&\frac{F_\rho(u,v)}{m} = \sqrt[3]{\frac{\zeta\!+\!1}{8/9}}
+
24 \pi ^2\rho^2\frac{ 2 \zeta ^2\!+\!8 \zeta\!+\!7}{5(\zeta\!+\!1) m^2}
+
O\big(\rho^3\big),\hspace{-0pt}
&\\
&R_1\!-\!R_0
=
\rho\sqrt[3]{\frac{3\zeta^3}{\zeta\!+\!1}}
+
\frac{6 \pi\rho^2}{m/\zeta}\frac{\zeta\!+\!2}{\zeta\!+\!1}
+
O\big(\rho^3\big),\hspace{-0pt}
&R_2\!-\!R_1
=
\rho\sqrt[3]{\frac{24}{\zeta\!+\!1}}
+
O\big(\rho^3\big),\\
&R_3\!-\!R_2
=
\rho\sqrt[3]{\frac{3\zeta^3}{\zeta\!+\!1}}
-
\frac{6 \pi\rho^2}{m/\zeta}\frac{\zeta\!+\!2}{\zeta\!+\!1}
+
O\big(\rho^3\big),
\hspace{-0pt}&\frac{R_1\!+\!R_2}2
=
\frac{m}{4\pi}\sqrt[3]{\frac{\zeta\!+\!1}{3}}
+
O\big(\rho^2\big),
\end{aligned}
\end{equation*}
and has the following asymptotics if $n=3$,
\begin{equation*}
\begin{aligned}
&\frac{F_\rho(u,v)}{m} = \sqrt[3]{\frac{\zeta\!+\!1}{8/9}} + \rho^2\frac{4 \pi}{5 m}\frac{7 \zeta ^2\!+\!28 \zeta\!+\!32}{\big((\zeta\!+\!1)/3\big)^{2/3}} + O\big(\rho^3\big),\hspace{-0pt}
&\\
&R_1\!-\!R_0
=
\rho\sqrt[3]{\frac{3\zeta^3}{\zeta\!+\!1}}
+
\frac{\sqrt{8 \pi}(\zeta \!+\!2) \zeta\rho^2}{\sqrt{m}\sqrt[6]{(\zeta\!+\!1)^5/3^5}}
+
O\big(\rho^3\big),\hspace{-0pt}
&R_2\!-\!R_1
=
\rho\sqrt[3]{\frac{24}{\zeta\!+\!1}}
+
O\big(\rho^3\big),\\
&R_3\!-\!R_2
=
\rho\sqrt[3]{\frac{3\zeta^3}{\zeta\!+\!1}}
-
\frac{\sqrt{8 \pi}(\zeta \!+\!2) \zeta\rho^2}{\sqrt{m}\sqrt[6]{(\zeta\!+\!1)^5/3^5}}
+
O\big(\rho^3\big),
\hspace{-0pt}
&\frac{R_1\!+\!R_2}2
=
\frac{\sqrt[6]{ (\zeta\!+\!1)/3}}{2\sqrt{2 \pi/m} }
+
O\big(\rho^2\big).
\end{aligned}
\end{equation*}
\end{corollary}

\begin{remark}
\label{remark on asymptotics of liposome candidates}$ $
\begin{enumerate}[label=\protect\CircleAroundChar{\arabic*}]
\item Due to the curvature of the bilayer, the inner $V$ layer is slightly thicker than the outer $V$ layer (i.e., $R_1\!-\!R_0>R_3\!-\!R_2$), and $F_\rho$ is penalized for bending on the second order (i.e., the $\rho^2$ term corresponding to the bending energy).

\item Under the equal mass assumption $R_3^n\!-\!R_2^n=R_1^n\!-\!R_0^n$, the second-order term of $F_\rho$ is 6 to 21 times as large as that of the optimal liposome candidate whose inner $V$ layer has slightly less mass than the outer $V$ layer.

\item Under the equal mass assumption, the difference in the thickness between the inner and outer $V$ layers is asymptotically three times that of the optimal liposome candidate, in both 2-D and 3-D. This relation is also true at least in 2-D even if the Coulombic nonlocal term $N$ is replaced by the 1-Wasserstein distance (see Appendix \ref{appendix Asymptotics with 1-Wasserstein distance}). It is therefore natural to wonder how universal this relation can be.
\end{enumerate}
\end{remark}

\subsection{Conjectures about the Gamma-limits}
\label{subsec: Elastica functional and Helfrich energy}

We now propose some conjectures about the Gamma-limits of the energy functional $F'_\rho:=\big(F_\rho-c(\zeta)\,m\big)/\rho^2$ as $\rho\to 0$. Throughout this subsection, $d$ is chosen to be the codimension of the expected geometry of the global minimizer as shown in Figure \ref{MorphologyCoefficient}:
\begin{equation*}
d=
\left\{
\begin{aligned}
1\,,&&\text{if}\;0<\zeta\leqslant\zeta_1\;\text{and}\;n=2,\\
2\,,&&\text{if}\;\zeta_1<\zeta<\infty\;\text{and}\;n=2,
\end{aligned}
\right.
\qquad
d=
\left\{
\begin{aligned}
1\,,&&\text{if}\;0<\zeta\leqslant\zeta_1\;\text{and}\;n=3,\\
2\,,&&\text{if}\;\zeta_1<\zeta\leqslant\zeta_2\;\text{and}\;n=3,\\
3\,,&&\text{if}\;\zeta_2<\zeta<\infty\;\text{and}\;n=3,
\end{aligned}
\right.
\end{equation*}
where $\zeta_1$ and $\zeta_2$ are defined in Conjecture \ref{conjecture: infimum energy linear in m}.

\begin{conjecture}
\label{2-D conjecture of Gamma-convergence}
Let $n=2$ and $\rho\to0$. For $\zeta\in(0,\zeta_1)$, the Gamma-limit of $F'_\rho$ is the elastica functional defined for closed $W^{2\hspace{0.2ex},\hspace{0.2ex}2}$ curves in $\mathbb R^2$ in the sense of Radon measure, similar to the elastica functional $\mathcal{W}$ mentioned in \cite[Theorem 4.1]{peletier2009partial}. For $\zeta\in(\zeta_1,\infty)$, the Gamma-limit of $F'_\rho$ is a mass partition functional defined for weighted Dirac delta point measures, similar to the mass partition functional $\mathsf{E}_0^\text{2d}$ mentioned in \cite[Theorem 6.1]{choksi2010small} (see also the mass partition functional $E_0$ mentioned in \cite[Theorem 4.2]{alama2022core}).
\end{conjecture}

\begin{conjecture}
\label{3-D conjecture of Gamma-convergence}
Let $n=3$ and $\rho\to0$. For $\zeta\in(\zeta_0,\zeta_1)$, the Gamma-limit of $F'_\rho$ is a quadratic form in the principal curvatures defined for closed surfaces in the sense of Radon measure, similar to \cite[Conjecture 2.4]{lussardi2014variational}. For $\zeta\in(\zeta_1,\zeta_2)$, the Gamma-limit of $F'_\rho$ is the elastica functional defined for closed $W^{2\hspace{0.2ex},\hspace{0.2ex}2}$ curves in $\mathbb R^3$ in the sense of Radon measure, which is a 3-D generalization of \cite[Theorem 4.1]{peletier2009partial}. For $\zeta\in(\zeta_2,\infty)$, the Gamma-limit of $F'_\rho$ is a mass partition functional defined for weighted Dirac delta point measures, similar to the mass partition functional $\mathsf{E}_0^\text{3d}$ mentioned in \cite[Theorem 4.3]{choksi2010small}.
\end{conjecture}

If the first statement in Conjecture \ref{2-D conjecture of Gamma-convergence} and the second statement in Conjecture \ref{3-D conjecture of Gamma-convergence} are correct, then Proposition \ref{elastica functional minimized by a circle} tells us that the global minimizer should resemble a circle.

\begin{proposition}
\label{elastica functional minimized by a circle}
Let $C$ be a closed $W^{2\hspace{0.2ex},\hspace{0.2ex}2}$ curve in $\mathbb R^3$ with a prescribed length, and let $\kappa$ be its curvature, then the elastica functional $\int_C\kappa^2\dd{s}$ is minimized when $C$ is a circle.
\end{proposition}

\begin{proof}
The length of $C$ is given by $\int_C1\dd{s}$ and is assumed to be fixed. By Cauchy\textendash Schwarz inequality we have $\int_C1^2\dd{s}\int_C\kappa^2\dd{s}\geqslant\big(\int_C|\kappa|\dd{s}\big)^2\geqslant4\pi^2$, where the second inequality is due to \cite[Theorem 7.2.3]{willmore1993riemannian} and becomes an equality if and only if $C$ is a planar convex curve.
\end{proof}

\begin{proposition}
\label{calculate undetermined coefficients of Helfrich energy}
Assuming the first statement in Conjecture \ref{3-D conjecture of Gamma-convergence} is correct, i.e., for $n=3$ and $\zeta\in(\zeta_0,\zeta_1)$, the Gamma-limit of $F'_\rho$ is of the following quadratic form (cf. Appendix \ref{sec:Wilmore} for Helfrich and Willmore energies)
\begin{equation}
\label{Helfrich energy}
\int_S(\lambda_1 H^2+\lambda_2 K)\dd{A},
\end{equation}
where $H=(\kappa_1\!+\!\kappa_2)/2$ (with $\kappa_1$ and $\kappa_2$ being principal curvatures, positive if $S$ is a sphere), $K=\kappa_1\kappa_2$, and $\lambda_1,\lambda_2\in\mathbb R$. If the recovery sequence is given by a bilayer with the middle $U$ layer of uniform thickness (similar to the construction in \cite[Section 4]{lussardi2014variational}), then we have
\begin{equation}
\label{moduli in Helfrich energy}
\lambda _1 = \frac{4}{15}\frac{1\!+\!4\zeta\!+\!\zeta^2}{\big((\zeta\!+\!1)/3\big)^{2/3}},\quad\lambda _2 = \frac{4\!-\!4\zeta\!-\!\zeta^2}{5 \big((\zeta\!+\!1)/3\big)^{2/3}}.
\end{equation}
\end{proposition}

\begin{proof}
According to Proposition \ref{Lagrange multiplier equation} and the Proof of Proposition \ref{proposition: energy for liposome candidates}, the liposome candidates shown in Corollary \ref{rescaled liposome asymptotics} satisfy the screening property. For any fixed $U$, there exists a unique $V$ satisfying the screening property \cite[Remark 4.2]{bonacini2016optimal}, which is a necessary condition for a minimizer (Proposition \ref{existing qualitative results}-\CircleAroundChar{3}). Therefore, for the Radon measure of a sphere in 3-D, the recovery sequence is by assumption given by radially symmetric liposome candidates in 3-D. According to \cite[Theorem 8]{van2008copolymer}, the 2-D liposome candidate can be regarded as a 3-D cylindrical bilayer which resembles a very long tube \cite[Middle of Figure 1]{wang2021analytical}. For a sphere in 3-D with radius $R$, we have $K=1/R^2$ and $H=1/R$. According to Corollary \ref{rescaled liposome asymptotics}, we have
\begin{equation*}
\frac{4 \pi}{15}\frac{\zeta ^2\!+\!4 \zeta\!+\!16}{\big((\zeta\!+\!1)/3\big)^{2/3}}=4\pi R^2(\lambda_1/R^2+\lambda_2/R^2).
\end{equation*}
For a cylinder in 3-D with length $L$ and radius $R$, we have $K=0$ and $H=1/(2R)$. Similarly, we have
\begin{equation*}
8 \pi ^2\frac{\zeta ^2\!+\!4 \zeta\!+\!1}{5(\zeta\!+\!1) m}L=2\pi R L\big(\lambda_1/(2R)^2+\lambda_2\cdot0\big).
\end{equation*}
In Corollary \ref{rescaled liposome asymptotics}, the radius $R$ of the 2-D liposome candidate can be approximated by $(R_1\!+\!R_2)/2\approx m\sqrt[3]{(\zeta\!+\!1)/3}\big/(4\pi)$. Therefore the coefficients $\lambda_1$ and $\lambda_2$ can be determined.
\end{proof}

\begin{remark}$ $
\label{triply periodic minimum surface remark}
\begin{enumerate}[label=\protect\CircleAroundChar{\arabic*}]
\item
By the Gauss\textendash Bonnet formula, for closed surfaces in the same homotopy class (i.e., closed surfaces of the same genus $g$), we have $\int_S K\dd{A}=4\pi(1\!-\!g)$, therefore \eqref{Helfrich energy} can be reduced to the Willmore energy $\int_S H^2\dd{A}$ as long as no topological change occurs.

\item
Define $\zeta_0=2(\sqrt{2}\!-\!1)\approx0.82843$.
We are now ready to explain the requirement $\zeta>\zeta_0$ imposed in Conjectures \ref{conjecture: infimum energy linear in m} and \ref{3-D conjecture of Gamma-convergence}. If such conjectures were also true for $\zeta<\zeta_0$, then we would still be able to prove \eqref{moduli in Helfrich energy} in Proposition \ref{calculate undetermined coefficients of Helfrich energy}. However, for $\zeta<\zeta_0$ we have $\lambda_2>0$. Therefore, the quadratic form in Proposition \ref{calculate undetermined coefficients of Helfrich energy} is not positive semi-definite, and it is energetically more favorable to have zero mean curvature and negative Gaussian curvature. As pointed out by \cite[Page 143]{porte1994micellar}, if $\lambda_2$ becomes positive, then it is preferable for a surface to have large genus $g$, or many "holes". For example, triply periodic minimal surfaces shown in Figure \ref{fig: TPMS} possess many "holes". Since the quadratic form is indefinite, it is desirable for those triply periodic minimal surfaces to have infinitesimal lattice constants. However, we expect an equilibrium to be attained at a finitely small lattice constant (on the order of $\rho$) due to counteracting higher-order terms.

\item It surprised us to find that $\lambda_2$ can be positive for relatively small $\zeta$. However, with hindsight we may come up with the following intuitive but non-rigorous explanation. For relatively large $\zeta$, bending of the bilayer is penalized because there will be a slight difference in the thickness between the inner and outer $V$ layers, as mentioned in Remark \ref{remark on asymptotics of liposome candidates}-\CircleAroundChar{1}. For relatively small $\zeta$, such penalty may be relatively small because the $V$ layers have vanishing thickness as $\zeta\to0$. Meanwhile, the middle $U$ layer may prefer saddle-splay deformations, because the $U$ layer becomes more "spread out" in this way and thus the Coulombic repulsion within the $U$ layer may decrease.

\item Our numerical simulations in Figure \ref{figure gyroid simulation} support the observation in Remark \ref{triply periodic minimum surface remark}-\CircleAroundChar{2}. More specifically, our numerical results indicate that a gyroid-like minimizer has lower energy-to-mass ratio than the planar bilayer for $\zeta=0.6<\zeta_0$, and thus the latter cannot be a global minimizer for $m\gg1$. For $\zeta=1>\zeta_0$, our numerical results suggest the opposite.

\item Various cubic bicontinuous structures resembling triply periodic minimal surfaces can be observed in copolymer systems and biological specimens (see \cite{deng1998three,lin2017tunable} and \cite[Section 4.1]{han2018overview}), e.g., in the endoplasmic reticulum, Golgi apparatus and mitochondria \cite[Bottom of Page 144]{schwarz2002bicontinuous}. Our finding provides a plausible explanation for such phenomena in the parameter regime $\zeta<\zeta_0$, and also demonstrates the ability of the Ohta\textendash Kawasaki energy to capture such aspects of amphiphile self-assembly.

\end{enumerate}
\end{remark}

\begin{figure}[H]
\centering
\begin{minipage}{0.18\textwidth}
\includegraphics[scale=0.4]{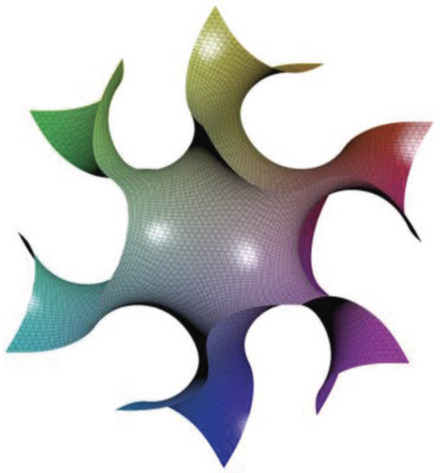}
\end{minipage}
\begin{minipage}{0.15\textwidth}
\includegraphics[scale=0.4]{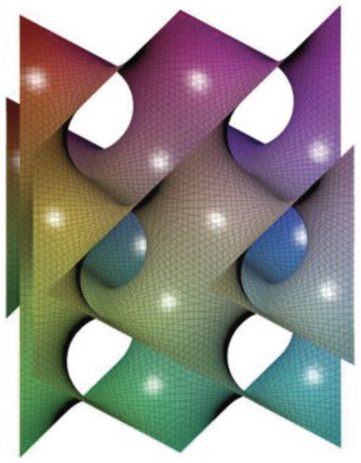}
\end{minipage}
\begin{minipage}{0.17\textwidth}
\includegraphics[scale=0.4]{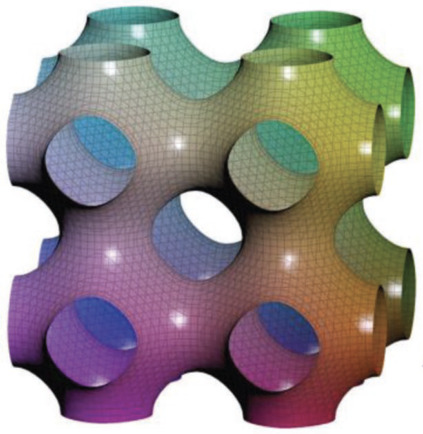}
\end{minipage}
\caption{Triply periodic minimal surfaces. Left: gyroid. Middle: Schwarz diamond. Right: Schwarz primitive. Reproduced from \cite[Figure 1]{han2018overview}.}
\label{fig: TPMS}
\end{figure}

\begin{remark}$ $
\begin{enumerate}[label=\protect\CircleAroundChar{\arabic*}]
\item
Conjectures \ref{2-D conjecture of Gamma-convergence} and \ref{3-D conjecture of Gamma-convergence} states that our nonlocal problem Gamma-converges to a local problem. Our rationale behind those conjectures is as follows. According to the screening property (Proposition \ref{existing qualitative results}-\CircleAroundChar{3}), each connected component of the minimizer satisfies the charge neutrality condition, and the total energy is just the sum of the energy for each connected component. Therefore we can consider each connected component separately. For each connected component, although the problem is nonlocal, the screening property may cause such nonlocality to decay sufficiently fast as $\rho\to0$.

\item Analogously, the screening property is also the reason that the electrostatic potential energy of a crystal is an extensive property. The electrostatic interactions are nonlocal, but an extensive property is local, i.e., it is proportional to the size of the crystal. In a variant model where the nonlocal term is given by the 1-Wasserstein distance, a similar phenomenon occurs: the nonlocal (or global) problem converges to the elastica bending energy which is a local problem \cite[Section 9.3]{peletier2009partial}.
\end{enumerate}
\end{remark}

\section{Phase-field reformulation}
\label{section: Phase-field reformulation}

In order to provide some evidence for our conjectures in Section \ref{subsec: Elastica functional and Helfrich energy}, we propose a phase-field reformulation which will be used for the numerical simulations in Section \ref{section: Phase-field simulations}. We also prove the Gamma-convergence of our phase-field reformulation to a Gamma-limit, which is shown to be equivalent to the original problem \eqref{sharp energy nonrelaxed}.

The phase-field method \cite{du2006simulating} is a useful tool to study the motions of interfaces. The basic idea is to use a narrow but diffuse interface in place of the sharp interface, and the thickness of the interfacial layer is controlled by a diffuseness parameter $\veps$. The interface is implicitly given by the level set of a smooth function, so there is no need to explicitly track the interface. Our phase-field reformulation is very similar to a previous work by two of us \cite[Section III.A]{10.1063/5.0148456}.

\subsection{Diffuse interface energy}
\label{subsec Diffuse interface energy}

On a bounded domain $\Omega\subset\bbR^n$ with $|\Omega|\geq m+\zeta m$, we define the following phase-field functional
\begin{equation*}
\mathcal E(u,v)=\mathcal P(u,v)+\gamma\mathcal N(u,v)+\mathcal C(u,v),
\end{equation*}
where $\mathcal P$ is the diffuse interface version of the perimeter
\begin{equation*}
\mathcal P(u,v) = \frac{\veps}{2}\int_\Omega|\nabla u|^2+\int_\Omega\frac{W(u,v)}{\veps},
\end{equation*}
with the diffuseness parameter $\veps>0$, and the double-well potential $W$ given by (see Figure \ref{W(u,v) plot})
\begin{equation}
\label{formula of potential W}
2W(u,v)/27:=4(u\!-\!u^2)^2/3+\min\{v,0\}^2+\min\{1\!-\!v,0\}^2+\min\{1\!-\!u\!-\!v,0\}^2,
\end{equation}
which is to penalize violations of the three conditions: \CircleAroundChar{1} $u\in\{0,1\}$; \CircleAroundChar{2} $0\leq v\leq1$; \CircleAroundChar{3} $u\!+\!v\leq1$. Since $|\nabla v|^2$ is not penalized in our energy, we choose a degenerate well for $v$ instead of the classical double well $(v-v^2)^2$, in order to prevent $v$ from being trapped in the local minima $0$ and $1$.
\begin{figure}[H]
\centering
\includegraphics[width=0.4\textwidth]{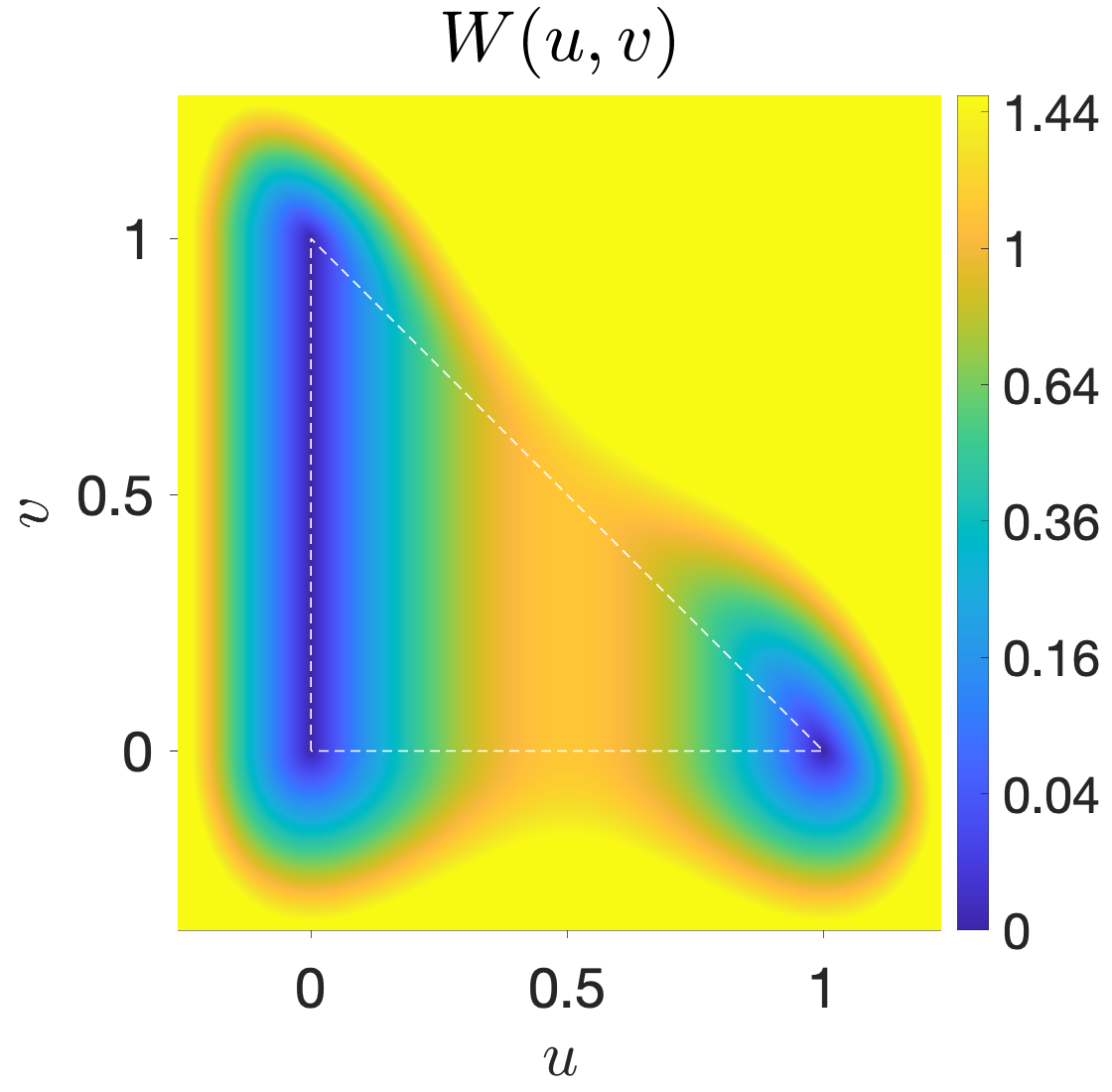}
\caption{Visualization of $W$. For clarity a truncated version $\min\{W,1.5\}$ is plotted.}
\label{W(u,v) plot}
\end{figure}

We choose a function $f$ which strictly increases on $[0,1]$ satisfying $f(0)=0$ and $f(1)=1$. The simplest choice is the identity function $f(z):=z$, which will be used in this section for our proof of the Gamma-convergence. However, for our simulations in Section \ref{section: Phase-field simulations}, we will choose $f$ to be a nonlinear function in order to achieve the numerical efficiency as mentioned in \cite[Section III.A]{10.1063/5.0148456}. As we will see later in Propositions \ref{compactness proposition} and \ref{prop: v must be indicator function}, if $(u,v)$ is a minimizer of $\mathcal E$ for $\veps\ll1$, then $u$ and $v$ should be approximately 0 or 1 at most places in $\Omega$. Therefore, the specific choice of $f$ should have a vanishing effect on the minimizer of $\mathcal E$ as $\veps\to0$.

With penalty coefficients $K_1,K_2>0$, the mass constraint term $\mathcal C$ is defined as follows
\begin{equation*}
2\mathcal C(u,v)=K_1\Big(m-\int_\Omega f(u)\Big)^2+K_2\Big(\zeta m-\int_\Omega f(v)\Big)^2.
\end{equation*}
The nonlocal term $\mathcal N$ is defined as follows
\begin{equation*}
\begin{aligned}
2\mathcal N(u,v)=&\int_\Omega\int_\Omega f\big(u(\vec x)\big)G(\vec x\!-\!\vec y)f\big(u(\vec y)\big)\dd{\vec y}\dd{\vec x}-\frac2\zeta\int_\Omega\int_\Omega f\big(u(\vec x)\big)G(\vec x\!-\!\vec y)f\big(v(\vec y)\big)\dd{\vec y}\dd{\vec x}\\
&+\frac1{\zeta^2}\int_\Omega\int_\Omega f\big(v(\vec x)\big)G(\vec x\!-\!\vec y)f\big(v(\vec y)\big)\dd{\vec y}\dd{\vec x},
\end{aligned}
\end{equation*}
where the nonlocal kernel $G$ is the Green's function of negative Laplacian $-\Delta$. We consider two types of boundary conditions for $G$. The first type is no boundary conditions, under which $G$ is the Green's function in the free space (i.e., the so-called fundamental solution). The second type is periodic boundary conditions, in which case we require $\Omega$ to be a rectangle ($n=2$) or a rectangular cuboid ($n=3$), and according to \cite[Equations (2.1) and (2.2)]{choksi2010small}, $G$ is the periodic extension of the fundamental solution plus a continuous function, satisfying $\int_\Omega G=0$. In Section \ref{subsubsection Justification for boundary conditions}, we justify that the above two types of boundary conditions are equivalent. In fact, for $n=1$ such equivalence has already manifested itself in the striking similarity between Theorems 1 and 2 in \cite{van2008copolymer} (note that $VUVU\cdots U$ therein is equivalent to $VUVU\cdots UV$, since an infinitesimal $V$ block can be appended to the end without energy penalty, as long as $V$-$0$ interfaces are not penalized).

\subsection{The sharp interface Gamma-limit}

In this subsection, we aim to prove the Gamma-convergence of $\mathcal E_\veps$ as $\veps\to0$ and $K_1,K_2\to\infty$ (where the subscript $\veps$ is used to emphasize the dependence of $\mathcal E$ on $\veps$).


As mentioned before, in our proof of the Gamma-convergence, we let $f$ be the identity function. As a first step, we let $K_1,K_2\to\infty$, so that the admissible function space of $\mathcal E_\veps$ is restricted to $\big\{u,v\in L^2(\Omega):\int_\Omega v=\zeta \int_\Omega u=\zeta m\big\}$, and thus the constraint term $\mathcal C$ vanishes. For convenience, in most parts of our proof of the Gamma-convergence, we do not explicitly incorporate the mass constraints, because it should be straightforward except in the proof of the limsup inequality (Proposition \ref{Limsup inequality proposition}), where we do provide more details of the mass constraints. In addition, we have
\begin{align}
2\mathcal N(u,v)=&\int_\Omega\int_\Omega u(\vec x)\,G(\vec x\!-\!\vec y)\,u(\vec y)\dd{\vec y}\dd{\vec x}-\frac2\zeta\int_\Omega\int_\Omega u(\vec x)\,G(\vec x\!-\!\vec y)\,v(\vec y)\dd{\vec y}\dd{\vec x}\notag\\
&+\frac1{\zeta^2}\int_\Omega\int_\Omega v(\vec x)\,G(\vec x\!-\!\vec y)\,v(\vec y)\dd{\vec y}\dd{\vec x}\notag\\
=&\int_\Omega\varphi(\vec x)\big(u(\vec x)-v(\vec x)/\zeta\big)\dd{\vec x}=-\int_{\Omega}\varphi\,\Delta\varphi=\int_{\Omega}|\nabla\varphi|^2,\label{linearized nonlocal term}
\end{align}
where $\varphi=G*(u\!-\!v/\zeta)$, i.e., $\varphi(\vec x)=\int_\Omega G(\vec x\!-\!\vec y)\big(u(\vec y)-v(\vec y)/\zeta\big)\dd{\vec y}$\, for any $\vec x\in\Omega$. If we impose no boundary conditions on $G$ (i.e., $G$ is the fundamental solution), then $\varphi$ satisfies $-\Delta\varphi=u-v/\zeta$; if we impose periodic boundary conditions on $G$, then $\varphi$ satisfies $-\Delta\varphi=\text{constant}+u-v/\zeta$ and $\int_{\Omega}\varphi=0$, where the constant is zero as long as $\int_\Omega v=\zeta\int_\Omega u$.

We assume that $\Omega\subset\bbR^n$ is an open bounded set with Lipschitz boundary. We extend the domain of definition of $\mathcal E_\veps$ to $L^2(\Omega)\times L^2(\Omega)$:
\begin{equation*}
\mathcal E_\veps(u,v) =
\left\{
\begin{aligned}
&\frac{\veps}{2}\int_\Omega|\nabla u|^2+\int_\Omega\frac{W(u,v)}{\veps}+\gamma\mathcal N(u,v),&&u\in W^{1,2}(\Omega),\;v\in L^2(\Omega),\;\mbox{$\int_\Omega v\!=\zeta\int_\Omega u=\!\zeta m$},\\
&\infty,&&\text{otherwise}.
\end{aligned}
\right.
\end{equation*}
We want to prove that $\mathcal E_\veps$ Gamma-converges to the following functional defined on $L^2(\Omega)\times L^2(\Omega)$:
\begin{equation*}
\mathcal E_0(u,v) =
\left\{
\begin{aligned}
&|Du|(\Omega)+\gamma\mathcal N(u,v),&&u\in BV(\Omega;\{0,1\}),\;v\in L^2(\Omega;[0,1]),\;vu=0\;\text{a.e.},\;\mbox{$\int_\Omega v\!=\zeta\int_\Omega u=\!\zeta m$},\\
&\infty,&&\text{otherwise},
\end{aligned}
\right.
\end{equation*}
where $|Du|$ is the absolute value of the distributional derivative of $u$. Note that $\mathcal E_0$ is a relaxed version of $E$, similar to \cite[Equation (3.1)]{bonacini2016optimal}.

Our proof of the Gamma-convergence of $\mathcal E_\veps$ to $\mathcal E_0$ follows closely that of the classical Cahn--Hilliard energy functional given in \cite{leoni2013gamma}. The novelty of our proof lies in the degeneracy of the potential $W$ shown in Figure \ref{W(u,v) plot}, i.e., $W$ has non-isolated minimizers $(u,v)\in\{0\}\times[0,1]$. Furthermore, $|\nabla u|^2$ is penalized in our energy while $|\nabla v|^2$ is not. Sternberg \cite[Section 2]{sternberg1988effect} also considered the case where $W$ has non-isolated minimizers, but both $|\nabla u|^2$ and $|\nabla v|^2$ are penalized there. In the rest of this section, we consider an arbitrary sequence of diffuseness parameters $\{\veps_k:k\in\mathbb N\}\subset\bbR_{>0}$ converging to 0. For brevity we abbreviate the subscript $k$ and write $\mathcal E_\veps$, which should be understood as $\mathcal E_{\veps_k}$ for some $k\in\mathbb N$. We also write $\lim\limits_{\veps\to0}\mathcal E_\veps$ instead of $\lim\limits_{k\to\infty}\mathcal E_{\veps_k}$.

\subsubsection{Compactness result}

\begin{definition}
\label{definition equi-integrable}
A sequence of functions $\{f_k:k\in\mathbb N\}\subset L^1(\Omega)$ is called equi-integrable (i.e., uniformly integrable) if: for any $\tau>0$, there exists an $r>0$, such that
\begin{equation*}
\sup_{k\in\mathbb N}\int_{\{|f_k|\geq r\}}|f_k|<\tau.
\end{equation*}
\end{definition}

\begin{remark}
In Definition \ref{definition equi-integrable}, note that $\{f_1\}$ is always equi-integrable. Indeed, we can approximate $f_1$ using a simple function $g$ so that $\int_\Omega|f_1-g|<\tau/2$, and then choose $r$ such that $\int_{\{|f_1|\geq r\}}|g|<\tau/2$, because $\max|g|<\infty$ and the Lebesgue measure of $\{|f_1|\geq r\}$ converges to 0 as $r\to\infty$. Similarly, any finite subset of $\{f_k:k\in\mathbb N\}$ is always equi-integrable, therefore we have an equivalent definition:
\begin{equation*}
\limsup_{k\to\infty}\int_{\{|f_k|\geq r\}}|f_k|<\tau.
\end{equation*}
\end{remark}

\begin{proposition}
\label{compactness proposition}
As $\veps\to0$, if a sequence $\big\{(u_\veps,v_\veps)\big\} \subset W^{1,2}(\Omega) \times L^2(\Omega)$ satisfies
\begin{equation*}
M:=\sup_{\veps} \mathcal{E}_\veps(u_\veps,v_\veps) < \infty\,,
\end{equation*}
then there exists a subsequence $\big\{(u_\veps,v_\veps)\big\}$ (for brevity, this subsequence is not relabeled) and $(u_0,v_0) \in BV(\Omega; \{0,1\} )\times L^2(\Omega;[0,1])$ such that
\begin{equation*}
u_\veps \to u_0 \text{ in } L^2(\Omega),\quad v_\veps \rightharpoonup v_0 \text{ in } L^2(\Omega),\quad v_0\,u_0=0\;\text{a.e.}
\end{equation*}
\end{proposition}

\begin{proof}\\
\noindent\textbf{Step 1:}
We show that $u_\veps^2\!+\!v_\veps^2$ is uniformly bounded in $ L^1(\Omega)$ and equi-integrable. According to \eqref{formula of potential W}, there exists $R>0$ such that $W(u,v)\geq u^2\!+\!v^2$ whenever $u^2\!+\!v^2\geq R$. For any $r \geq R$, we have
\begin{equation*}
\int_{\{u_\veps^2+v_\veps^2 \geq r\}} (u_\veps^2+v_\veps^2)  \leq \int_{ \{ u_\veps^2+v_\veps^2 \geq r\} } W(u_\veps,v_\veps)\leq\mathcal{E}_\veps(u_\veps,v_\veps)\,\veps \leq M\veps\,,
\end{equation*}
which implies that $u_\veps^2\!+\!v_\veps^2$ is equi-integrable. Moreover, since $\Omega$ is bounded, $\int_\Omega (u_\veps^2\!+\!v_\veps^2)$ is bounded by $M\veps+R|\Omega|$.\\[-10pt]

\noindent\textbf{Step 2:} We show that a subsequence of $\{u_\veps\}$ converges pointwise \text{a.e.} to some $u_0\in BV(\Omega;\{0,1\})$ (by mimicking \cite{leoni2013gamma}). Denote $W_1(z):=\min\big\{|z\!-\!z^2|,1\big\}$ and $\eta(z):=\int_0^zW_1(r)\dd{r}$. We have
\begin{equation*}
M\geq\mathcal{E}_\veps(u_\veps,v_\veps)\geq\int_\Omega|\nabla u_\veps|\sqrt{W(u_\veps,v_\veps)}\geq\int_\Omega|\nabla u_\veps|\;|u_\veps-u_\veps^2|\geq\int_\Omega|\nabla u_\veps|\,W_1(u_\veps)=\int_\Omega\big|\nabla\eta(u_\veps)\big|.
\end{equation*}
Denote $\eta_\veps=\eta(u_\veps)$. Since $0\leq W_1\leq1$ and $\eta(0)=0$, we know $|\eta(z)|\leq |z|$ and thus $|\eta_\veps|\leq|u_\veps|$. According to Step 1, $\{u_\veps\}$ and thus $\{\eta_\veps\}$ are uniformly bounded in $L^2(\Omega)$ and thus in $L^1(\Omega)$. Therefore $\{\eta_\veps\}$ is uniformly bounded in $BV(\Omega)$. According to \cite[Theorem 5.5]{evans2015measure}, there exists a subsequence $\{\eta_\veps\}$ (not relabeled) and some $\eta_0\in BV(\Omega)$ such that $\eta_\veps\to\eta_0$ in $L^1(\Omega)$. A further subsequence $\{\eta_\veps\}$ (not relabeled) converges pointwise \text{a.e.} to $\eta_0$. Since $\eta(z)$ is strictly increasing and continuous in $z$, its inverse $\eta^{-1}$ is continuous. Define $u_0=\eta^{-1}(\eta_0)$, since $u_\veps=\eta^{-1}(\eta_\veps)$, we know that $u_\veps$ converges pointwise \text{a.e.} to $u_0$. Moreover, by Fatou's lemma \cite[Lemma 1.83-(i)]{fonseca2007modern} we have $\int_\Omega W_1(u_0)^2\leq\liminf\limits_{\veps\to0}\int_\Omega W_1(u_\veps)^2$, where $\int_\Omega W_1(u_\veps)^2\leq\mathcal{E}_\veps(u_\veps,v_\veps)\,\veps \leq M\veps$. Therefore $W_1(u_0)^2=0$ \text{a.e.}, that is, $u_0\in\{0,1\}$ a.e. So $\eta_0\in\big\{\eta(0),\eta(1)\big\}$ a.e. Since $\eta_0\in BV(\Omega)$, we can write $\eta_0=\eta(0)\,(1\!-\!\bm1_U)+\eta(1)\,\bm1_U$ with a set $U\subset\Omega$ of finite perimeter. Hence, $u_0=\bm1_U$ belongs to $BV(\Omega;\{0,1\})$.\\[-10pt]

\noindent\textbf{Step 3:} We show that the subsequence of $\{u_\veps\}$ obtained in Step 2 converges to $u_0$ in $L^2(\Omega)$. This is guaranteed by Vitali's convergence theorem \cite[Theorems 2.22, 2.24 and 2.29]{fonseca2007modern}, in view of Step 1.\\[-10pt]

\noindent\textbf{Step 4:} We show that a subsequence of $\{v_\veps\}$ converges weakly in $L^2(\Omega)$ to some $v_0\in L^2(\Omega)$. According to Step 1, $\{v_\veps\}$ is uniformly bounded in $L^2(\Omega)$. By \cite[Proposition 2.46-(iv)]{fonseca2007modern}, it has a subsequence (not relabeled) such that $v_\veps\rightharpoonup v_0$ in $L^2(\Omega)$ for some $v_0\in L^2(\Omega)$.\\[-10pt]

\noindent\textbf{Step 5:} We show that $v_0\in L^2(\Omega;[0,1])$ and $v_0\,u_0=0$ a.e. In view of Steps 2, 3 and 4, we can assume that the sequence $\big\{(u_\veps,v_\veps)\big\}$ satisfies $v_\veps\rightharpoonup v_0$ in $L^2(\Omega)$,\; $u_\veps\to u_0$ in $L^2(\Omega)$,\; and $u_\veps\to u_0$ a.e. Our goal is to prove $\int_\Omega W(u_0,v_0)=0$. Since $W(u,v)$ is continuous in $u$, by Fatou's lemma, we have the following
\begin{equation*}
\int_\Omega W(u_0,v_0)\leq\liminf_{\veps\to0}\int_\Omega W(u_\veps,v_0).
\end{equation*}
According to \eqref{formula of potential W}, $W(u,v)$ is convex in $v$, and we know $W_v(u,v):=\partial_vW(u,v)=27\big(\min\{v,0\}+\max\{v\!-\!1,0\}+\max\{v\!+\!u\!-\!1,0\}\big)$, therefore we have
\begin{equation*}
\int_\Omega W(u_\veps,v_0)\leq\int_\Omega W(u_\veps,v_\veps)+\int_\Omega W_v(u_\veps,v_0)\,(v_0\!-\!v_\veps),
\end{equation*}
where the first summand on the right-hand side can be bounded by $\mathcal{E}_\veps(u_\veps,v_\veps)\,\veps$, and the second summand can be split into
\begin{equation*}
\int_\Omega W_v(u_0,v_0)\,(v_0\!-\!v_\veps)+\int_\Omega \big(W_v(u_\veps,v_0)-W_v(u_0,v_0)\big)\,(v_0\!-\!v_\veps),
\end{equation*}
where the first summand converges to 0 because $W_v(u_0,v_0)\in L^2(\Omega)$ and $v_\veps\rightharpoonup v_0$ in $L^2(\Omega)$, and the absolute value of the second summand can be bounded by
\begin{equation*}
\|v_0\!-\!v_\veps\|_{L^2(\Omega)}\cdot\big\|W_v(u_\veps,v_0)-W_v(u_0,v_0)\big\|_{L^2(\Omega)}.
\end{equation*}
Since $\{v_\veps\}$ is uniformly bounded in $L^2(\Omega)$, we know $\|v_0\!-\!v_\veps\|_{L^2(\Omega)}$ is uniformly bounded. Moreover, we have $\big|W_v(u_\veps,v_0)-W_v(u_0,v_0)\big|\leq27|u_\veps\!-\!u_0|$, so $W_v(u_\veps,v_0)\to W_v(u_0,v_0)$ in $L^2(\Omega)$. To summarize, we have proved $\int_\Omega W(u_\veps,v_0)\to0$ as $\veps\to0$. Therefore $W(u_0,v_0)=0$ a.e.
\end{proof}

\subsubsection{Liminf inequality}
\label{Subsection Liminf inequality}

\begin{proposition}
\label{Liminf inequality proposition}
For any $u_0,v_0\in L^2(\Omega)$ and $\{u_\veps\},\{v_\veps\}\subset L^2(\Omega)$ such that both $u_\veps\rightarrow u_0$ and $v_\veps\rightharpoonup v_0$ in $L^2(\Omega)$, we have
\begin{equation*}
\mathcal E_0(u_0,v_0)\leq \liminf_{\veps\to0}\mathcal E_\veps(u_\veps,v_\veps).
\end{equation*}
\end{proposition}

\begin{proof}
We assume that the above right-hand side is finite (otherwise there is nothing to prove). Without loss of generality (by extracting a subsequence if necessary), we further assume that the $\liminf$ is actually a limit. By Proposition \ref{compactness proposition}, we know $u_0\in BV(\Omega;\{0,1\})$, $v_0\in L^2(\Omega;[0,1])$, and $u_0v_0=0$ a.e. We can also assume $\mathcal E_\veps(u_\veps,v_\veps)<\infty$ for all $\veps$, and thus $\{u_\veps\}\subset W^{1,2}(\Omega)$.

\noindent\textbf{Step 1:} We first consider the case where $\gamma=0$. Denote $W_1(z):=\min\big\{|z\!-\!z^2|,1\big\}$, and $\eta(z):=\int_0^zW_1(r)\dd{r}$. Because $W_1(u)^2\leq(u-u^2)^2\leq W(u,v)/18$, we have
\begin{equation*}
\mathcal E_\veps(u_\veps,v_\veps)\geq \frac{\veps}{2}\int_\Omega|\nabla u_\veps|^2+\frac{18}{\veps}\int_\Omega W_1(u_\veps)^2\geq6\int_\Omega|\nabla u_\veps|\,W_1(u_\veps)=6\int_\Omega\big|\nabla\eta(u_\veps)\big|.
\end{equation*}
Since $0\leq W_1\leq1$, we know $\big|\eta(z)\!-\!\eta(r)\big|\leq |z\!-\!r|$ for any $z,r\in\bbR$. Denote $\eta_\veps=\eta(u_\veps)$ and $\eta_0=\eta(u_0)$, then $\eta_\veps\to\eta_0$ in $L^2(\Omega)$ and thus in $L^1(\Omega)$. According to \cite[Theorem 5.2]{evans2015measure}, we have
\begin{equation*}
\liminf_{\veps\to0}\int_\Omega|\nabla\eta_\veps|=\liminf_{\veps\to0}|D\eta_\veps|(\Omega)\geq|D\eta_0|(\Omega).
\end{equation*}
Noticing $\eta_0=\eta(1)\,u_0$ and $\eta(1)=1/6$, we obtain
\begin{equation*}
\liminf\limits_{\veps\to0}\mathcal E_\veps(u_\veps,v_\veps)\geq|Du_0|(\Omega)=\text{Per}_\Omega \{\vec x\in\Omega:u_0(\vec x)=1\}.
\end{equation*}

\noindent\textbf{Step 2:} We now consider the case where $\gamma>0$. We want to prove the following
\begin{equation*}
\lim_{\veps\to0}\mathcal N(u_\veps,v_\veps)=\mathcal N(u_0,v_0).
\end{equation*}
Denote $w=u-v/\zeta$. From \eqref{linearized nonlocal term} we recall $2\mathcal N(u,v)=\int_\Omega\varphi\,w$, where $\varphi(\vec x)=G*w:=\int_\Omega G(\vec x\!-\!\vec y)\,w(\vec y)\dd{\vec y}$. Denote $w_\veps=u_\veps-v_\veps/\zeta$ and $w_0=u_0-v_0/\zeta$ with $\varphi_\veps=G*w_\veps$ and $\varphi_0=G*w_0$, then
\begin{equation*}
\begin{aligned}
2\big|\mathcal N(u_\veps,v_\veps)-\mathcal N(u_0,v_0)\big|&=\Big|\int_\Omega\varphi_\veps\,w_\veps-\int_\Omega\varphi_0\,w_0\Big|\to0,
\end{aligned}
\end{equation*}
where we have used Lemma \ref{weak version of boundedness of inverse Laplacian} and the fact that $w_\veps \rightharpoonup w_0$ in $L^2(\Omega)$.
\end{proof}


\begin{lemma}
\label{boundedness of inverse Laplacian}
Recall that $G$ is the Green's function of $-\Delta$ under either no boundary conditions (i.e., the fundamental solution) or periodic boundary conditions. For $\eta\in L^2(\Omega)$, define $\psi(\vec x)=\int_\Omega G(\vec x\!-\!\vec y)\,\eta(\vec y)\dd{\vec y}$, then there is a constant $C\in\bbR_{>0}$ independent of $\eta$ and $\vec x$, such that $|\psi(\vec x)|\leq C\|\eta\|_{L^2(\Omega)}$.
\end{lemma}

\begin{proof}
According to the Cauchy\textendash Schwarz inequality, we have $$|\psi(\vec x)|^2=\Big|\int_\Omega G(\vec x\!-\!\vec y)\,\eta(\vec y)\dd{\vec y}\Big|^2\leq \int_\Omega G(\vec x\!-\!\vec y)^2\dd{\vec y}\,\int_\Omega \eta(\vec y)^2\dd{\vec y}=\int_\Omega G(\vec x\!-\!\vec y)^2\dd{\vec y}\,\|\eta\|_{L^2(\Omega)}^2.$$ Therefore we only need to estimate $\int_\Omega G(\vec x\!-\!\vec y)^2\dd{\vec y}$.

\noindent\textbf{Step 1:} We first consider the case where $G$ is the fundamental solution with $n=3$. In such case, we have $G(\vec x\!-\!\vec y)=\big(4\pi|\vec x\!-\!\vec y|\big)^{-1}$. Inspired by \cite[Page 159]{gilbarg2001elliptic}, we denote $R=\sqrt[3]{3|\Omega|/(4\pi)}$ so that $|\Omega|=\big|B(\vec x;R)\big|$, where $B(\vec x;R)$ is a ball of radius $R$ centered at $\vec x$. Using the monotonicity of $G$, we obtain for any $\vec x\in\Omega$,
\begin{equation*}
\int_\Omega\frac{\dd{\vec y}}{|\vec x\!-\!\vec y|^2}\leq \int_{B(\vec x;R)}\frac{\dd{\vec y}}{|\vec x\!-\!\vec y|^2}=4\pi R=\sqrt[3]{48\pi^2|\Omega|}.
\end{equation*}

\noindent\textbf{Step 2:} We now consider the case where $G$ is the fundamental solution with $n=2$. In such case, we have $G(\vec x\!-\!\vec y)=-\ln|\vec x\!-\!\vec y|/(2\pi)$. Notice that for any $r\in\bbR_{>0}$, we have
\begin{equation*}
(\ln r)^2=-|\ln r|\,\ln r+2\big(\max\{0,\ln r\}\big)^2,
\end{equation*}
where the first summand monotonically decreases and can be bounded in a similar manner to Step 1, while the second summand monotonically increases and can be bounded in terms of $\ln\sup\big\{|\vec x\!-\!\vec y|:\vec x,\vec y\in\Omega\big\}$. Therefore, we can prove that $\int_\Omega\big(\ln|\vec x\!-\!\vec y|\big)^2\dd{\vec y}$ is uniformly bounded in $\vec x$.

\noindent\textbf{Step 3:} We now consider the case where $G$ is equipped with periodic boundary conditions. In this case $G(\vec x\!-\!\vec y)$ is periodic, therefore $\int_\Omega G(\vec x\!-\!\vec y)^2\dd{\vec y}$ is independent of $\vec x$. According to \cite[Equations (2.1) and (2.2)]{choksi2010small}, $G$ is the sum of the fundamental solution and a regular part, therefore $\int_\Omega G(\vec x\!-\!\vec y)^2\dd{\vec y}$ is finite.
\end{proof}

\begin{lemma}
\label{weak version of boundedness of inverse Laplacian}
If $\eta_\veps \rightharpoonup \eta_0$ in $L^2(\Omega)$, define $\psi_\veps(\vec x)=\int_\Omega G(\vec x\!-\!\vec y)\,\eta_\veps(\vec y)\dd{\vec y}$, and analogously for $\psi_0$. Then
\begin{equation*}
\bigg|\int_\Omega(\psi_\veps\,\eta_\veps\!-\!\psi_0\,\eta_0)\bigg|\to0\quad\text{as}\;\veps\to0.
\end{equation*}
\end{lemma}

\begin{proof}
Notice that according to Lemma \ref{boundedness of inverse Laplacian}, we have $\psi_\veps\,,\,\psi_0\in L^\infty(\Omega)$. Moreover,
\begin{equation*}
\begin{aligned}
\bigg|\int_\Omega(\psi_\veps\,\eta_\veps\!-\!\psi_0\,\eta_0)\bigg|
&\leqslant
\bigg|\int_\Omega(\psi_\veps\!-\!\psi_0)\,\eta_\veps\bigg|
+
\bigg|\int_\Omega\psi_0\,(\eta_\veps\!-\!\eta_0)\bigg|\\
&\leqslant
\|\psi_\veps\!-\!\psi_0\|_{L^2(\Omega)}\,\|\eta_\veps\|_{L^2(\Omega)}+\bigg|\int_\Omega\psi_0\,(\eta_\veps\!-\!\eta_0)\bigg|\,,
\end{aligned}
\end{equation*}
where $\big|\int_\Omega\psi_0\,(\eta_\veps\!-\!\eta_0)\big|\to0$ and $\|\eta_\veps\|_{L^2(\Omega)}$ has an upper bound independent of $\veps$, because $\eta_\veps\rightharpoonup\eta_0$ in $L^2(\Omega)$. In addition, we have
\begin{equation*}
\begin{aligned}
\lim_{\veps\to0}\|\psi_\veps\!-\!\psi_0\|_{L^2(\Omega)}^2&=\lim_{\veps\to0}\int_\Omega\big(\psi_\veps(\vec x)\!-\!\psi_0(\vec x)\big)^2\dd{\vec x}
=\int_\Omega\lim_{\veps\to0}\big(\psi_\veps(\vec x)\!-\!\psi_0(\vec x)\big)^2\dd{\vec x}\\
&=\int_\Omega\lim_{\veps\to0}\Big(\int_\Omega G(\vec x\!-\!\vec y)\big(\eta_\veps(\vec y)\!-\!\eta_0(\vec y)\big)\dd{\vec y}\Big)^2\dd{\vec x}
=0\,,
\end{aligned}
\end{equation*}
where the last equality is again due to $\eta_\veps \rightharpoonup \eta_0$ in $L^2(\Omega)$, and the second equality is due to the Dominated Convergence Theorem. To see why the Dominated Convergence Theorem applies, notice that according to Lemma \ref{boundedness of inverse Laplacian}, $|\psi_0(\vec x)|$ is uniformly bounded in $\vec x$, and that $|\psi_\veps(\vec x)|$ is uniformly bounded in $\vec x$ and $\veps$, because $\|\eta_\veps\|_{L^2(\Omega)}$ is uniformly bounded in $\veps$.
\end{proof}

\subsubsection{Limsup inequality}
\label{Subsection Limsup inequality}

\begin{proposition}
\label{Limsup inequality proposition}
For any $u_0,v_0\in L^2(\Omega)$, there exist $\{u_\veps\},\{v_\veps\}\subset L^2(\Omega)$ such that both $u_\veps\rightarrow u_0$ and $v_\veps\rightarrow v_0$ in $L^2(\Omega)$, satisfying
\begin{equation*}
\mathcal E_0(u_0,v_0)\geq \limsup_{\veps\to0}\mathcal E_\veps(u_\veps,v_\veps).
\end{equation*}
\end{proposition}

\begin{proof}
We assume that the above left-hand side is finite (otherwise there is nothing to prove), therefore we have $u_0\in BV(\Omega;\{0,1\})$, $v_0\in L^2(\Omega;[0,1])$, and $u_0v_0=0$ a.e.

\noindent\textbf{Step 1:} We first consider the case where $\gamma=0$ and $n=1$. We further assume $\Omega=[-1,1]$ and $u_0=\bm1_{[0,1]}$. Therefore $v_0(x)=0$ for $x\in[0,1]$ a.e. We take
\begin{equation*}
u_\veps(x)=\big(1\!+\!\tanh(3x/\veps)\big)\big/2\;\;\text{for}\;\;x\in[-1,1],\quad\text{and}\;\;v_\veps=v_0\bm1_{[-1,-\delta]}\;\;\text{where}\;\;\delta=\veps^{2/3}.
\end{equation*}
By Dominated Convergence Theorem, we have both $u_\veps\to u_0$ and $v_\veps\to v_0$, pointwise \text{a.e.} and thus in $L^2(\Omega)$. We can compute
\begin{equation*}
\frac12\int_\Omega|\nabla u_\veps|^2=\frac12\int_{-1}^1{u'}_{\!\!\veps}(x)^{\,2}\dd{x}=\frac{1}{4 \veps}\bigg(3-\tanh^2\Big(\frac3\veps\Big)\bigg)\tanh\!\Big(\frac3\veps\Big).
\end{equation*}
In addition, we have
\begin{equation*}
W(u_\veps,v_\veps)=18(u_\veps\!-\!u_\veps^2)^2+27\min\{1\!-\!u_\veps\!-\!v_\veps,0\}^2/2\,,
\end{equation*}
where the first summand has the following integral
\begin{equation*}
\int_\Omega18(u_\veps\!-\!u_\veps^2)^2=\frac\veps4 \bigg(3-\tanh^2\Big(\frac3\veps\Big)\bigg)\tanh\!\Big(\frac3\veps\Big),
\end{equation*}
and the second summand can be bounded as follows
\begin{equation*}
\mbox{$\int_\Omega$}\min\{1\!-\!u_\veps\!-\!v_\veps,0\}^2=\mbox{$\int_{-1}^{-\delta}$}\min\{1\!-\!u_\veps\!-\!v_0,0\}^2\leq\mbox{$\int_{-1}^{-\delta}$}u_\veps^2\,,
\end{equation*}
where the inequality is due to $v_0\leq1$ a.e. Since $u_\veps$ increases monotonically, we have
\begin{equation*}
\int_{-1}^{-\delta}u_\veps^2 \leq (1-\delta)\,u_\veps^2(-\delta) \leq \big(1\!-\!\tanh(3\delta/\veps)\big)^2=\Big(1\!-\!\tanh\big(3\veps^{-1/3}\big)\Big)^2.
\end{equation*}
Noticing $1\!-\!\tanh z=2/(e^{2z}\!+\!1)<2e^{-2z}$, we obtain $\int_{-1}^{-\delta}u_\veps^2=o(\veps)$. To summarize, we have
\begin{equation*}
\mathcal E_\veps(u_\veps,v_\veps)=\frac{\veps}{2}\int_\Omega|\nabla u_\veps|^2+\int_\Omega\frac{W(u_\veps,v_\veps)}{\veps}\leq\frac{1}{2}\bigg(3-\tanh^2\Big(\frac3\veps\Big)\bigg)\tanh\!\Big(\frac3\veps\Big)+o(1),
\end{equation*}
where the right-hand side converges to $1$ as $\veps\to0$. Therefore $\mathcal E_0(u_0,v_0)=1\geq\limsup\limits_{\veps\to0}\mathcal E_\veps(u_\veps,v_\veps)$.

So far in our proof of the Gamma-convergence, we have not explicitly dealt with the mass constraints, which should be straightforward because the weak convergence in $L^2(\Omega)$ implies the convergence of the Lebesgue measure. However, we need to provide more details here. To make sure $\int_\Omega v_\veps=\int_\Omega v_0$, we rescale $v_\veps$ defined above, i.e., we redefine $v_\veps=v_0\bm1_{[-1,-\delta]}\int^{0}_{-1}v_0/\int^{-\delta}_{-1}v_0$, where
\begin{equation*}
\delta=\veps^{2/3},\quad\mbox{$\int^{0}_{-1}v_0$}=\zeta m,\quad\mbox{$\int^{-\delta}_{-1}v_0$}=\zeta m-\!\mbox{$\int^0_{-\delta}v_0$}\geqslant\zeta m\!-\!\delta.
\end{equation*}
In this way, we have $v_\veps\leqslant\zeta m/(\zeta m\!-\!\delta)=1+O\big(\veps^{2/3}\big)$ and thus $\veps^{-1}\!\int_\Omega\min\{1\!-\!v_\veps,0\}^2\leqslant O\big(\veps^{1/3}\big)$. Therefore, after the above modification, $v_\veps$ satisfies the mass constraint while $\mathcal E_\veps(u_\veps,v_\veps)$ increases only by order $\veps^{1/3}$.

\noindent\textbf{Step 2:} We now consider the case where $\gamma=0$ and $n>1$. We assume $u_0=\bm1_U$ for some open set $U\subset\bbR^n$ with $\partial U$ being a nonempty compact hypersurface of differentiability class $C^2$. We further assume $\mathcal H^{\!n-1}(\partial U\cap\partial\Omega)=0$, where $\mathcal H$ is the Hausdorff measure. We can rewrite $u_0=g_0(d_U)$ where $g_0=\bm1_{[0,\infty)}$, and $d_U$ is the signed distance function defined by
\begin{equation*}
d_U(\vec x)=
\left\{
\begin{aligned}
&\inf\big\{|\vec x-\vec y|:\vec y\in\partial U\big\},&&\text{if}\;\vec x\in U\,,\\
-&\inf\big\{|\vec x-\vec y|:\vec y\in\partial U\big\},&&\text{if}\;\vec x\not\in U\,.
\end{aligned}
\right.
\end{equation*}
We take $u_\veps=g_\veps(d_U)$ and $v_\veps=h_\veps(d_U)v_0$, where $g_\veps(z):=\big(1\!+\!\tanh(3z/\veps)\big)\big/2$, and $h_\veps:=\bm1_{(-\infty,-\delta]}$ with $\delta=\veps^{2/3}$. The function $d_U$ is Lipschitz continuous and satisfies $|\nabla d_U|=1$ \text{a.e.} in $\bbR^n$ \cite[Theorems 2.1 and 3.1]{doi:10.1137/1.9780898719826.ch7}. Let $l$ denote the size of the level set of $d_U$:
\begin{equation*}
l(r):=\mathcal H^{\!n-1}\big(\{\vec x\in\Omega:d_U(\vec x)=r\}\big),\quad r\in\bbR.
\end{equation*}
By Lemma \ref{coarea formula Lipschitz}, we have
\begin{equation*}
\begin{aligned}
\mathcal E_\veps(u_\veps,v_\veps)&=\frac{\veps}{2}\int_\Omega|\nabla u_\veps|^2+\int_\Omega\frac{W(u_\veps,v_\veps)}{\veps}\\
&=\int_\Omega\bigg(\frac{\veps}{2}|\nabla u_\veps|^2+\frac{18}{\veps}(u_\veps\!-\!u_\veps^2)^2+\frac{27}{2\veps}\min\{1\!-\!u_\veps\!-\!v_\veps,0\}^2\bigg)\\
&\leq\int_\Omega\bigg(\frac{\veps}{2}|\nabla u_\veps|^2+\frac{18}{\veps}(u_\veps\!-\!u_\veps^2)^2+\frac{27}{2\veps}\min\big\{1\!-\!u_\veps\!-\!h_\veps(d_U),0\big\}^2\bigg)\\
&=\int_\Omega\bigg(\frac{\veps}{2}g_\veps'(d_U)^2+\frac1\veps W\Big(g_\veps(d_U),h_\veps(d_U)\Big)\bigg)\,|\nabla d_U|\\
&=\int_{\bbR}\bigg(\frac{\veps}{2}g_\veps'(r)^2+\frac1\veps W\Big(g_\veps(r),h_\veps(r)\Big)\bigg)\,l(r)\dd{r}:=\!\int_{\bbR}\!a_\veps(r)l(r)\dd{r}.
\end{aligned}
\end{equation*}
According to Step 1, we know $\int_{-1}^1a_\veps(r)\dd{r}\leq1+o(1)$. For any fixed $\tau>0$, we have $a_\veps(r)\to0$ as $\veps\to0$ uniformly for $r\in\bbR\backslash(-\tau,\tau)$. Since $\Omega$ is bounded, by Lemma \ref{coarea formula Lipschitz}, $\int_{\bbR}l(r)\dd{r}=\int_\Omega |\nabla d_U|=|\Omega|<\infty$. Because $\partial U$ is of differentiability class $C^2$ with $\mathcal H^{\!n-1}(\partial U\cap\partial\Omega)=0$, we know that $l(r)$ is continuous at $r=0$ satisfying $\lim\limits_{r\to0}l(r)=\mathcal H^{\!n-1}(\partial U\cap\Omega)$ \cite[Lemma 5.8]{leoni2013gamma}. Therefore,
\begin{equation*}
\limsup\limits_{\veps\to0}\mathcal E_\veps(u_\veps,v_\veps)\leq\limsup\limits_{\veps\to0}\big(l(0)+o(1)\big)=\mathcal H^{\!n-1}(\partial U\cap\Omega)=\mathcal E_0(u_0,v_0).
\end{equation*}
We now show both $u_\veps\to u_0$ and $v_\veps\to v_0$ in $L^2(\Omega)$. To this end, we use $|\nabla d_U|=1$ and Lemma \ref{coarea formula Lipschitz} again to obtain
\begin{equation*}
\int_\Omega(u_\veps\!-\!u_0)^2=\int_\Omega\big(g_\veps(d_U)\!-\!g_0(d_U)\big)^2\,|\nabla d_U|=\int_\bbR\big(g_\veps(r)\!-\!g_0(r)\big)^2\,l(r)\dd{r}.
\end{equation*}
As mentioned above, $\int_{\bbR}l<\infty$. Noticing $|g_\veps\!-\!g_0|\leq|g_\veps|+|g_0|\leq2$ and $g_\veps\to g_0$ pointwise \text{a.e.}, by Dominated Convergence Theorem, we have $\int_\Omega(u_\veps\!-\!u_0)^2\to0$. Similarly, we have $\int_\Omega(v_\veps\!-\!v_0)^2\to0$.

In order to guarantee the mass constraint $\int_\Omega u_\veps=m$, we can make some technical modifications to the above $u_\veps$ in a similar way to the classical Cahn--Hilliard energy functional \cite[Equation (26)]{leoni2013gamma}. Similar to Step 1, we can also modify $v_\veps$ to guarantee the mass constraint $\int_\Omega v_\veps=\zeta m$.

\noindent\textbf{Step 3:} We now remove the regularity assumption imposed on $U$ in Step 2. By Lemma \ref{Approximation of a set of finite perimeter}, there exists a sequence of open sets $U_j$ with $\partial U_j$ being a nonempty compact hypersurface of differentiability class $C^2$ and satisfying $\mathcal H^{\!n-1}(\partial U_j\cap\partial\Omega)=0$, such that $\bm 1_{U_j}\to \bm1_U$ in $L^2(\Omega)$ and $\text{Per}_\Omega U_j\to\text{Per}_\Omega U$. According to Step 2, for each fixed $j$ we can find a sequence $\{u_{j,\,\veps}\}\subset W^{1,2}(\Omega)$ and $\{v_{j,\,\veps}\}\subset L^2(\Omega)$ such that both $u_{j,\,\veps}\to \bm 1_{U_j}$ and $v_{j,\,\veps}\to(1\!-\!\bm 1_{U_j})\,v_0$ in $L^2(\Omega)$ as $\veps\to0$, in addition to
\begin{equation*}
\limsup\limits_{\veps\to0}\mathcal E_\veps(u_{j,\,\veps},v_{j,\,\veps})\leq\mathcal H^{\!n-1}(\partial U_j\cap\Omega)=\text{Per}_\Omega U_j.
\end{equation*}
Therefore we have
\begin{equation*}
\limsup\limits_{j\to\infty}\,\limsup\limits_{\veps\to0}\,\mathcal E_\veps(u_{j,\,\veps},v_{j,\,\veps})\leq\limsup\limits_{j\to\infty}\text{Per}_\Omega U_j=\text{Per}_\Omega U.
\end{equation*}
Noticing
\begin{equation*}
\lim\limits_{j\to\infty}\,\lim\limits_{\veps\to0}\,\|u_{j,\,\veps}\!-\!u_0\|_{L^2(\Omega)}=\lim\limits_{j\to\infty}\,\|\bm 1_{U_j}\!-\!\bm1_U\|_{L^2(\Omega)}=0,
\end{equation*}
and
\begin{equation*}
\lim\limits_{j\to\infty}\,\lim\limits_{\veps\to0}\,\|v_{j,\,\veps}\!-\!v_0\|_{L^2(\Omega)}=\lim\limits_{j\to\infty}\,\big\|(1\!-\!\bm 1_{U_j})\,v_0\!-\!(1\!-\!\bm 1_U)v_0\big\|_{L^2(\Omega)}=0,
\end{equation*}
we know
\begin{equation*}
\limsup\limits_{j\to\infty}\,\limsup\limits_{\veps\to0}\,\big(\|u_{j,\,\veps}\!-\!u_0\|_{L^2(\Omega)}+\|v_{j,\,\veps}\!-\!v_0\|_{L^2(\Omega)}\big)=0.
\end{equation*}
By Lemma \ref{Diagonalization Argument lemma}, we have a diagonal sequence $\{u_{j_\veps,\,\veps}\}$ and $\{v_{j_\veps,\,\veps}\}$ such that both $u_{j_\veps,\,\veps}\to u_0$ and $v_{j_\veps,\,\veps}\to v_0$ in $L^2(\Omega)$, and
\begin{equation*}
\limsup_{\veps\to0}\mathcal E_\veps(u_{j_\veps,\,\veps},v_{j_\veps,\,\veps})\leq\text{Per}_\Omega U=\mathcal E_0(u_0,v_0).
\end{equation*}

\noindent\textbf{Step 4:} We now consider the case where $\gamma>0$. According to Step 2 in the proof of Proposition \ref{Liminf inequality proposition}, we have
\begin{equation*}
\lim_{\veps\to0}\mathcal N(u_\veps,v_\veps)=\mathcal N(u_0,v_0).
\end{equation*}
\end{proof}

\begin{lemma}
\label{coarea formula Lipschitz}
Coarea Formula for Lipschitz Functions \cite[Theorem 1.14]{leoni2013gamma}. On an open set $\Omega\subset\bbR^n$, let $\psi:\Omega\to\bbR$ be Lipschitz continuous and let $\eta:\bbR\to\bbR$ be Borel measurable. If $\eta\circ \psi$ is integrable, then
\begin{equation*}
\int_\bbR\eta(r)\,\mathcal H^{n-1}\big(\{\vec x\in\Omega:\psi(\vec x)=r\}\big)\dd{r}=\int_\Omega\eta(\psi)\,|\nabla\psi|.
\end{equation*}
\end{lemma}

\begin{lemma}
\label{Approximation of a set of finite perimeter}
Approximation of a set of finite perimeter \cite[Lemma 1.15]{leoni2013gamma}. If a bounded open set $\Omega\subset\bbR^n$ has Lipschitz boundary, and $U\subset\bbR^n$ has finite perimeter, then there exists a sequence of open sets $U_j$ with $\partial U_j$ being a nonempty compact hypersurface of differentiability class $C^2$ and satisfying $\mathcal H^{\!n-1}(\partial U_j\cap\partial\Omega)=0$, such that $\bm 1_{U_j}\to \bm1_U$ in $L^2(\Omega)$, $\text{Per}_\Omega U_j\to\text{Per}_\Omega U$, and $|U_j|=|U|$.
\end{lemma}

\begin{lemma}
\label{Diagonalization Argument lemma}
Diagonalization Argument. If double-indexed sequences $\{a_{j,\,k}\}\subset\bbR$ and $\{b_{j,\,k}\}\subset\bbR$ satisfy
\begin{equation*}
\limsup_{j\to\infty}\,\limsup_{k\to\infty}\,a_{j,\,k}\leq A,\qquad \limsup_{j\to\infty}\,\limsup_{k\to\infty}\,b_{j,\,k}\leq B,
\end{equation*}
for some constants $A,B\in\bbR$. Then there exists a diagonal sequence $j_k\to\infty$ as $k\to\infty$ such that
\begin{equation*}
\limsup_{k\to\infty}\,a_{j_k,\,k}\leq A,\qquad \limsup_{k\to\infty}\,b_{j_k,\,k}\leq B.
\end{equation*}
\end{lemma}

\begin{proof}
Define
\begin{equation*}
j_1^*=\min\big\{j\in\mathbb N:\limsup\limits_{k\to\infty}a_{j,\,k}
\leq 
A\!+\!1\;\text{and}\;\limsup\limits_{k\to\infty}b_{j,\,k}
\leq B\!+\!1\big\}.
\end{equation*}
We claim that $j_1^*$ exists. Otherwise for each $j\in\mathbb N$, we have $j\in \mathcal A \cup \mathcal B$, where $\mathcal A:=\big\{j\in\mathbb N:\limsup\limits_{k\to\infty}a_{j,\,k} >
A\!+\!1\big\}$ and $\mathcal B:=\big\{j\in\mathbb N:\limsup\limits_{k\to\infty}b_{j,\,k} > B\!+\!1\big\}$. Therefore $\big|\mathcal A\big|=\infty$ or $\big|\mathcal B\big|=\infty$. If $\big|\mathcal A\big|=\infty$, then $\limsup\limits_{j\to\infty}\,\limsup\limits_{k\to\infty}\,a_{j,\,k}\geq A\!+\!1$, which is a contradiction.

From the definition of $j_1^*$, we know that $k_1^*$ (defined as follows) exists.
\begin{equation*}
k_1^*:=\min\big\{k\in\mathbb N:\sup_{l\geq k}a_{j_1^*,\,l}\leq A\!+\!2\;\text{and}\;\sup_{l\geq k}b_{j_1^*,\,l}\leq B\!+\!2\big\}.
\end{equation*}

Recursively, for $p=[2,\infty)\cap\mathbb N$, define
\begin{equation*}
j_p^*=\min\big\{j>j_{p-1}^*:\limsup\limits_{k\to\infty}a_{j,\,k}\leq A\!+\!1/p\;\text{and}\;\limsup\limits_{k\to\infty}b_{j,\,k}\leq B\!+\!1/p\big\},
\end{equation*}
and
\begin{equation*}
k_p^*=\min\big\{k>k_{p-1}^*:\sup_{l\geq k}a_{j_p^*,\,l}\leq A\!+\!2/p\;\text{and}\;\sup_{l\geq k}b_{j_p^*,\,l}\leq B\!+\!2/p\big\}.
\end{equation*}
For each $p\geq1$ and each $k\in[k_p^*, k_{p+1}^*)\cap\mathbb N$, define $j_k=j_p^*$, then
\begin{equation*}
a_{j_k,\,k}\leq\sup\limits_{l\geq k_p^*}a_{j_p^*,\,l}\leq A\!+\!2/p,\quad b_{j_k,\,k}\leq\sup\limits_{l\geq k_p^*}b_{j_p^*,\,l}\leq B\!+\!2/p.
\end{equation*}
\end{proof}

\subsubsection{Convergence of global minimizers}

\begin{theorem}
\label{theorem Gamma convergence}
$\mathcal E_0$ is the Gamma-limit of $\mathcal E_\veps$ as $\veps\to0$.
\end{theorem}
\begin{proof}
See Propositions \ref{Liminf inequality proposition} and \ref{Limsup inequality proposition} for the liminf and limsup inequalities.
\end{proof}

\begin{proposition}
$\mathcal E_\veps$ has a global minimizer.
\end{proposition}
\begin{proof}
We use the direct method in the calculus of variations. We can show that for any fixed $\veps$, any minimizing sequence $\big\{(u_j,v_j)\big\}$ of $\mathcal E_\veps$ are uniformly bounded in $W^{1,2}(\Omega) \times L^2(\Omega)$, and thus possesses a subsequence (not relabeled) satisfying $u_j \rightharpoonup u_\infty$ in $W^{1,2}(\Omega)$ and $v_j \rightharpoonup v_\infty$ in $L^2(\Omega)$ for some $u_\infty\in W^{1,2}(\Omega)$ and $v_\infty\in L^2(\Omega)$. By the Rellich\textendash Kondrachov theorem, $u_j \to u_\infty$ both in $L^2(\Omega)$ and pointwise (up to a further subsequence). Similar to Step 5 of the Proof of Proposition \ref{compactness proposition}, we can show $W(u_\infty,v_\infty)\leqslant\liminf\limits_{j\to\infty} W(u_j,v_j)$. Similar to Step 2 of Proposition \ref{Liminf inequality proposition}, we can show $\mathcal N(u_\infty,v_\infty)=\lim\limits_{j\to\infty}\mathcal N(u_j,v_j)$. Finally, since $u_j \rightharpoonup u_\infty$ in $W^{1,2}(\Omega)$ and $u_j \to u_\infty$ in $L^2(\Omega)$, by the Cauchy\textendash Schwarz inequality we know $\|\nabla u_\infty\|_{L^2(\Omega)}\leqslant\liminf\limits_{j\to\infty}\|\nabla u_j\|_{L^2(\Omega)}$. Therefore, we have proven $\mathcal E_\veps(u_\infty,v_\infty)\leqslant\liminf\limits_{j\to\infty}\mathcal E_\veps(u_j,v_j)$, and therefore $(u_\infty,v_\infty)$ is a global minimizer of $\mathcal E_\veps$.
\end{proof}

\begin{theorem}
Let $(u_\veps, v_\veps) \in W^{1,2}(\Omega) \times L^2(\Omega)$ be a global minimizer of $\mathcal E_{\veps}$ with $\veps\to0$, then any subsequence of $\big\{(u_\veps, v_\veps)\big\}$ has a further subsequence (not relabeled) such that $u_\veps \to u_0$ in $L^2(\Omega)$ and $v_\veps \rightharpoonup v_0$ in $L^2(\Omega)$, where $(u_0,v_0) \in BV(\Omega; \{0,1\} )\times L^2(\Omega;[0,1])$ is a global minimizer of $\mathcal E_0$, satisfying $\mathcal E_0(u_0,v_0)=\liminf\limits_{\veps\to0}\mathcal E_{\veps}(u_\veps, v_\veps)$.
\end{theorem}

\begin{proof}
According to \cite[Bottom of Page 18]{van2008partial}, $\{\mathcal E_\veps\}$ is equi-coercive thanks to the compactness result (Proposition \ref{compactness proposition}). Since $\mathcal E_0$ is the Gamma-limit of $\mathcal E_\veps$ as $\veps\to0$ (Theorem \ref{theorem Gamma convergence}), according to \cite[Theorem 1.4.5]{van2008partial}, $\mathcal E_0$ has a global minimizer $(u_0,v_0)$ such that $\mathcal E_0(u_0,v_0)=\liminf\limits_{\veps\to0}\mathcal E_{\veps}(u_\veps, v_\veps)$. Thanks to the compactness result (Proposition \ref{compactness proposition}) and the limsup inequality (Proposition \ref{Limsup inequality proposition}), we know that the global minimizer sequence $\big\{(u_\veps, v_\veps)\big\}$ is precompact, and every cluster point of this sequence is a global minimizer of $\mathcal E_0$ \cite[Theorem 1.4.5]{van2008partial}.
\end{proof}

\subsection{Equivalence to the original problem}
\label{subsec: boundaryconditions}
We are now left with some justifications to make in order to establish the equivalence between the Gamma-limit $\mathcal E_0$ and the original problem $E$. Recall that $E$ given by \eqref{sharp energy nonrelaxed} is a functional of two subsets of $\bbR^n$, with $G$ being the fundamental solution. In the definition of $\mathcal E_0(u,v)$, the function $v$ is allowed to take any value between 0 and 1. However, as we show in Section \ref{subsec: Justifications}, the minimizer $(u_*,v_*)$ of $\mathcal E_0$ satisfies $v_*\in\{0,1\}$ \text{a.e.}, which means that $v_*$ can be regarded as an indicator function. Since $\mathcal E_0$ is the Gamma-limit of $\mathcal E_\veps$, the boundary conditions on $G$ should be inherited from $\mathcal E_\veps$ to $\mathcal E_0$, which can be either no boundary conditions (i.e., $G$ is the fundamental solution) or periodic boundary conditions. In Section \ref{subsubsection Justification for boundary conditions}, we show the equivalence between those two types of boundary conditions.

\subsubsection{Justification for the relaxation}
\label{subsec: Justifications}

We can regard $\mathcal E_0$ as a relaxed version of $E$, since the second argument of $\mathcal E_0$ is allowed to take intermediate values between 0 and 1. As we now explain, such relaxation does not affect the energy minimizers.

\begin{proposition}
\label{prop: v must be indicator function}
For any fixed $u\in BV(\Omega;\{0,1\})$, if $v_*$ is a global minimizer of $\mathcal N(u,v)$ (given by \eqref{linearized nonlocal term}) among all the $v\in L^2(\Omega;[0,1])$ such that $vu=0$ \text{a.e.} and $\int_\Omega v=\zeta\int_\Omega u$, then $v_*$ must be an indicator function, i.e., $v_*\in\{0,1\}$ a.e.
\end{proposition}

\begin{proof}
Our proof is very similar to \cite[Equation (3.11)]{bonacini2016optimal} (see also \cite[Lemma 4]{van2008copolymer}). Given any fixed $\tau>0$, we want to prove $|S|=0$, where $S=\big\{\vec x\in\Omega:v_*(\vec x)\in(\tau,1\!-\!\tau)\big\}$. If $|S|>0$, define the perturbation $\psi=(\varphi_*\!-\!c)\bm1_S$, where $\varphi_*=G*(u\!-\!v_*/\zeta)$ and $c=\int_S\varphi_*/|S|$. According to \cite[Theorem 9.9]{gilbarg2001elliptic}, we have $\varphi_*\in W^{2\,,\,p}(\Omega)$ for all $1<p<\infty$, and $\Delta\varphi_*=u\!-\!v_*/\zeta$ \text{a.e.} on $\Omega$. In particular, we have $\Delta\varphi_*<-\tau/\zeta$ \text{a.e.} on $S$, because $v_*>\tau>0$ \text{a.e.} on $S$ and $u=0$ \text{a.e.} on $S$ (due to $v_*u=0$ \text{a.e.} on $\Omega$).

We compute the following first-order variation:
\begin{equation*}
\frac\dd{\dd\delta}\mathcal N(u,v_*\!-\!\psi\delta) \Big|_{\delta=0} = \int_\Omega\psi\varphi_*/\zeta = \int_\Omega\varphi_*(\varphi_*\!-\!c)\bm1_S/\zeta = \int_S\varphi_*(\varphi_*\!-\!c)/\zeta = \int_S(\varphi_*\!-\!c)^2/\zeta.
\end{equation*}
Due to the minimality condition on $v_*$, the above left-hand side is zero, and thus $\varphi_*$ is constant \text{a.e.} on $S$. According to \cite[Lemma 7.7]{gilbarg2001elliptic}, which asserts that the weak derivatives of a Sobolev function are zero \text{a.e.} on its level set, we have $\nabla\varphi_*=\vec0$ \text{a.e.} on $S$, and consequently $\Delta\varphi_*=0$ \text{a.e.} on $S$, which is a contradiction.
\end{proof}

\subsubsection{Justification for boundary conditions}
\label{subsubsection Justification for boundary conditions}

\paragraph{Compactly supported minimizers}
We now establish that periodic boundary conditions and no boundary conditions on $G$ are equivalent for the purpose of energy minimization in the sharp interface limit, as long as $\Omega$ is sufficiently large so that the minimizers are compactly supported in $\Omega$. The key is the screening property satisfied by the energy minimizers \cite[Corollary 3.3]{bonacini2016ground}, i.e., different connected components of a minimizer do not interact with each other. So far we have been ambiguously using $G$ for the Green's functions of $-\Delta$ under either periodic boundary conditions or no boundary conditions. From now on, we let $\widetilde G$ denote the former, and let $G$ denote the latter. We use analogous notations for other symbols as well: e.g., $E$ is given by \eqref{sharp energy nonrelaxed} with the nonlocal kernel $G$ being the fundamental solution, and $\widetilde E$ is the periodic counterpart with the nonlocal kernel being the Green's function under periodic boundary conditions.
\begin{proposition}
\label{proposition: equivalence between two boundary conditions}
Let $(U,V)$ be a global minimizer of $E$, and let $(\widetilde U,\widetilde V)$ be a global minimizer of $\widetilde E$, under the mass constraints $|U|=|\widetilde U|=m$ and $|V|=|\widetilde V|=\zeta m$. If $\Omega$ is large enough for $U$, $V$, $\widetilde U$ and $\widetilde V$ to be compactly supported in $\Omega$ with positive distance to $\partial\Omega$, then $E(U,V)=\widetilde E(U,V)=E(\widetilde U,\widetilde V)=\widetilde E(\widetilde U,\widetilde V)$.
\end{proposition}

\begin{proof}

\noindent\textbf{Step 1:} We first prove $E(U,V)=\widetilde E(U,V)\geqslant \widetilde E(\widetilde U,\widetilde V)$. It suffices to prove $N(U,V)=\widetilde N(U,V)$, where $\widetilde N$ is the periodic counterpart of $N$ (i.e., $\widetilde N$ is the nonlocal term with the nonlocal kernel being the Green's function under periodic boundary conditions). According to Proposition \ref{existing qualitative results}-\CircleAroundChar{3}, we have $\phi_{U,V}=0$ in $\Omega\backslash(U\cup V)$, where $\phi_{U,V}$ is the electrostatic potential $\phi$ associated with $U$ and $V$, given by \eqref{eqn: electrostatic potential}. Since $\phi_{U,V}$ satisfies periodic boundary conditions and the Poisson's equation $-\Delta\phi_{U,V}=\bm1_U-\bm1_V/\zeta$, we have $\tilde\phi_{U,V}=\text{constant}+\phi_{U,V}$, where $\tilde\phi_{U,V}$ is the solution to $-\Delta\tilde\phi_{U,V}=\bm1_U-\bm1_V/\zeta$ under periodic boundary conditions. According to \eqref{alternative expression of N(U,V)}, we have $2N(U,V)=\int_{U}\phi_{U,V}-\int_{V}\phi_{U,V}/\zeta=\int_{U}\tilde\phi_{U,V}-\int_{V}\tilde\phi_{U,V}/\zeta=2\widetilde N(U,V)$.

\noindent\textbf{Step 2:} We now prove $\widetilde E(\widetilde U,\widetilde V)=E(\widetilde U,\widetilde V)\geqslant E(U,V)$. It suffices to prove $\widetilde N(\widetilde U,\widetilde V)=N(\widetilde U,\widetilde V)$. In Step 3 we will prove $\tilde \phi_{\widetilde U,\widetilde V}=\text{constant}$ in $\Omega\backslash(\widetilde U\cup \widetilde V)$, so we can extend $\tilde \phi_{\widetilde U,\widetilde V}$ to be this constant in $\bbR^n\backslash\Omega$. Since $\phi_{\widetilde U,\widetilde V}$ vanishes at infinity \cite[Lemma 3.1]{bonacini2016optimal}, $\tilde \phi_{\widetilde U,\widetilde V}-\phi_{\widetilde U,\widetilde V}$ is harmonic and bounded in $\bbR^n$, therefore $\tilde \phi_{\widetilde U,\widetilde V}-\phi_{\widetilde U,\widetilde V}=\text{constant}$. Similar to Step 1, we can prove $\widetilde N(\widetilde U,\widetilde V)=N(\widetilde U,\widetilde V)$.

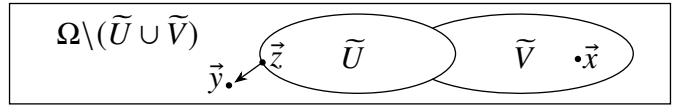
\begin{figure}[H]
\centering
\begin{tikzpicture}
\draw[line width = 0.5pt] (-60pt,11pt) rectangle (190pt,49pt);
\draw[line width = 0.5pt] (71.7pt,30pt) ellipse (37pt and 15pt);
\draw[line width = 0.5pt] (100pt,20.35pt) arc [start angle=-140, end angle=140, x radius=43pt, y radius =15pt];
\fill (35.7pt,26.7pt) circle[radius=1.2pt];
\draw [-{Stealth[length=5pt]},line width = 0.5pt](35.7pt,26.7pt) -- (25pt,19pt);
\node at (41pt,29.5pt) {$\vec z$};
\node at (-15pt,38pt) {$\Omega\backslash(\widetilde U\cup \widetilde V)$};
\node at (70pt,30pt) {$\widetilde U$};
\node at (135pt,30pt) {$\widetilde V$};
\fill (155pt,28pt) circle[radius=1.2pt];
\node at (160pt,29pt) {$\vec x$};
\fill (23pt,18pt) circle[radius=1.2pt];
\node at (18pt,20.8pt) {$\vec y$};
\end{tikzpicture}
\caption{Step 3 in the Proof of Proposition \ref{proposition: equivalence between two boundary conditions}.}
\label{sketch of Step 3 of equivalence proof}
\end{figure}

\noindent\textbf{Step 3:} Now let us prove $\tilde \phi_{\widetilde U,\widetilde V}=\text{constant}$ in $\Omega\backslash(\widetilde U\cup \widetilde V)$. Our proof is similar to \cite[Section 4]{bonacini2016optimal}. According to \cite[Top of Page 8]{bonacini2016ground}, $E$ is a volume perturbation of the perimeter, and $U$ is a quasi-minimizer of the perimeter, so $\partial U$ is of differentiability class $C^\infty$ (Proposition \ref{existing qualitative results}-\CircleAroundChar{2}). Similarly, $\widetilde E$ is also a volume perturbation of the perimeter, so we can assume that $\widetilde U$ and $\widetilde V$ are open, and that $\partial\widetilde U$ is of differentiability class $C^\infty$. If $\widetilde V$ is not open, we can instead consider the points at which the Lebesgue density of $\widetilde V$ is positive, similar to \cite[Page 1141]{bonacini2016optimal}. See Figure \ref{sketch of Step 3 of equivalence proof} for visualization.

Using the optimality of $\widetilde V$ for $\widetilde N(\widetilde U,\,\cdot\,)$, we can prove the following by contradiction:
\begin{equation*}
\tilde \phi_{\widetilde U,\widetilde V}(\vec x)\geqslant\tilde \phi_{\widetilde U,\widetilde V}(\vec y),\quad\text{for any}\;\vec x\in \widetilde V\;\text{and any}\;\vec y\in \Omega\backslash\overline{\widetilde U\cup \widetilde V}.
\end{equation*}
Otherwise, we can decrease the energy by moving some negative charge from a neighborhood of $\vec x$ to a neighborhood of $\vec y$. More precisely, since $\tilde\phi_{\widetilde U,\widetilde V}$ is continuous and $2\widetilde N(\widetilde U,\widetilde V)=\int_{\widetilde U}\tilde\phi_{\widetilde U,\widetilde V}-\int_{\widetilde V}\tilde\phi_{\widetilde U,\widetilde V}/\zeta$, we can define a competitor $\widetilde V_r=B(\vec y;r)\cup\widetilde V\backslash B(\vec x;r)$ where $r$ is small enough, and $\widetilde N(\widetilde U,\widetilde V_r)<\widetilde N(\widetilde U,\widetilde V)$.

On $\Omega\backslash(\widetilde U\cup \widetilde V)$, since $\tilde \phi_{\widetilde U,\widetilde V}$ is harmonic and satisfies the periodic boundary conditions, its minimum $a$ and maximum $b$ are attained on the boundary $\partial(\widetilde U\cup \widetilde V)$. If $\tilde \phi_{\widetilde U,\widetilde V}$ is non-constant in $\Omega\backslash(\widetilde U\cup \widetilde V)$, then $a<b\leqslant\tilde \phi_{\widetilde U,\widetilde V}(\vec x)$ for any $\vec x\in \widetilde V$, and the minimum $a$ can only be attained on $\partial \widetilde U\backslash\partial \widetilde V$, because $\tilde \phi_{\widetilde U,\widetilde V}$ is continuous in $\Omega$. Let $\vec z\in\partial \widetilde U\backslash\partial \widetilde V$ such that $\tilde \phi_{\widetilde U,\widetilde V}(\vec z)=a$. Without loss of generality, let us assume that $\widetilde U$ is connected. Because $\tilde \phi_{\widetilde U,\widetilde V}\geqslant a$ on both $\partial \widetilde U\backslash\partial \widetilde V$ and $\partial \widetilde U\cap\partial \widetilde V$, and $-\Delta\tilde \phi_{\widetilde U,\widetilde V}=1$ on $\widetilde U$, by the strong minimum principle for superharmonic functions, we have $\tilde \phi_{\widetilde U,\widetilde V}>a$ on $\widetilde U$. Since $\widetilde U$ satisfies the interior ball condition ($\partial\widetilde U$ is of differentiability class $C^\infty$), the Hopf boundary point lemma \cite[Lemma 3.4]{gilbarg2001elliptic} states that the outer normal derivative of $\tilde \phi_{\widetilde U,\widetilde V}$ at $\vec z$ is negative, which is a contradiction to the minimality of $a$, because there is $\vec y$ outside of $\widetilde U\cup \widetilde V$ such that $\vec y$ is very close to $\vec z$ and $\tilde \phi_{\widetilde U,\widetilde V}(\vec y)<a$.
\end{proof}

\paragraph{Non-compactly supported minimizers}
Even if $\Omega$ is not large enough and the global minimizer of $\widetilde E$ is not compactly supported in $\Omega$, we believe the asymptotics in Conjecture \ref{Periodic asymptotics conjecture} to hold true. Our rationale is that according to \cite[Theorem 8]{van2008copolymer}, any lower-dimensional structure with zero dipole moment can be extended to a higher dimension using a radially symmetric cutoff function, with asymptotically the same energy-to-mass ratio.
\begin{conjecture}
\label{Periodic asymptotics conjecture}
Assume that $(\widetilde U,\widetilde V)$ satisfies the mass constraints $|\widetilde U|=m$ and $|\widetilde V|=\zeta m$, but we no longer require $(\widetilde U,\widetilde V)$ to be a global minimizer or compactly supported in $\Omega$. Additionally, we assume that $(\widetilde U,\widetilde V)$ has zero dipole moment, i.e., $\int_\Omega\vec x\big(\bm1_{\widetilde U}(\vec x)-\bm1_{\widetilde V}(\vec x)/\zeta\big)\dd{\vec x}=\vec0$, which can be achieved by selecting a suitable translational representative according to Lemma \ref{zero dipole moment by translation}. Let $(\widetilde U_j,\widetilde V_j)$ denote the juxtaposition of $j^n$ copies of $(\widetilde U,\widetilde V)$ (there are $j$ cycles along the direction of each standard basis vector in $\bbR^n$, similar to the cubic crystal structure). Then $\widetilde E(\widetilde U,\widetilde V)=\lim\limits_{j\to\infty}E(\widetilde U_j,\widetilde V_j)/j^n$.
\end{conjecture}

\begin{remark}$ $
\label{crystal intuition remark}
\begin{enumerate}[label=\protect\CircleAroundChar{\arabic*}]
\item
For $n=3$, Conjecture \ref{Periodic asymptotics conjecture} has a well-known discrete variant in which $\widetilde U$ and $\widetilde V$ are replaced by Dirac delta functions representing point charges with zero surface area. When calculating the electrostatic potential energy-to-mass ratio of a crystal made up of many (but finite) unit cells, solid physicists and material scientists usually take a shortcut: they consider only one unit cell and use Ewald summation which implicitly assumes periodic boundary conditions on Poisson's equation \cite[Left of Page 7888]{hummer1998molecular}. Such Ewald summation gives an asymptotically correct value for a large but finite crystal surrounded by vacuum as long as the dipole moment in the unit cell is zero \cite[Equation (1.8)]{smith1981electrostatic}, in which case the net Coulomb interaction decays sufficiently fast so that the summation is absolutely convergent. In this sense, periodic boundary conditions and no boundary conditions are equivalent for zero dipole moment.

\item
For a nonzero dipole moment, the correct electrostatic energy-to-mass ratio is the sum of an intrinsic part and an extrinsic part. The intrinsic part is given by the above Ewald summation and is dependent on the unit cell but independent of the global shape of the crystal \cite{herce2007electrostatic}. The extrinsic part depends on the shape of the crystal, and can be interpreted as the outcome of "effective" charges distributed on the surface of the crystal (those surface charges reproduce the total dipole moment produced by the unit cells in the bulk) \cite{kantorovich1999coulomb, smith1981electrostatic}. In this sense, periodic boundary conditions and no boundary conditions are not equivalent for a nonzero dipole moment.

\item
The above extrinsic part should produce a voltage across the crystal formed by polar unit cells. However, this voltage seems to be absent in everyday life except in pyroelectric materials (whose polarization depends on the temperature). After a sufficiently long time at a constant temperature, even pyroelectric materials will lose such a voltage, because external charges will build up on the surface of the crystal through leakage currents (conducted by the crystal itself or the ambient atmosphere), thus canceling out that voltage in a similar way to grounding or earthing. Therefore, it seems more realistic to assume that the electrostatic potential is zero on the surface of the crystal, which leads to the homogeneous Dirichlet boundary conditions. According to Lemma \ref{gounding minimizes the electrostatic potential energy}, the electrostatic potential energy is actually minimized by the homogeneous Dirichlet boundary conditions. Many physicists and chemists refer to homogeneous Dirichlet boundary conditions as "tin foil" or "conducting" boundary conditions (imagine a crystal wrapped in a tin foil, or submerged in a conducting medium). They also believe that periodic boundary conditions and homogeneous Dirichlet boundary conditions are equivalent (see \cite[Bottom-right of Page 5024]{roberts1994unit}, \cite[Top-left of Page 6134]{figueirido1995finite} and \cite[Right of Page 124106-1]{herce2007electrostatic}). Therefore, the physical reality may be better approximated by periodic boundary conditions than by no boundary conditions in certain scenarios \cite[Page 6167]{kantorovich1999coulomb}.
\end{enumerate}
\end{remark}

The following lemma shows that within a bounded and connected domain, among all the electrostatic potentials that solve the same Poisson's equation, the electrostatic potential energy is minimized by the one satisfying homogeneous Dirichlet boundary conditions.

\begin{lemma}
\label{gounding minimizes the electrostatic potential energy}
Assuming that $\Omega$ is connected. For any fixed $\phi\in W^{1\,,\,2}(\Omega)$, consider the functional $h\mapsto\|\nabla(\phi-h)\|_{L^2(\Omega)}$ defined on the set consisting of all harmonic functions on $\Omega$. Its global minimizer $h_*$ is unique (up to addition of a constant) and satisfies Dirichlet boundary conditions $h_*=\phi$ on $\partial\Omega$.
\end{lemma}

\begin{proof}
Let $h_*$ be a global minimizer, we can consider a competitor $h_*+\delta h$, where $\delta\in\bbR$ and $h$ is any harmonic function on $\Omega$. Since $h_*$ is stationary, we have
\begin{equation*}
0=\frac{\dd}{\dd \delta}\|\nabla(\phi\!-\!h_*\!-\!\delta h)\|_{L^2(\Omega)}^2\Big|_{\delta=0}
=2\int_\Omega\nabla h\cdot\nabla(h_*\!-\!\phi)=2\int_{\partial\Omega}(h_*\!-\!\phi)\frac{\partial h}{\partial \vec n}-2\int_\Omega(h_*\!-\!\phi)\Delta h.
\end{equation*}
On the one hand, if $h_*-\phi$ is constant on $\partial\Omega$, then the above right-hand side vanishes, because $\int_{\partial\Omega}\partial h/\partial \vec n=\int_{\Omega}\Delta h=0$. On the other hand, if $h_*-\phi$ is non-constant on $\partial\Omega$, then we can find a solution to Laplace's equation $\Delta h=0$ under Neumann boundary conditions $\partial h/\partial \vec n=\text{constant}+\phi-h_*$ on $\partial\Omega$, so that $h_*$ is not stationary against the small perturbations along the direction $h$. In summary, $h_*-\phi$ must be constant on $\partial\Omega$.
\end{proof}

The following lemma shows that any $\widetilde U$ and $\widetilde V$ contained in $\Omega$ can be periodically (or circularly) shifted so that the dipole moment vanishes, thus satisfying the assumption in Conjecture \ref{Periodic asymptotics conjecture}. Without loss of generality, we only present the 2-D case ($n=2$).

\begin{lemma}
\label{zero dipole moment by translation}
For any fixed $g\in L^1\big([0,1]^2\big)$, we extend it periodically in both directions. If $\int_{[0,\,1]^2}g(\vec x)\dd{\vec x}=0$, then there exists $\vec t=(t_1,t_2)\in[0,1]^2$ such that $\int_{[0,\,1]^2}\vec x\,g(\vec x\!+\!\vec t\,)\dd{\vec x}=\vec0$.
\end{lemma}
\begin{proof}
Define $\vec h(\,\vec t\,)=\big(h_1(t_1,t_2),h_2(t_1,t_2)\big)=\int_{[0,\,1]^2}\vec x\,g(\vec x\!+\!\vec t\,)\dd{\vec x}$, then $\vec h$ is continuous in $\bbR^2$. Since $g$ is periodic in both directions, $\vec h$ is also periodic. Moreover, we have $h_1(t_1,t_2)=\int_0^1\int_0^1 x_1\,g(x_1\!+\!t_1,x_2\!+\!t_2)\dd{x_2}\dd{x_1}=\int_0^1\int_0^1 x_1\,g(x_1\!+\!t_1,x_2)\dd{x_2}\dd{x_1}=h_1(t_1,0)$, and
\begin{equation*}
\begin{aligned}
\int_0^1 h_1(t_1,0)\dd{t_1}&=\int_0^1\int_0^1\int_0^1 x_1\,g(x_1\!+\!t_1,x_2)\dd{x_2}\dd{x_1}\dd{t_1}\\
&=\int_0^1\int_0^1\int_0^1 x_1\,g(x_1\!+\!t_1,x_2)\dd{t_1}\dd{x_2}\dd{x_1}\\
&=\int_0^1\int_0^1\int_0^1 x_1\,g(t_1,x_2)\dd{t_1}\dd{x_2}\dd{x_1}\\
&=\frac12\int_0^1\int_0^1 g(t_1,x_2)\dd{t_1}\dd{x_2}=0.
\end{aligned}
\end{equation*}
Therefore, there exists $t^*_1\in[0,1]$ such that for any $t_2\in\bbR$, we have $h_1(t^*_1,t_2)=h_1(t^*_1,0)=0$. Similarly, there exists $t^*_2\in[0,1]$ such that for any $t_1\in\bbR$, we have $
h_2(t_1,t^*_2)=0$. In summary we have $\vec h(t^*_1,t^*_2)=\vec 0$.
\end{proof}

\section{Phase-field simulations}
\label{section: Phase-field simulations}
In this section, we present some numerical simulations based on the phase-field reformulation proposed in Section \ref{section: Phase-field reformulation}.

The numerical methods used in the simulations are very similar to those adopted in a previous work for the liquid drop model \cite[Section III]{10.1063/5.0148456}. The main difference is that two phase-field functions are used here, whereas only one was used in the previous work. Similar to \cite[Section III.A]{10.1063/5.0148456}, we use the following nonlinear function for the definition of $\mathcal E$ in Section \ref{subsec Diffuse interface energy}:
\begin{equation*}
f(z):=3z^2\!-\!2z^3.
\end{equation*}
For the time-marching scheme, we still use the convex splitting scheme, but our choice of the convex splitting is different from \cite[Equation (8)]{10.1063/5.0148456}. For the potential well $W$ defined by \eqref{formula of potential W}, our choice here is $W = W_1 + W_2$, where $W_1(u,v) = 87 u^2/2+27 u v +27 v^2$ and $W_2=W-W_1$. It is easy to check the convexity of $W_1$ and the concavity of $W_2$ on $[-0.1,1.1]\times[-0.1,1.1]$, which are sufficient to stabilize our numerical scheme. Similar to \cite[Equation (5)]{10.1063/5.0148456}, in order to find the local minimizers, we use the following $L^2$ gradient flow which is called the penalized Allen\textendash Cahn\textendash Ohta\textendash Kawasaki (pACOK) dynamics:
\begin{equation*}
\frac{\partial u}{\partial t}=-L_1\frac{\delta \mathcal E}{\delta u},\qquad \frac{\partial v}{\partial t}=-L_2\frac{\delta \mathcal E}{\delta v},
\end{equation*}
where $L_1,L_2>0$ are called mobility coefficients. We use the semi-implicit time-marching scheme similar to \cite[Equation (8)]{10.1063/5.0148456}.

In most simulations, the phase-fields almost vanish near $\partial\Omega$. According to Section \ref{subsubsection Justification for boundary conditions}, no boundary conditions (i.e., $G$ is the fundamental solution) and periodic boundary conditions on $G$ are equivalent for the purpose of energy minimization. Therefore we only need to consider periodic boundary conditions, which allow us to use the Fourier spectral method similar to \cite[Section III.D]{10.1063/5.0148456}.

Although the perimeter of $V$ is not penalized in \eqref{sharp energy nonrelaxed}, we may add a penalty term $\int_\Omega|\nabla v|^2$ to $\mathcal E$ in order for $v$ to converge faster. In practice, the penalty coefficient of $\int_\Omega|\nabla v|^2$ can be chosen to be around $1/1250000$ of that of $\int_\Omega|\nabla u|^2$. Note that doing so is optional, and it will not significantly alter the numerical results but only provide some regularization for $v$.

In Sections \ref{subsec: Simulations in 2-D} and \ref{subsec: Simulations in 3-D}, we present some possible local minimizers as well as some snapshots taken from pACOK dynamics, in order to provide numerical evidence for our conjectures about the Gamma-expansion of the sharp interface energy in 2-D and 3-D for $\zeta<\zeta_1$ (see the first statement in Conjecture \ref{2-D conjecture of Gamma-convergence} and the first statement in Conjecture \ref{3-D conjecture of Gamma-convergence}, respectively).

\subsection{Simulations in 2-D}
\label{subsec: Simulations in 2-D}

In our 2-D simulations, we choose $\zeta=1$, $\gamma=1500$, $\veps=5\times10^{-2}$, $K_1 = 3\times10^{4}$, $K_2 = 24\times10^{3}/5$, $L_1 = 1$, $L_2 = 5$, $\Delta t=25\times10^{-5}/2$ (which is the time step), and the simulation domain is chosen to be $[0,13/5]\times[0,13/5]$ which is then discretized into $256\times256$ uniform grid points. We visualize the results using RGB images of $256\times256$ pixels, where the hydrophobic region $U$ is drawn in purplish pink, and the hydrophilic region $V$ is drawn in greenish yellow. More specifically, we assign an RGB triplet to every possible value of $(u,v)$. We assign $(211,95,183)$, $(220,220,98)$ and $(255,255,255)$ to $(1,0)$, $(0,1)$ and $(0,0)$, respectively. For other values of $(u,v)$, the RGB triplet is linearly interpolated from the above three, and then truncated to $[0,255]^3$. The color at the position $\vec x$ is then determined by the RGB triplet assigned to $\big(u(\vec x),v(\vec x)\big)$.

We present 13 local minimizers in Figure \ref{figure local minimizers 2D}. Next to each local minimizer is $(m,\mathcal E/m)$. Each local minimizer is obtained by choosing a suitable initial value, and then numerically simulating the pACOK dynamics until the iteration converges. The reason we treat them as local minimizers is that they appear to be stable against random perturbations in our simulations. One could further investigate their stability by evaluating the second-order variation of the energy or using rigorously validated numerics (similar to \cite{martine2022microscopic}).

\begin{figure}[H]
\raggedleft
\parbox[b][][t]{13ex}{\centering$(0.1\,,15.632)$}\parbox[b][][t]{13ex}{$ $}\parbox[b][][t]{13ex}{\centering$(0.4\,,15.049)$}\parbox[b][][t]{13ex}{\centering$(0.6\,,14.929)$}\parbox[b][][t]{13ex}{\centering$(0.8\,,14.89)$}\parbox[b][][t]{13ex}{\centering$(1\,,14.873)$}\parbox[b][][t]{13ex}{\centering$(1.2\,,14.864)$}\\
\includegraphics[width=91ex, bb=0 0 1816 521]{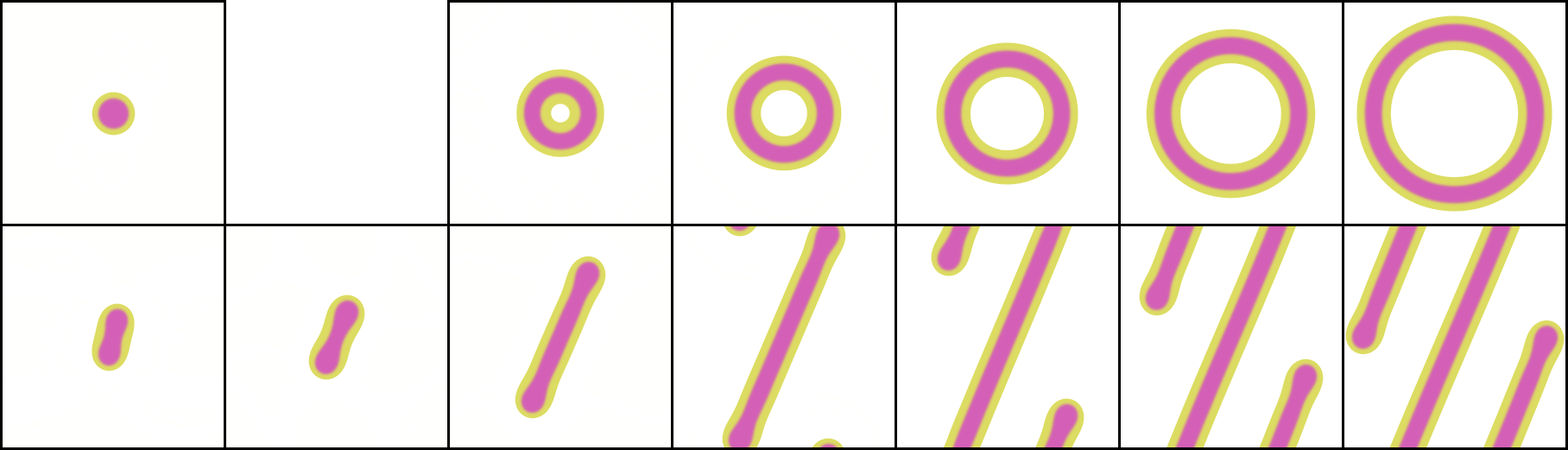}\\
\makebox[13ex][c]{$(0.15\,,15.573)$}\makebox[13ex][c]{$(0.2\,,15.37)$}\makebox[13ex][c]{$(0.4\,,15.109)$}\makebox[13ex][c]{$(0.6\,,15.021)$}\makebox[13ex][c]{$(0.8\,,14.977)$}\makebox[13ex][c]{$(1\,,14.951)$}\makebox[13ex][c]{$(1.2\,,14.933)$}

\includegraphics[width=55ex]{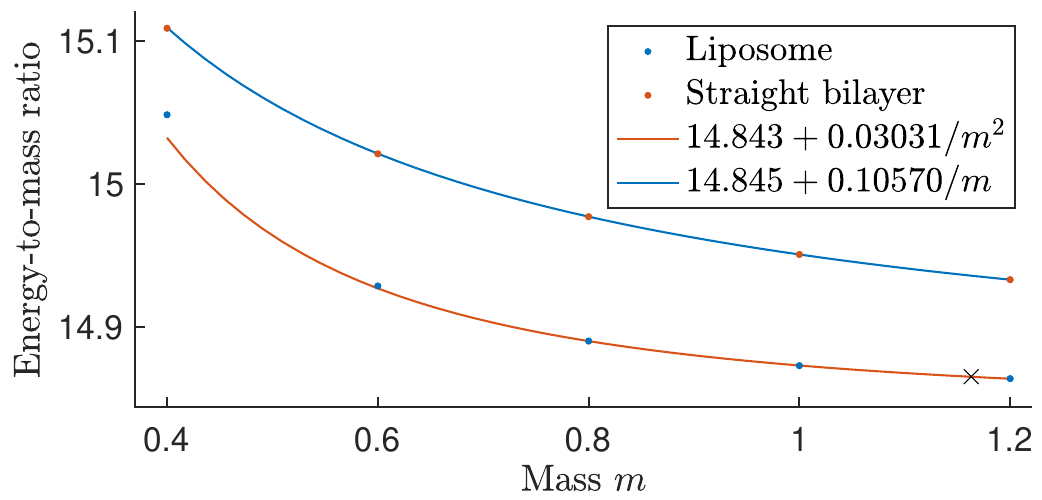}\hspace{8ex}
\caption{Top: local minimizers obtained in numerical simulations. Each square box represents the simulation domain containing a local minimizer. Next to each local minimizer are the mass and the energy-to-mass ratio. Bottom: curve fitting to the energy-to-mass ratios of the above local minimizers.}
\label{figure local minimizers 2D}
\end{figure}

As we can see in Figure \ref{figure local minimizers 2D}, the local minimizer is a micelle for $m=0.1$, which is consistent with the 2-D generalization of Proposition \ref{existing qualitative results}-\CircleAroundChar{4}. As $m$ increases to $0.15$ and $0.2$, the micelle is no longer stable and deforms into a shape that is similar to the eye-mask shaped local minimizers in \cite[Figure 8]{10.1063/5.0148456}. As $m$ further increases to $0.4$ and beyond, as shown in the second row of Figure \ref{figure local minimizers 2D}, the local minimizer becomes more and more elongated, resembling a straight bilayer of approximately uniform thickness, except that near the two ends the $U$ layer is slightly thicker while the $V$ layer is slightly thinner. The two open ends cause the straight bilayer to have slightly higher energy than the local minimizer shown in the first row of Figure \ref{figure local minimizers 2D}, which resembles a liposome. The liposome and straight bilayer seem to prefer roughly the same thickness, which is consistent with Remark \ref{remark on nonrescaled liposome asymptotics}-\CircleAroundChar{2}. Their energy-to-mass ratios also seem to converge to roughly the same constant as $m\to\infty$, and the convergence rate for the liposome seems to be second-order, which is consistent with Corollary \ref{simple asymptotics of the optimal liposome candidate}. The convergence rate for the straight bilayer seems to be first-order, indicating that the two open ends carry asymptotically constant energy penalties. Our numerical results therefore suggest a Gamma-expansion very similar to \cite[Equation (1.6)]{peletier2009partial} for $n=2$ and $\zeta=1$.

For the liposome local minimizers in Figure \ref{figure local minimizers 2D}, we note that the inner $V$ layer is slightly thicker than the outer $V$ layer, and that such a difference becomes less noticeable as $m$ increases, which are consistent with Remark \ref{remark on asymptotics of liposome candidates}-\CircleAroundChar{1}. Our numerical evidence indicates that for $n=2$ and $\zeta=1$, the optimal liposome candidate (whose asymptotics is given in Corollaries \ref{simple asymptotics of the optimal liposome candidate} and \ref{rescaled liposome asymptotics}) is indeed a local minimizer, and might even be a global minimizer. In fact, if the first statement in Conjecture \ref{2-D conjecture of Gamma-convergence} is true, i.e., for $\zeta<\zeta_1$ the second-order term in the Gamma-expansion of our energy is the elastica functional, then the global minimizer should be approximately (or exactly) circular according to Proposition \ref{elastica functional minimized by a circle}.

In order to provide some numerical evidence for the first statement in Conjecture \ref{2-D conjecture of Gamma-convergence} and verify the resistance of the bilayer to bending, we carry out two simulations of the pACOK dynamics, as shown in Figure \ref{Dynamics 2D figure}. In the first simulation, we choose $m=1.163$, and we choose the initial value to be a bilayer of approximately uniform thickness, which resembles a random non-convex closed curve. We observe that over time, the bilayer becomes convex and resembles an ellipse, and eventually becomes circular after a sufficiently long time. The terminal value is shown in the top-right of Figure \ref{Dynamics 2D figure}, and its energy-to-mass ratio is indicated by the black cross $\times$ in the bottom-right of Figure \ref{figure local minimizers 2D}. In the second simulation, we choose $m=0.6$, and we choose the initial value by making a small hole in the liposome local minimizer shown in the top-middle of Figure \ref{figure local minimizers 2D}. We observe that over time, the bilayer straightens and eventually converges to a straight bilayer, which is a rigid transformation of the straight bilayer shown in the middle of the second row in Figure \ref{figure local minimizers 2D}. The convergence is relatively slow in both simulations, which is not surprising if the elastica functional is the second-order term in the Gamma-expansion.

\begin{figure}[H]
\centering

\parbox[b][][t]{13ex}{\centering14.8973}\parbox[b][][t]{13ex}{\centering14.8788}\parbox[b][][t]{13ex}{\centering14.8709}\parbox[b][][t]{13ex}{\centering14.8686}\parbox[b][][t]{13ex}{\centering14.867}\parbox[b][][t]{13ex}{\centering14.8657}\parbox[b][][t]{13ex}{\centering14.8655}

\includegraphics[width=91ex, bb=0 0 1816 521]{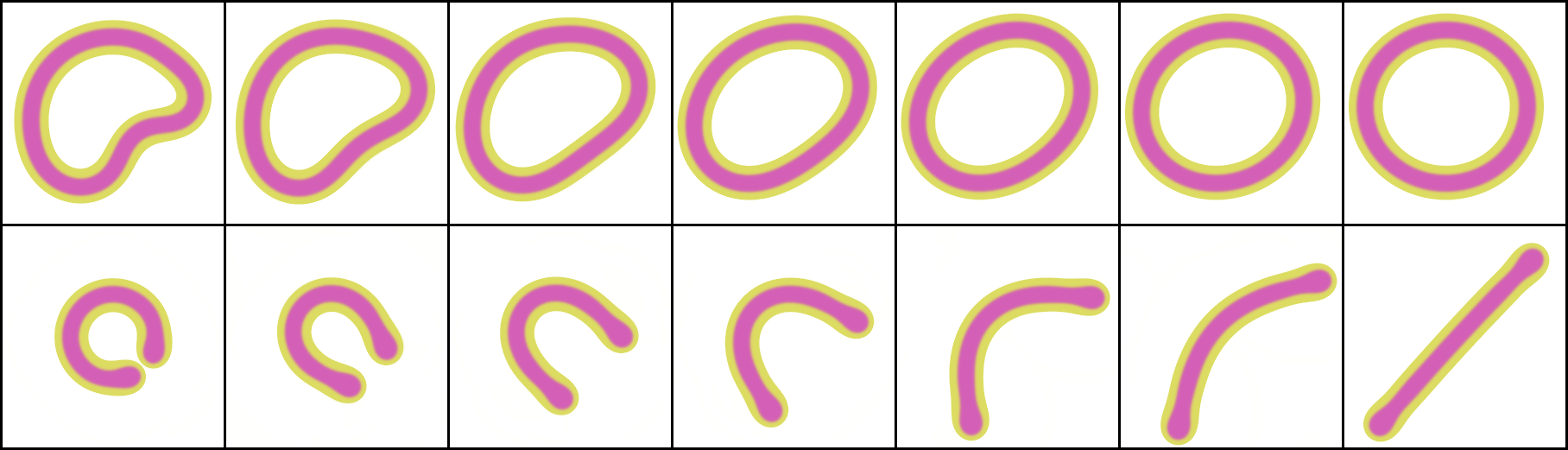}

\parbox[b][][t]{13ex}{\centering15.0947}\parbox[b][][t]{13ex}{\centering15.0711}\parbox[b][][t]{13ex}{\centering15.059}\parbox[b][][t]{13ex}{\centering15.047}\parbox[b][][t]{13ex}{\centering15.0328}\parbox[b][][t]{13ex}{\centering15.0257}\parbox[b][][t]{13ex}{\centering15.0213}

\caption{Numerical simulations of pACOK dynamics. Each square box represents a snapshot with $t$ increasing from left to right. Next to each snapshot is the energy-to-mass ratio. Top: the first simulation with $m=1.163$ and the initial value resembling a random curve. Bottom: the second simulation with $m=0.6$ and the initial value resembling a perforated liposome.}
\label{Dynamics 2D figure}
\end{figure}

\subsection{Simulations in 3-D}
\label{subsec: Simulations in 3-D}
In our 3-D simulations, we choose $\gamma=500$, $L_1 = 1$, $L_2 = 4$, and the simulation domain $[0,X]\times[0,Y]\times[0,Z]$ is discretized into $\texttt{N}\times \texttt{M}\times \texttt{P}$ uniform grid points. We visualize the results using the following two MATLAB commands:
\begin{center}
\centering
\texttt{isosurface(u+v,1/2); isosurface(u,1/2);}
\end{center}
where the former is set to be greenish yellow and transparent, representing the boundary of $\overline{U\cup V}$, while the latter is set to be purplish pink and opaque, representing the boundary of $U$. In order to visualize the inner structures, we also plot the cross-section with the cutting plane parallel to the front view and passing through the center of the simulation box.

We present 8 stationary points in Figure \ref{simulation liposome local minimizer}, each of which is the terminal value of the pACOK dynamics starting from a suitable initial value and evolving over a long period of time until the shape barely changes. We present two simulations of the pACOK dynamics in Figure \ref{pACOK of holes}. In Figures \ref{simulation liposome local minimizer} and \ref{pACOK of holes}, we choose $\zeta=1$, $\veps=6\times10^{-2}$, $K_1 = 5/2\times10^{4}$, $K_2 = 4\times10^{3}$, and $\Delta t=2.1\times10^{-5}$. In the top rows of Figures \ref{simulation liposome local minimizer} and \ref{pACOK of holes}, we choose $X = Y = Z = 3.66$ and $\texttt{M}=\texttt{N}=\texttt{P}=256$. In the last row in Figure \ref{pACOK of holes} and in the left three columns of the last row in Figure \ref{simulation liposome local minimizer}, we choose $X = Y = 4Z = 3.66$ and $\texttt{M}=\texttt{N}=4\texttt{P}=256$. In the rightmost column of the last row in Figure \ref{simulation liposome local minimizer}, we choose $X/2 = Y/2 = 4Z = 3.66$ and $\texttt{M}/2=\texttt{N}/2=4\texttt{P}=256$.

\begin{figure}[H]
\centering

\parbox[b][][t]{65ex}{\raggedright
\parbox[b][][t]{13.0375ex}{\centering$(1\,,10.694)$}\parbox[b][][t]{13.0375ex}{\centering$(1.6\,,10.554)$}\parbox[b][][t]{13.0375ex}{\centering$(2.4\,,10.477)$}\parbox[b][][t]{13.0375ex}{\centering$(7\,,10.378)$}

\includegraphics[width=52.15ex]{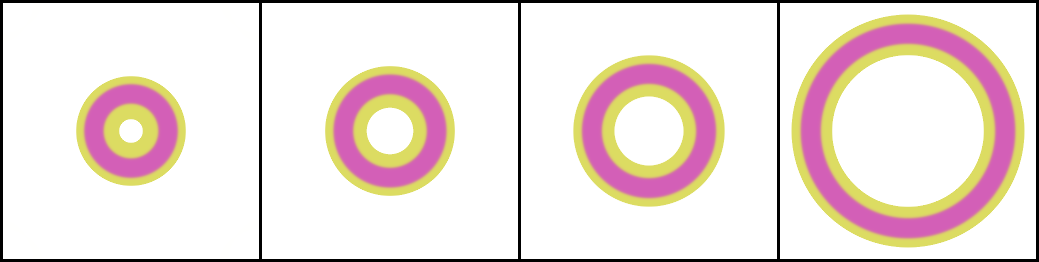}}

\parbox[b][][t]{65ex}{\raggedright\includegraphics[width=52.15ex]{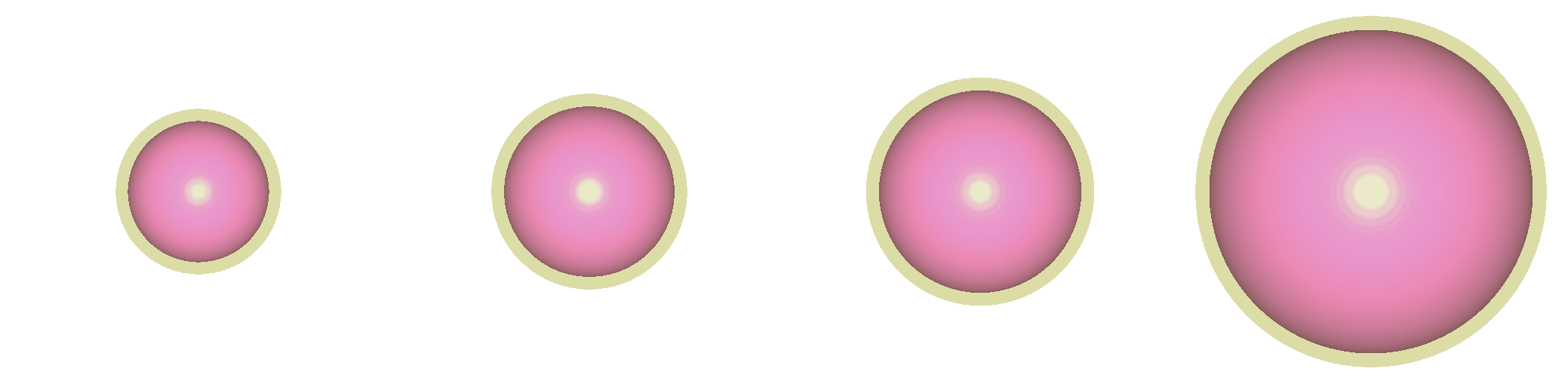}}

\includegraphics[width=65ex]{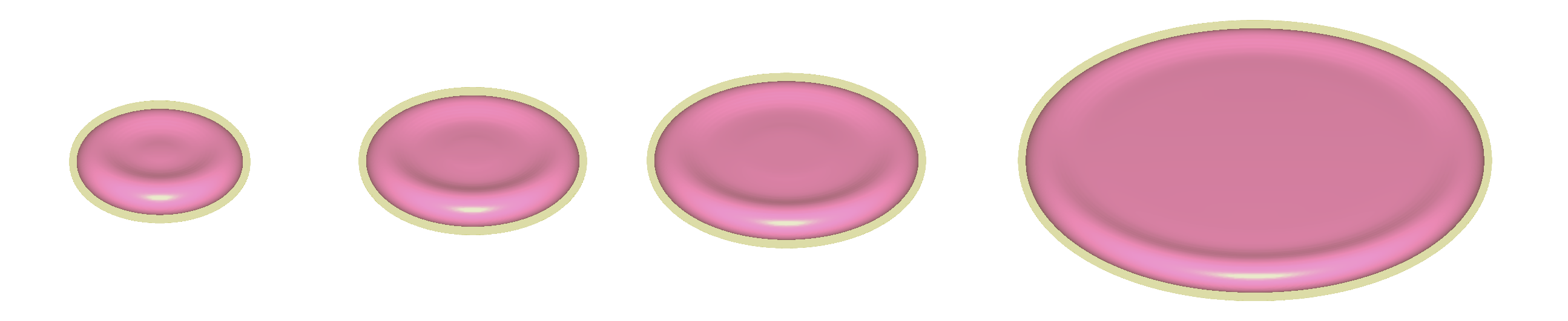}

\includegraphics[width=65ex]{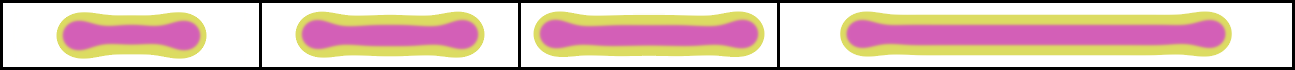}

\parbox[b][][t]{13.0375ex}{\centering$(1\,,10.841)$}\parbox[b][][t]{13.0375ex}{\centering$(1.6\,,10.733)$}\parbox[b][][t]{13.0375ex}{\centering$(2.4\,,10.659)$}\parbox[b][][t]{26.075ex}{\centering$(7\,,10.521)$}

\includegraphics[width=55ex]{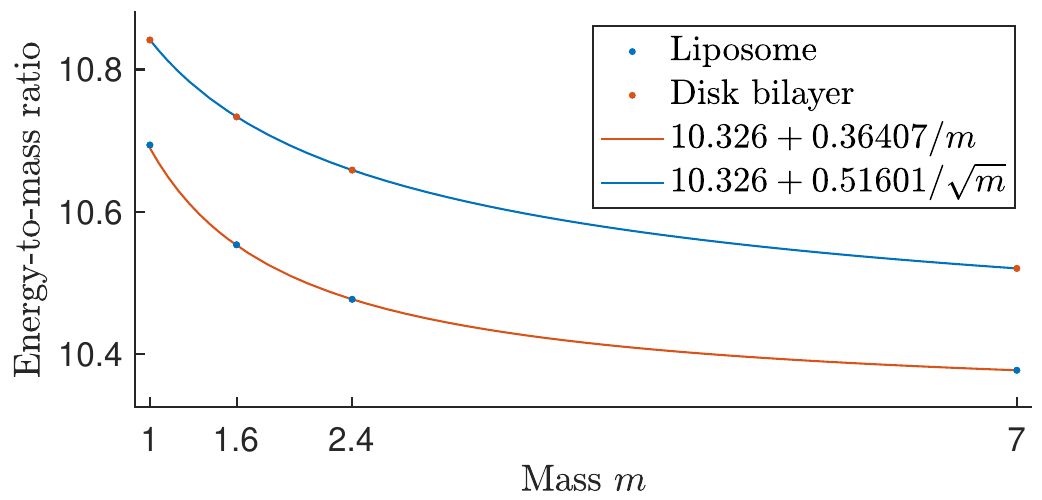}\hspace{10ex}
\caption{Top: stationary points obtained in numerical simulations. Each rectangle represents the cross-section of the simulation box. Next to each stationary point are the mass and the energy-to-mass ratio. Bottom: curve fitting to the energy-to-mass ratios of the above stationary points.}
\label{simulation liposome local minimizer}
\end{figure}

The stationary points in Figure \ref{simulation liposome local minimizer} resemble liposomes and disk bilayers, and they prefer roughly the same thickness, which is consistent with Remark \ref{remark on nonrescaled liposome asymptotics}-\CircleAroundChar{2}. Near the rim of a disk bilayer, the $U$ layer is slightly thicker while the $V$ layer is slightly thinner, which is a typical manifestation of frustration. Therefore, as $m\to\infty$, the disk bilayer should have a radius of order $\sqrt{m}$ and a thickness of order 1. Its rim should have a perimeter of order $\sqrt{m}$ and thus carry an energy penalty of order $\sqrt{m}$. Consequently, its energy-to-mass ratio should converge to a constant with order $1/2$, which is confirmed by the bottom of Figure \ref{simulation liposome local minimizer}. We can also see that the energy-to-mass ratio of the liposome converges to roughly the same constant with order $1$, which is consistent with Corollary \ref{simple asymptotics of the optimal liposome candidate}. Our numerical calculations show that the liposome has lower energy than the disk bilayer, so that the latter cannot be a global minimizer. The liposome seems to be a local minimizer in our simulation. However, we are not confident that the disk bilayer is also a local minimizer, although its shape remains almost unchanged after $3.5\times 10^7$ iterations. This is because the disk bilayer resembles an open surface, and we expect that the perimeter of its rim is penalized on the first order in the Gamma-expansion, and that the bending energy is only a second-order effect (see Section \ref{section Rescaled energy functional}). Therefore, although the disk bilayer has zero bending energy, its priority should be to close in on itself and form a closed surface for sufficiently large $m$. In fact, we can see that for $m=7$ the disk bilayer in Figure \ref{simulation liposome local minimizer} has higher energy than the curved bilayer shown in the top-left of Figure \ref{pACOK of holes}, because the latter has smaller rim perimeter despite larger bending energy.


As mentioned above, for a bilayer resembling an open surface, we expect a first-order energy penalty associated with its rim. In order to gain more insights, we carry out two simulations in Figure \ref{pACOK of holes}. In the first and second simulations, the initial value is chosen by making a hole in the liposome ($m=7$) and the disk bilayer ($m=2.4$) that are obtained in Figure \ref{simulation liposome local minimizer}, respectively. We observe that the hole diminishes and vanishes over time (cf. Bottom of Figure \ref{Dynamics 2D figure} where the hole enlarges). Therefore, our numerical results demonstrate the self-healing property of lipid bilayers in 3-D, which is consistent with experimental observations \cite{creasy2008self} and is essential to biological membranes.
\begin{figure}[H]
\centering

\parbox[b][][t]{13.0375ex}{\centering10.457}\parbox[b][][t]{13.0375ex}{\centering10.431}\parbox[b][][t]{13.0375ex}{\centering10.419}\parbox[b][][t]{13.0375ex}{\centering10.392}\parbox[b][][t]{13.0375ex}{\centering10.386}\parbox[b][][t]{13.0375ex}{\centering10.379}\parbox[b][][t]{13.0375ex}{\centering10.378}

\includegraphics[width=91.15ex]{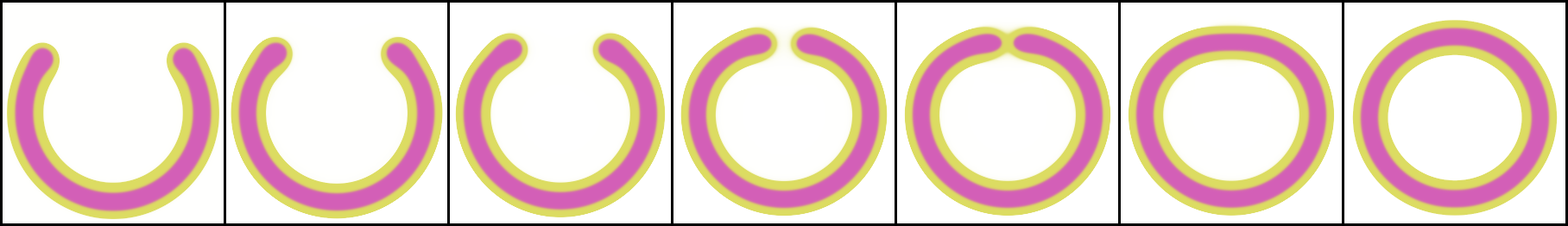}

\includegraphics[width=91.15ex]{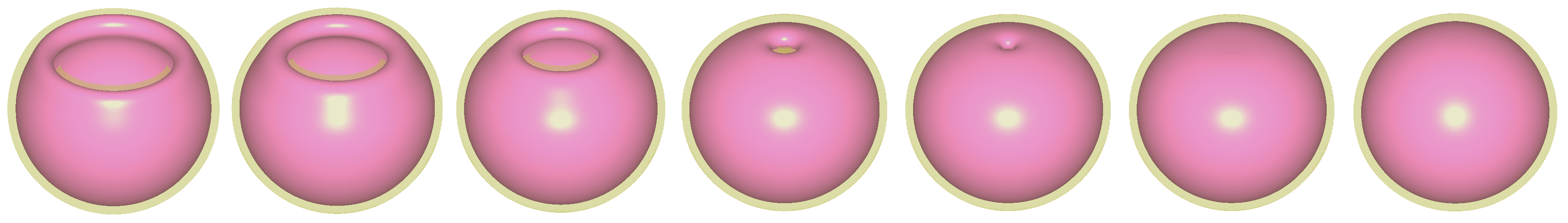}

\includegraphics[width=65.15ex]{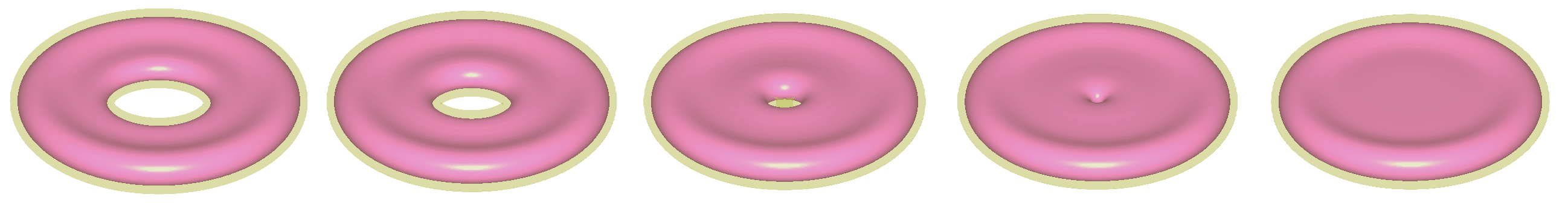}

\includegraphics[width=65.15ex]{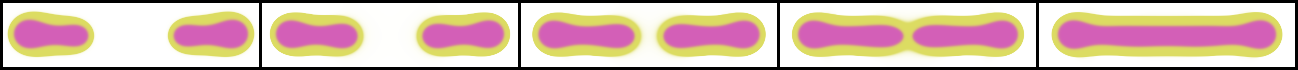}

\parbox[b][][t]{13.0375ex}{\centering10.828}\parbox[b][][t]{13.0375ex}{\centering10.767}\parbox[b][][t]{13.0375ex}{\centering10.7}\parbox[b][][t]{13.0375ex}{\centering10.678}\parbox[b][][t]{13.0375ex}{\centering10.659}

\caption{Numerical simulations of pACOK dynamics with $t$ increasing from left to right. Next to each snapshot is the energy-to-mass ratio. Top: the first simulation with $m=7$ and the initial value resembling a perforated liposome. Bottom: the second simulation with $m=2.4$ and the initial value resembling a perforated disk bilayer.}
\label{pACOK of holes}
\end{figure}

In order to provide some numerical evidence for the first statement in Conjecture \ref{3-D conjecture of Gamma-convergence} (i.e., the Willmore energy appears in the second-order Gamma-expansion), we present two numerical simulations in Figure \ref{figure torus simulation}. Recall that the Clifford torus and its image under a conformal transformation are non-isolated local minimizers of the Willmore energy (see Figure \ref{Clifford torus}). In the first and second simulations, we start from an initial value resembling a torus and a deformed torus, respectively, let it evolve according to the pACOK dynamics over a long period of time until the shape barely changes, and plot the terminal value in the left and right of Figure \ref{figure torus simulation}, respectively. In both simulations we choose $m=11.27$, $\zeta=1$, $\veps=0.1$, $K_1 = 1.5\times10^{4}$, $K_2 = 2.4\times10^{3}$, $\Delta t=9\times10^{-5}/8$, $X = Y = 2Z = 5.124$, and $\texttt{M}=\texttt{N}=2\texttt{P}=256$. In Figure \ref{figure torus simulation}, the terminal values shown in the left and right resemble the Clifford torus and its image under a conformal map, respectively. The former seems to be a local minimizer, while the latter has slightly higher energy than the former and seems to be evolving very slowly in the direction of becoming the former. We think the reason why the latter is not a local minimizer is that diffuse interfaces are used in our phase-field simulations, or that the bilayer has nonzero thickness, so that higher-order terms in the Gamma-expansion destroy the non-isolated local minimality.

\begin{figure}[H]
\centering
\parbox[b][][t]{18.2525ex}{\centering10.2523}\parbox[b][][t]{18.2525ex}{\centering10.2524}

\includegraphics[width=36.61ex]{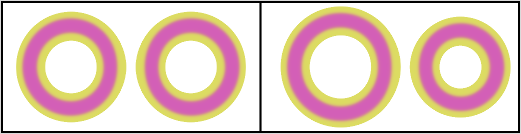}

\includegraphics[width=36.61ex]{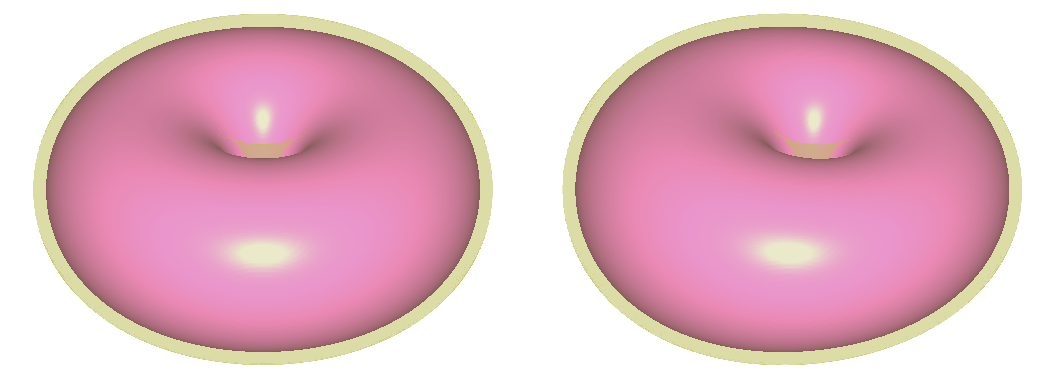}
\caption{Terminal values of the pACOK dynamics after a sufficiently long time. Left: the first simulation with the initial value resembling a torus. Right: the second simulation with the initial value resembling a deformed torus.}
\label{figure torus simulation}
\end{figure}

In order to provide numerical evidence for Remark \ref{triply periodic minimum surface remark}-\CircleAroundChar{2} (i.e., triply periodic minimal surfaces may be preferred over planar bilayer for small $\zeta$), we present four local minimizers shown in Figure \ref{figure gyroid simulation}. They are obtained as the terminal values of the pACOK dynamics starting from suitable initial values after a sufficiently long time. We choose $\veps=3.5\times10^{-2}$, $K_1 = 30\times10^{4}/7$, $K_2 = 48\times10^{3}/7$, $\Delta t=21\times10^{-5}/40$, and $\texttt{M}=\texttt{N}=\texttt{P}=512$. We choose $\zeta=0.6$ for the left two local minimizers, and choose $\zeta=1$ for the right two local minimizers. We choose $X = Y = Z = 3.51$ for the left three local minimizers, and choose $X = Y = Z = 3.6855$ for the rightmost local minimizer. From left to right, we choose $m=7.6886$, $11.5742$, $7.1262$ and $11.8021$, respectively, so that the respective energy-to-mass ratios are locally minimized with respect to $m$. From Figure \ref{figure gyroid simulation} we can see that for $\zeta=0.6$, the gyroid-like local minimizer has lower energy-to-mass ratio than the planar bilayer, and vice versa for $\zeta=1$. This combined with Conjecture \ref{Periodic asymptotics conjecture} is consistent with Remark \ref{triply periodic minimum surface remark}-\CircleAroundChar{2}.
\begin{figure}[H]
\centering

\parbox[b][][t]{18.4ex}{\centering9.6212}\parbox[b][][t]{18.4ex}{\centering9.5983}\parbox[b][][t]{18.4ex}{\centering10.3749}\parbox[b][][t]{18.4ex}{\centering10.3964}

\parbox[b][][t]{18.4ex}{\centering\includegraphics[width=12.61148ex]{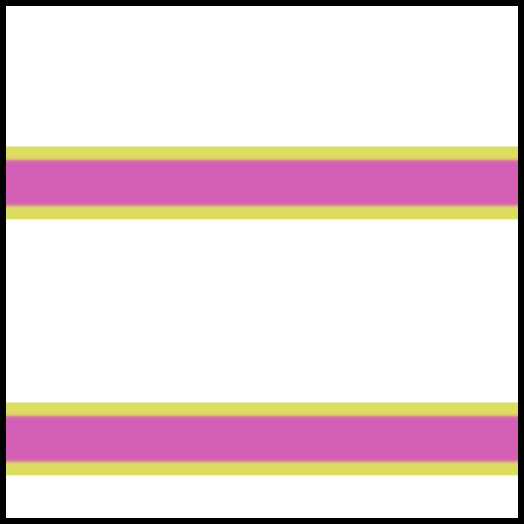}}\parbox[b][][t]{18.4ex}{\centering\includegraphics[width=12.61148ex]{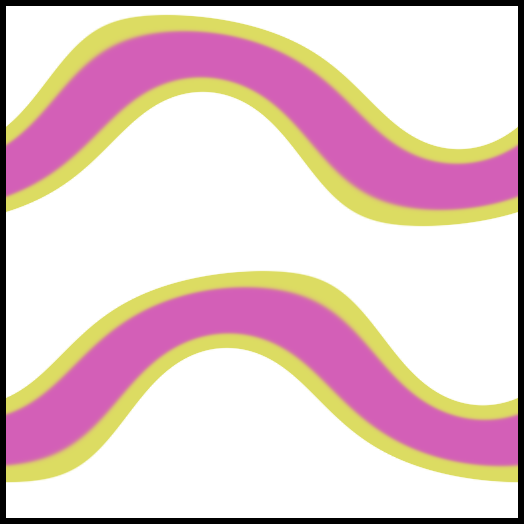}}\parbox[b][][t]{18.4ex}{\centering\includegraphics[width=12.61148ex]{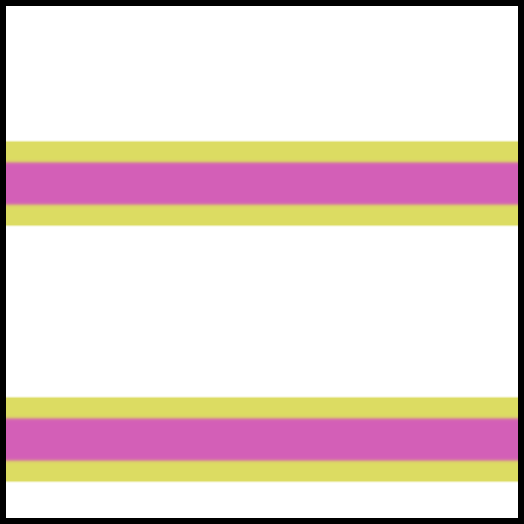}}\parbox[b][][t]{18.4ex}{\centering\includegraphics[width=13.242055ex]{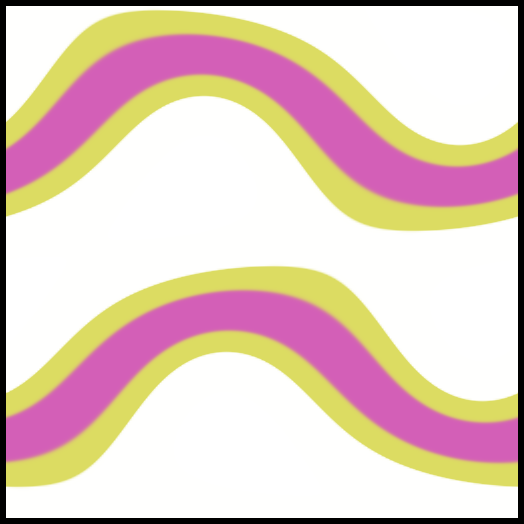}}

\parbox[t][][b]{18.4ex}{\centering\includegraphics[width=17.49335ex]{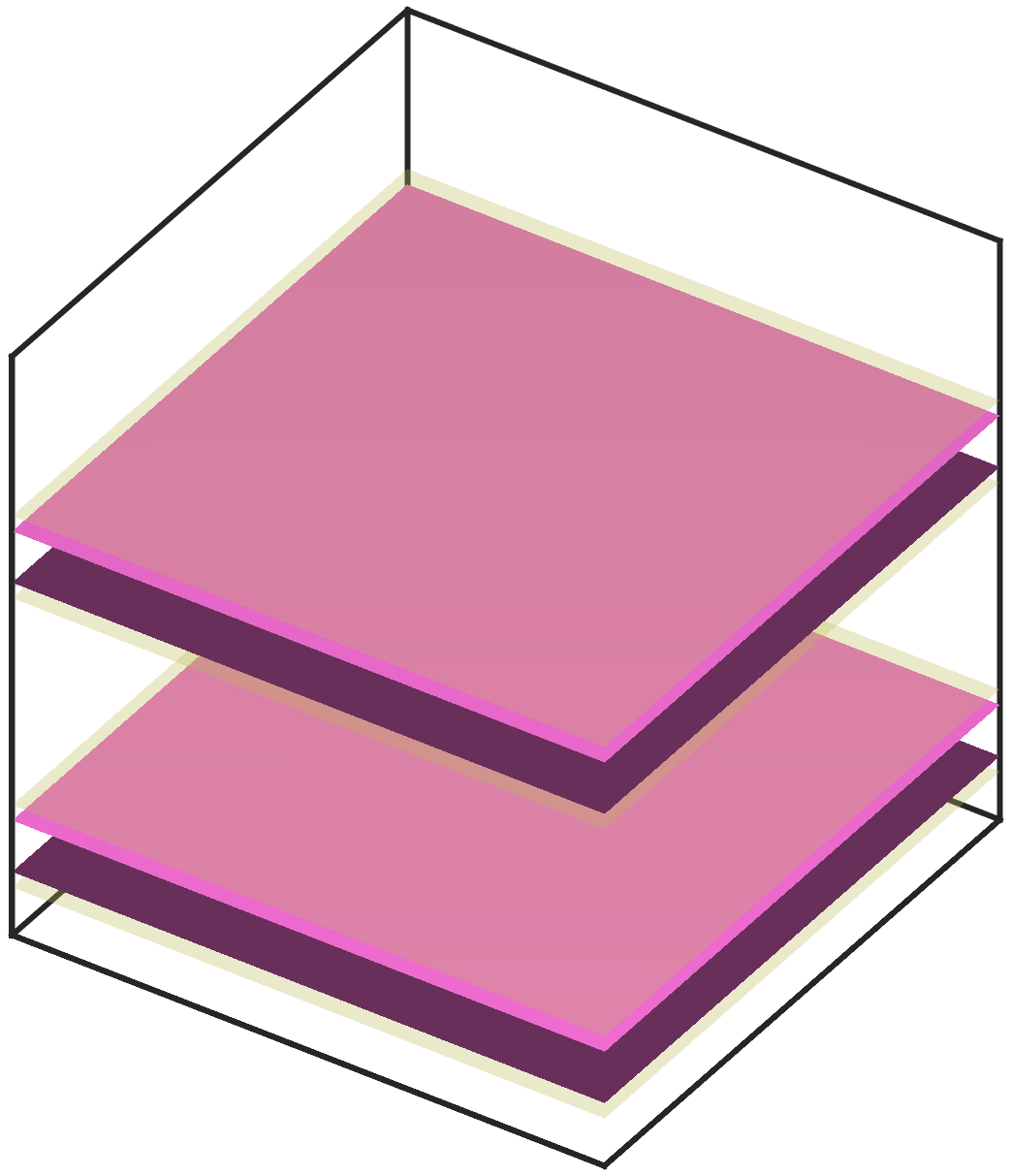}}\parbox[t][][b]{18.4ex}{\centering\includegraphics[width=17.49335ex]{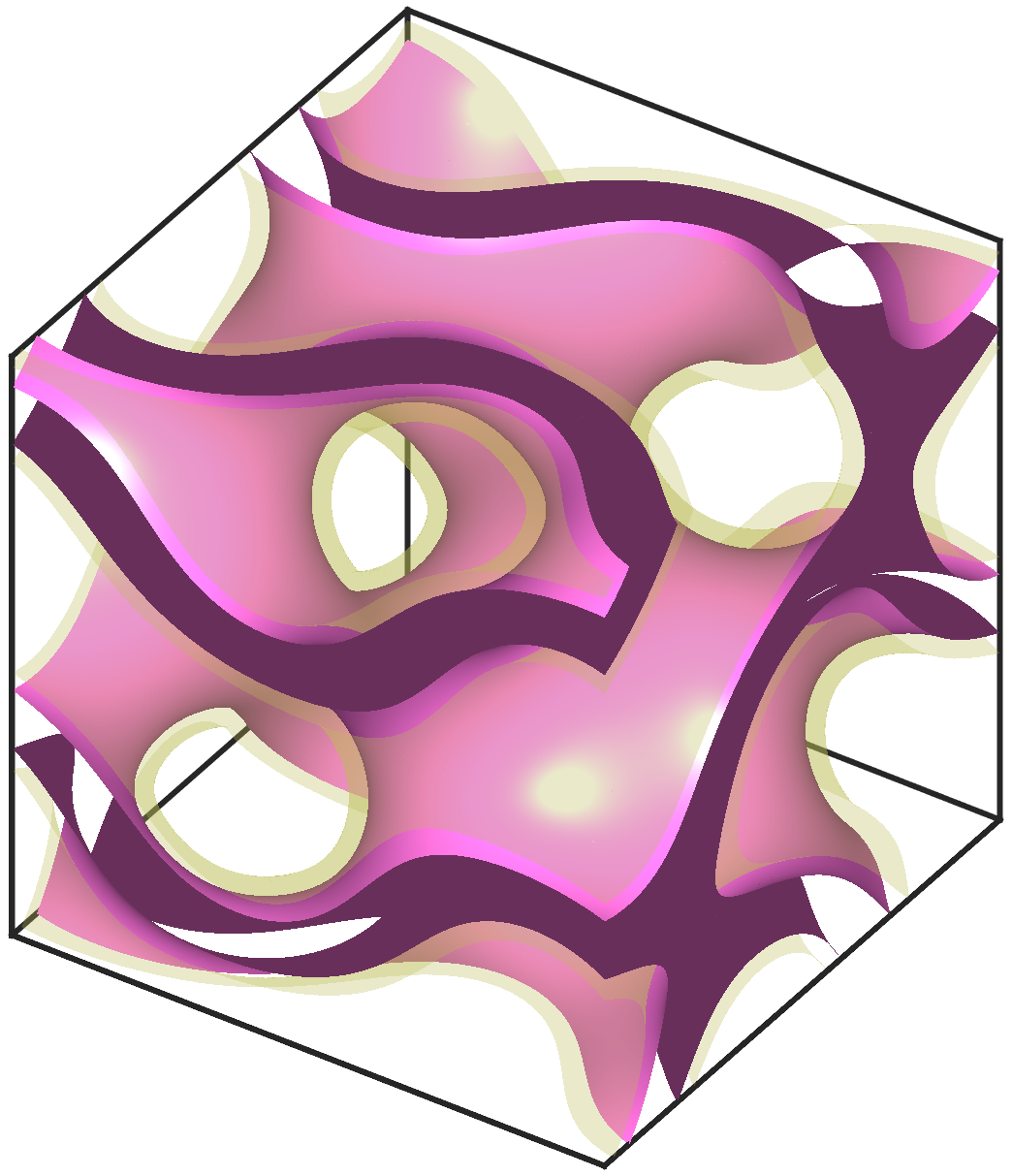}}\parbox[t][][b]{18.4ex}{\centering\includegraphics[width=17.49335ex]{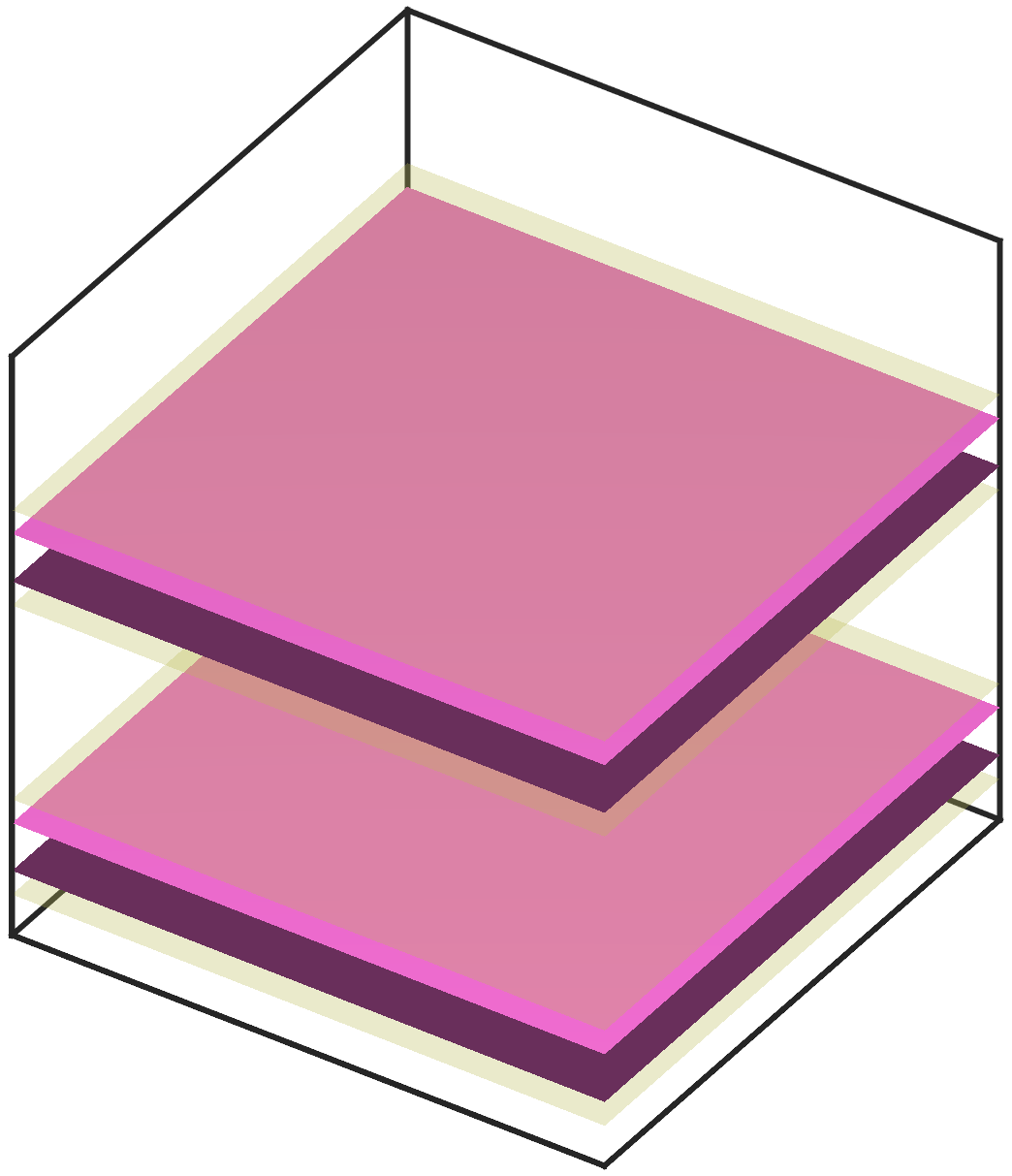}}\parbox[t][][b]{18.4ex}{\centering\includegraphics[width=18.36802ex]{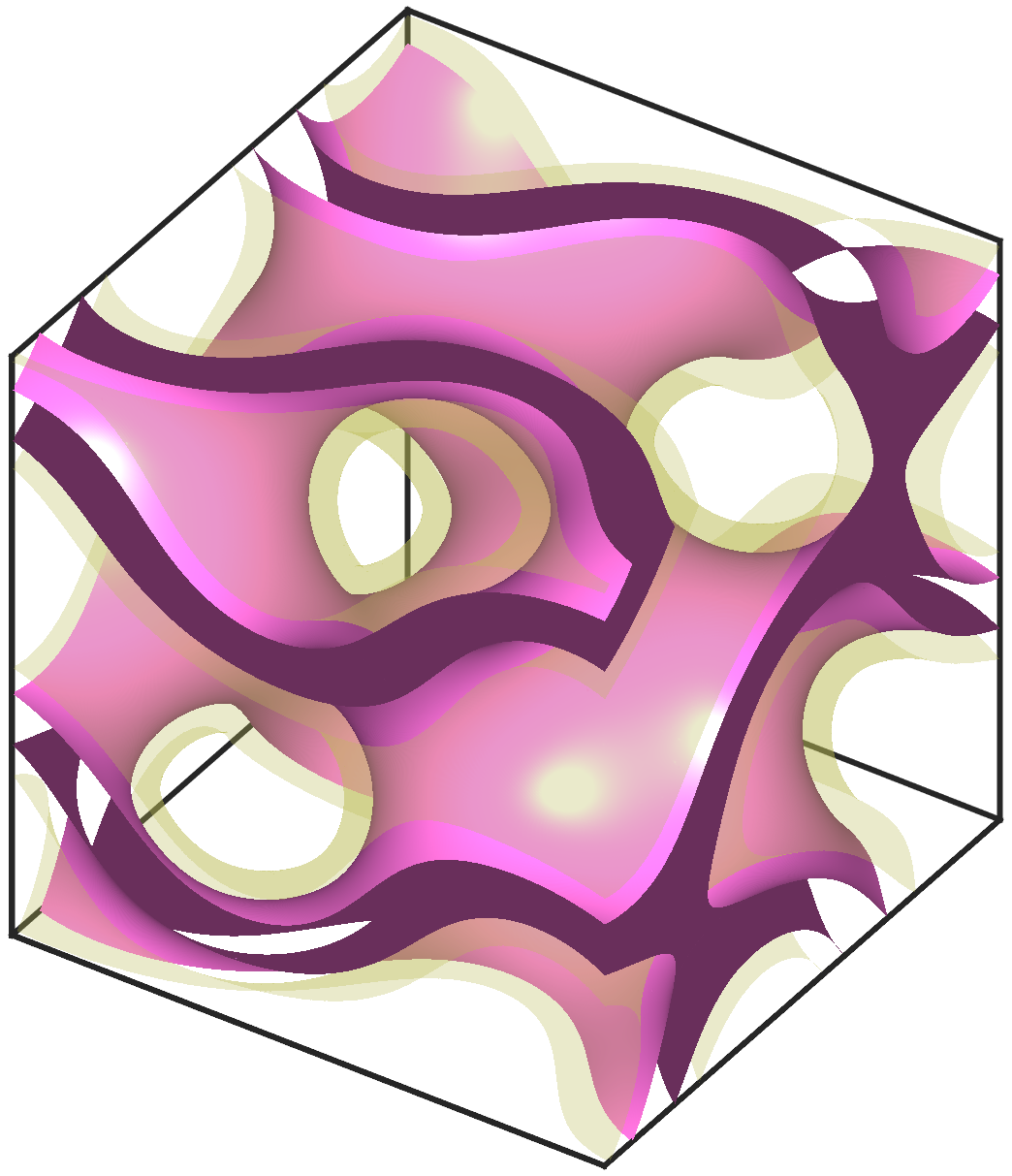}}

\caption{Four local minimizers. From left to right: $\zeta=0.6,\,0.6,\,1,\,1$. Next to each local minimizer is the energy-to-mass ratio.}
\label{figure gyroid simulation}
\end{figure}
In theory, the energy-to-mass ratio of the planar bilayer should equal that of the liposome in the limit of $m\to\infty$, which is given by $\sqrt[3]{9\gamma(\zeta\!+\!1)/8}$ in Corollary \ref{simple asymptotics of the optimal liposome candidate}. With $\gamma=500$, this constant equals $9.6549$ and $10.4004$ for $\zeta=0.6$ and $1$, respectively, which are slightly different than the respective numerical energy-to-mass ratios of the first and third local minimizers shown in Figure \ref{figure gyroid simulation}, with the relative error being $-0.35\%$ and $-0.25\%$, respectively. We believe that this error is due to the diffuse interfaces used in our simulations. In fact, the initial values in this simulation are interpolated from the terminal values obtained in another simulation with a coarser grid, where we chose $\veps=7\times10^{-2}$ and $\texttt{M}=\texttt{N}=\texttt{P}=256$, with other parameters being the same. In the simulation with a coarser grid, we obtained four local minimizers similar to those shown in Figure \ref{figure gyroid simulation}, with their energy-to-mass ratios being $9.5211$, $9.4955$, $10.2989$, and $10.3184$, respectively from left to right. As $\veps\to0$ and $\texttt{M}\,,\texttt{N}\,,\texttt{P}\to\infty$, we expect that the first and third energy-to-mass ratios converge to their respective theoretical values, and that the second and fourth ones maintain their relative differences to the first and third ones, respectively.

\section{Discussion}
\label{section discussion}
In Figure \ref{MorphologyCoefficient}, we can see that as $\zeta$ increases, the optimal morphology should transition from bilayer membrane to cylindrical micelle to spherical micelle. In the liquid drop model \cite{10.1063/5.0148456}, a ball loses stability when its mass exceeds a threshold. Similarly, we expect that there exist two thresholds of $\zeta$, beyond which the bilayer membrane and cylindrical micelle lose stability, respectively. For $\zeta=1$, the straight bilayer membrane is stable on any 2-D periodic strip \cite[Figure 1]{van2009stability}, and we think that the stability analysis therein can be generalized to any $\zeta\in(0,\infty)$, thus allowing us to determine the threshold of $\zeta$ beyond which the bilayer membrane is unstable. Note that such a threshold should be higher than $\zeta_1$ in Figure \ref{MorphologyCoefficient}.

As we mentioned in Remark \ref{remark on asymptotics of liposome candidates}, the inner $V$ layer of the optimal liposome is slightly thicker but has slightly less mass compared to the outer $V$ layer. We can intuitively explain this phenomenon in Figure \ref{Figure mismatch for curved bilayer}-b, where a mismatch occurs as soon as a lipid bilayer membrane is curved. The inner monolayer becomes slightly more densely packed, while the outer monolayer becomes slightly less crowded.
Those changes will inevitably increase the energy. To alleviate such a problem, some lipids may be transferred from the inner monolayer to the outer monolayer, as shown in Figure \ref{Figure mismatch for curved bilayer}-d. Our results indicate that when a closed bilayer membrane deforms (e.g., from a sphere to an ellipsoid), the lipids in the outer monolayer should flow from low curvature areas to high curvature areas, and vice versa for the inner monolayer. According to the fluid mosaic model \cite{singer1972fluid}, the lipids in a bilayer membrane can move easily within each monolayer, and they can also, albeit relatively slowly, move from one monolayer to the other (a movement known as flip-flop) \cite{allhusen2017ins,porcar2020lipid}. The flip-flop process can be facilitated by certain proteins known as flippase, floppase and scramblase. The flippase moves lipids from the outer monolayer to the inner monolayer (flipping), the floppase does the opposite (flopping), and the scramblase does both.
\begin{figure}[H]
\centering
\includegraphics[width=0.9\textwidth]{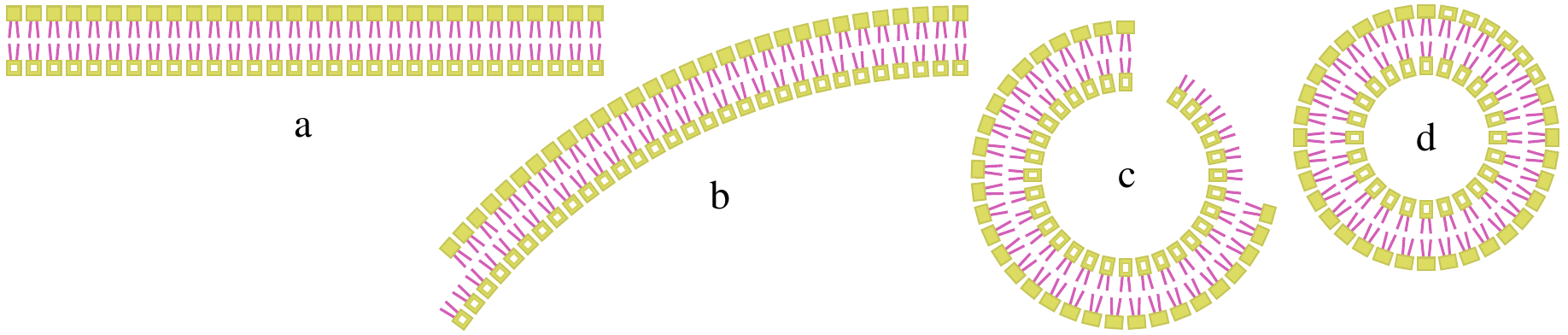}
\caption{Mismatch of the two monolayers in a curved bilayer.}
\label{Figure mismatch for curved bilayer}
\end{figure}

In this paper, we have studied a degenerate version of the Ohta\textendash Kawasaki energy and demonstrated its remarkable ability to reproduce the fascinating phenomena exhibited by self-assembling amphiphiles. We have presented some asymptotic and numerical evidence for the partial localization property of our model. Partial localization is coined in \cite[Section 1.2]{peletier2009partial} and refers to the concentration to lower-dimensional structures. An important example is the bilayer membrane formed by lipids: it is thin along one direction, but relatively large along the other two. Such a structure is vital to the biological membranes of every living cell. Our study may help us better understand the formation of lipid bilayer membranes. In fact, there are only two terms in our energy: the perimeter term which models the immiscibility between water and hydrophobic tails, and a Coulombic nonlocal term which models the attractive force between the heads and tails. It turns out that our model still possesses similar properties even if this Coulombic term is replaced by the 1-Wasserstein distance \cite[Section 9.5]{peletier2009partial}, which penalizes the heads for straying far from the tails. Neither our Coulombic term nor the 1-Wasserstein distance keeps track of which head is connected to which tail. Therefore, the specific structure of the lipid molecule (a head and a tail connected by a covalent bond) is not the essence of partial localization, although it is a practical way to enforce the long-range attractive force that is needed in our model.

For the variant model mentioned above (in which the nonlocal term is the 1-Wasserstein distance), only the case of $\zeta=1$ has been considered in \cite{peletier2009partial,lussardi2014variational}. In view of the rich complexity exhibited in our model for various $\zeta$, it might be interesting to revisit this variant model in the general cases $\zeta>0$. On this note, we also draw attention to the well-known linkage between the 2-Wasserstein distance and the nonlocal (negative) Sobolev space norm \cite{figalli2021invitation}.

Our study is a step towards understanding the pattern formation phenomena from the viewpoint of energetic competition. It is the competition between the short- and long- range terms in the Ohta\textendash Kawasaki energy that gives rise to various interesting mesoscopic periodic patterns that are commonly observed in block copolymers and many other systems \cite{xu2022ternary,choksi2012global,glasner2018multidimensional,muratov2002theory}. We show that in the degenerate case (i.e., only the interface of $U$ is penalized), the Ohta\textendash Kawasaki energy is capable of reproducing the partial localization feature of self-assembling amphiphiles. It is natural to ask how far can our results be generalized. For example, it might be of mathematical interest to explore other variant models with the Euclidean perimeter replaced by the 1-perimeter \cite{goldman2019optimality}, a fractional perimeter \cite{dipierro2017rigidity}, or a general nonlocal perimeter \cite{cesaroni2017isoperimetric}, with the Coulomb potential replaced by a Yukawa potential \cite{fall2018periodic}, a Riesz potential \cite{bonacini2014local}, a fractional inverse Laplacian kernel \cite[Appendix]{chan2019lamellar}, or a general nonlocal kernel \cite{luo2022nonlocal}.

\section*{Acknowledgments}

We would like to thank Xuenan Li, Johan W\"{a}rneg\r{a}rd, and Qi Zhang for helpful discussions. We thank the referee for helpful suggestions. This research is supported in part by US NSF DMS-1937254 and DMS-2309245.  We acknowledge computing resources from Columbia University's Shared Research Computing Facility project, which is supported by NIH Research Facility Improvement Grant 1G20RR030893-01, and associated funds from the New York State Empire State Development, Division of Science Technology and Innovation (NYSTAR) Contract C090171, both awarded April 15, 2010.

\section*{Data availability}
The data that support the findings of this study are openly available in Open Science Framework at \href{http://doi.org/10.17605/OSF.IO/U2896}{DOI:10.17605/OSF.IO/U2896}.

\section*{Declarations of interest}
None.

\appendix

\section{Helfrich and Willmore energies}\label{sec:Wilmore}

In this appendix we provide background on the Helfrich energy and Willmore energy. According to \cite[Equation (2)]{guckenberger2017theory}, the Helfrich energy has been used to model a sheet-like membrane that resembles a regular closed surface $S$ in $\mathbb R^3$: $\int_S\big(\lambda_1(H\!-\!H_0)^2\!+\!\lambda_2K\big)\dd{A}$, where $H$ is the mean curvature (mean of principal curvatures), $H_0$ is the spontaneous curvature, $K$ is the Gaussian curvature, the bending modulus $\lambda_1$ is a positive constant, and the Gaussian (or saddle-splay) modulus $\lambda_2$ is a constant. For monolayers, $H_0$ is nonzero in general; for bilayers, $H_0$ is zero because of symmetry. By the Gauss\textendash Bonnet formula, we have $\int_S K\dd{A}=4\pi(1\!-\!g)$, where $g$ is the genus of $S$, which is a constant as long as $S$ does not undergo topological changes. In this way the Helfrich energy for bilayer membranes can be reduced to the Willmore energy $\int_S H^2\dd{A}$, as long as no topological change happens.

The unique global minimizer of the Willmore energy is a sphere \cite{willmore1965note}. Under an additional constraint that $S$ is of genus $g$, there exists a constrained global minimizer of the Willmore energy for any given $g\in\mathbb N_0$ \cite{bauer2003existence}. For $g=0$ (spherical topology), the sphere is the only minimizer of the Willmore energy, and the minimum energy is $4\pi$. For $g=1$ (toroidal topology), the Clifford torus \cite{marques2014min} is the unique minimizer up to conformal transformations (this is because the Willmore energy is conformal invariant \cite[Section 5.1]{seifert1997configurations}), and the minimum energy is $2\pi^2$. For $g\geqslant2$, the minimizer is unknown but conjectured to be Lawson's surfaces \cite{hsu1992minimizing}. The minimum energy converges to $8\pi$ as $g$ goes to infinity (and is conjectured to be monotonically increasing) \cite{kuwert2010large}.

The above-mentioned toroidal minimizer of the Willmore energy for $g=1$ was conjectured in 1965 \cite{willmore1965note} and proved in 2014 \cite{marques2014min}. Such a long-held conjecture was partly supported by the experimental observation of toroidal structures formed by artificial membranes \cite[Figures 3 and 4]{michalet1995vesicles} (see also \cite{mutz1991observation,fourcade1992experimental}), as shown in Figure \ref{Clifford torus}. In addition, there were some numerical evidence from phase-field simulations of the Willmore energy \cite[Section 4.3.2]{du2006simulating}. This result is also manifested in our numerical simulations (see Figure \ref{figure torus simulation}), thus providing support for the first statement in Conjecture \ref{3-D conjecture of Gamma-convergence}.
\begin{figure}[H]
\centering
\includegraphics[height=70pt]{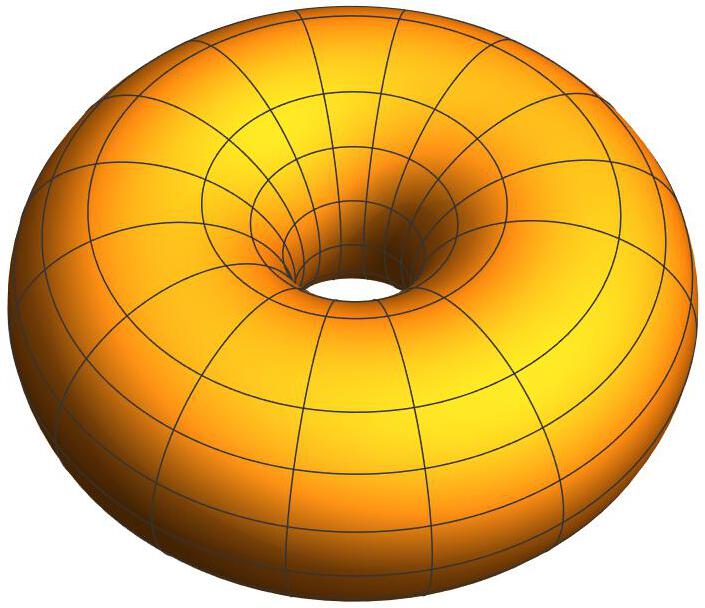}
\qquad
\includegraphics[height=70pt]{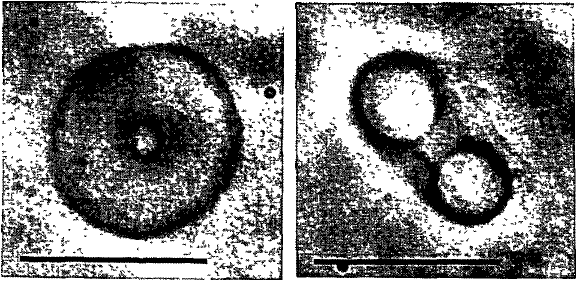}\\
\includegraphics[height=70pt]{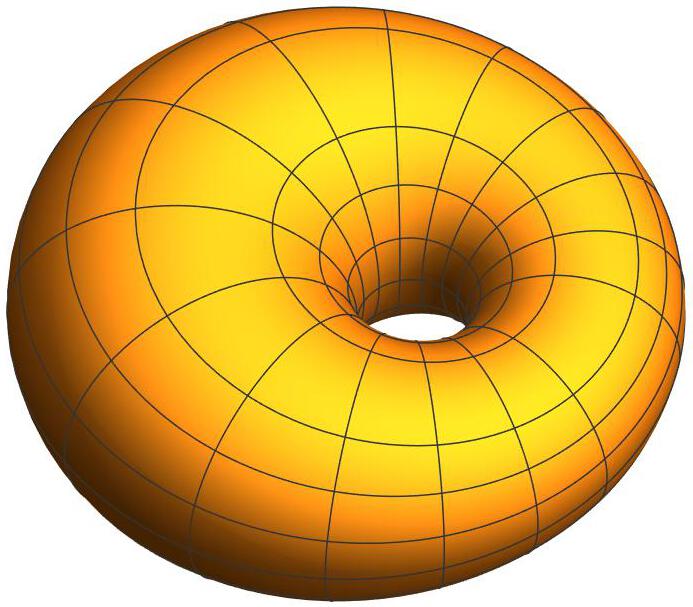}
\qquad
\includegraphics[height=70pt]{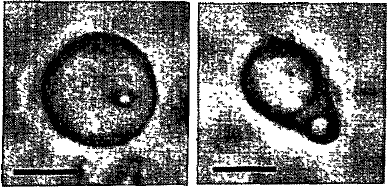}
\caption{Top-left: Clifford torus given by \cite[Equation (5.3)]{seifert1997configurations}. Bottom-left: Its image under a conformal transformation. Top-right: Clifford torus observed in \cite[Figure 3]{michalet1995vesicles}. Bottom-right: Deformed Clifford torus observed in \cite[Figure 4]{michalet1995vesicles}. Bar indicates 10 {\textmu}m. The right part is reproduced from \cite{michalet1995vesicles}.}
\label{Clifford torus}
\end{figure}

\section{Calculations of radially symmetric candidates}
\label{appendix Calculations of radially symmetric candidates}
In this appendix we derive the asymptotic results presented in Section \ref{subsec:Liposome candidates}.
\begin{proposition}
\label{proposition: energy for liposome candidates}
The energy \eqref{sharp energy nonrelaxed} of a liposome candidate is $E(U,V) = \text{Per}\;U +\gamma N(U,V)$, where
\begin{equation*}
\text{Per}\;U=\left\{
\begin{aligned}
&2\pi(R_1\!+\!R_2),&&n=2,\\
&4\pi(R_1^2\!+\!R_2^2),&&n=3.
\end{aligned}
\right.
\end{equation*}
For $n=2$, we have
\begin{equation}
\label{energy radial candidates n = 2}
\begin{aligned}
16\zeta^2N(U,V)/\pi = (&1\!-\!\zeta ^2)\big(R_2^4-R_1^4\big)+R_0^4-R_3^4+4 \Big(R_3^4 \ln R_3-R_0^4 \ln R_0\\
&+(\zeta\!+\!1)\big(2 R_0^2 R_1^2-(\zeta\!+\!1) R_1^4\big) \ln R_1-(\zeta\!+\!1)\big(2 R_3^2 R_2^2-(\zeta \!+\!1) R_2^4\big) \ln R_2\Big).
\end{aligned}
\end{equation}
For $n=3$, we have
\begin{equation}
\label{energy radial candidates n = 3}
\begin{aligned}
15 \zeta ^2N(U,V)/\pi = 6\big(R_0^5-R_3^5\big)+(\zeta\!+\!1)\Big(10 R_2^2 R_3^3-(6 \zeta \!+\!4) R_2^5+(6 \zeta\!+\!4) R_1^5-10 R_0^3 R_1^2\Big).
\end{aligned}
\end{equation}
\end{proposition}

\begin{proposition}
\label{Lagrange multiplier equation}
Any minimizer of $E(U,V)$ among the liposome candidates must satisfy the following conditions if $n=2$,
\begin{align}
R_3^2 \ln R_3^2-R_0^2 \ln R_0^2&=(\zeta\!+\!1) \big(R_2^2 \ln R_2^2-R_1^2 \ln R_1^2\big),\notag\\
4 \zeta^2(\zeta\!+\!1)^{-1}\gamma^{-1}\big(R_1^{-1}\!+\!R_2^{-1}\big)&=\zeta\big(R_2^2\!-\!R_1^2\big)+\big(R_0^2-(\zeta\!+\!1) R_1^2\big) \ln\big(R_2^2/R_1^2\big),\notag
\end{align}
and satisfy the following equations if $n=3$ (see also \cite[Equation (A.1)]{bonacini2016optimal}),
\begin{equation*}
\begin{aligned}
R_3^2\!-\!R_0^2&=(\zeta\!+\!1)(R_2^2\!-\!R_1^2),\\
12\zeta^2\gamma^{-1}(R_1^{-1}\!+\!R_2^{-1})&=(3 \zeta\!+\!2)\big(R_3^2\!-\!R_0^2\big)+2(\zeta \!+\!1)\big(R_0^3/R_1\!-\!R_3^3/R_2\big).
\end{aligned}
\end{equation*}
\end{proposition}

\begin{theorem}
\label{asymptotics of the optimal liposome candidate}
With $\zeta$ and $\gamma$ fixed, as $m\to\infty$, the minimizer of $E(U,V)$ among the liposome candidates has the following asymptotics if $n=2$,
\begin{equation*}
\begin{aligned}
&\frac{E(U,V)}m = \sqrt[3]{\gamma\frac{\zeta\!+\!1}{8/9}} + \frac{8 \pi ^2}{5}\frac{\zeta ^2\!+\!4 \zeta\!+\!1}{\gamma  (\zeta\!+\!1) m^2} + O\Big(\frac{1}{m^3}\Big),
\hspace{-25pt}
&R_2\!-\!R_1 = \sqrt[3]{\frac{24/\gamma}{\zeta\!+\!1}}
+
O\Big(\frac{1}{m^2}\Big),\\
&\big(R_3^2\!-\!R_2^2\big)-\big(R_1^2\!-\!R_0^2\big)=\frac{4 \zeta  (\zeta\!+\!2)}{\sqrt[3]{3(\gamma\zeta\!+\!\gamma)^2}}+O\Big(\frac{1}{m^2}\Big),
\hspace{-25pt}
&\frac{R_1\!+\!R_2}2 = \frac{m}{4\pi}\sqrt[3]{\gamma\frac{\zeta\!+\!1}{3}}
+
O\Big(\frac1{m}\Big),\\
&R_1\!-\!R_0 = \sqrt[3]{\frac{3}{\gamma}\frac{\zeta^3}{\zeta\!+\!1}}
+
\frac{2 \pi  \zeta}{\gamma  m}\frac{\zeta\!+\!2}{\zeta\!+\!1}
+
O\Big(\frac{1}{m^2}\Big),
\hspace{-25pt}
&R_3\!-\!R_2 = \sqrt[3]{\frac{3}{\gamma}\frac{\zeta^3}{\zeta\!+\!1}}
-
\frac{2 \pi  \zeta}{\gamma  m}\frac{\zeta\!+\!2}{\zeta\!+\!1}
+
O\Big(\frac{1}{m^2}\Big),
\end{aligned}
\end{equation*}
and has the following asymptotics if $n=3$,
\begin{equation*}
\begin{aligned}
&\frac{E(U,V)}m = \sqrt[3]{\gamma\frac{\zeta\!+\!1}{8/9}} + \frac{4 \pi}{15m}\frac{\zeta ^2\!+\!4 \zeta\!+\!16}{\big(\gamma(\zeta\!+\!1)/3\big)^{2/3}} + O\Big(\frac{1}{m^{3/2}}\Big),
\hspace{-35pt}
&R_2\!-\!R_1 = \sqrt[3]{\frac{24/\gamma}{\zeta\!+\!1}}+O\Big(\frac{1}{m}\Big),\\
&\big(R_3^3\!-\!R_2^3\big)-\big(R_1^3\!-\!R_0^3\big)=\frac{\sqrt{6m}\zeta(\zeta\!+\!2)}{\sqrt{\pi\gamma  (\zeta\!+\!1)}}+O\Big(\frac{1}{m^{1/2}}\Big),
\hspace{-35pt}
&\frac{R_1\!+\!R_2}2 =
\frac{\sqrt[6]{\gamma  (\zeta\!+\!1)/3}}{2\sqrt{2 \pi/m} }
+
O\Big(\frac{1}{m^{1/2}}\Big),\\
&R_1\!-\!R_0 = \sqrt[3]{\frac{3}{\gamma}\frac{\zeta^3}{\zeta\!+\!1}}
+
\frac{(\zeta\!+\!2) \sqrt{8\pi/m}}{\sqrt[6]{3(\gamma\zeta\!+\!\gamma)^5}\big/\zeta}
+
O\Big(\frac{1}{m}\Big),
\hspace{-35pt}
&R_3\!-\!R_2 = \sqrt[3]{\frac{3}{\gamma}\frac{\zeta^3}{\zeta\!+\!1}}
-
\frac{(\zeta\!+\!2) \sqrt{8\pi/m}}{\sqrt[6]{3(\gamma\zeta\!+\!\gamma)^5}\big/\zeta}
+
O\Big(\frac{1}{m}\Big).
\end{aligned}
\end{equation*}

\end{theorem}

\begin{proposition}
\label{equal volume asymptotics proposition}
Under the additional assumption that the inner and outer $V$ layers have the same mass \cite[Equation (5.16)]{van2008partial}, i.e.,
\begin{equation}
\label{equal volume constraint}
R_3^n\!-\!R_2^n=R_1^n\!-\!R_0^n,
\end{equation}
with $\zeta$ and $\gamma$ fixed, as $m\to\infty$, the minimizer of $E(U,V)$ among the liposome candidates (satisfying \eqref{equal volume constraint}) has the following asymptotics if $n=2$,
\begin{equation*}
\begin{aligned}
&\frac{E(U,V)}m = \sqrt[3]{\gamma\frac{\zeta\!+\!1}{8/9}}
+
24 \pi ^2\frac{ 2 \zeta ^2\!+\!8 \zeta\!+\!7}{5 \gamma  (\zeta\!+\!1) m^2}
+
O\Big(\frac{1}{m^3}\Big),\hspace{-20pt}&\\
&R_1\!-\!R_0
=
\sqrt[3]{\frac{3}{\gamma}\frac{\zeta^3}{\zeta\!+\!1}}
+
\frac{6 \pi \zeta}{m \gamma}\frac{\zeta\!+\!2}{\zeta\!+\!1}
+
O\Big(\frac{1}{m^2}\Big),
\hspace{-20pt}
&R_2\!-\!R_1
=
\sqrt[3]{\frac{24/\gamma}{\zeta\!+\!1}}
+
O\Big(\frac{1}{m^2}\Big),\\
&R_3\!-\!R_2
=
\sqrt[3]{\frac{3}{\gamma}\frac{\zeta^3}{\zeta\!+\!1}}
-
\frac{6 \pi \zeta}{m \gamma}\frac{\zeta\!+\!2}{\zeta\!+\!1}
+
O\Big(\frac{1}{m^2}\Big),
\hspace{-20pt}
&\frac{R_1\!+\!R_2}2
=
\frac{\sqrt[3]{\gamma  (\zeta\!+\!1)/3}}{4 \pi/m}
+
O\Big(\frac{1}{m}\Big),
\end{aligned}
\end{equation*}
and has the following asymptotics if $n=3$,
\begin{equation*}
\begin{aligned}
&\frac{E(U,V)}m = \sqrt[3]{\gamma\frac{\zeta\!+\!1}{8/9}} + \frac{4 \pi}{5 m}\frac{7 \zeta ^2\!+\!28 \zeta\!+\!32}{\big(\gamma(\zeta\!+\!1)/3\big)^{2/3}} + O\Big(\frac{1}{m^{3/2}}\Big),\hspace{-20pt}&\\
&R_1\!-\!R_0
=
\sqrt[3]{\frac{3}{\gamma}\frac{\zeta^3}{\zeta\!+\!1}}
+
\frac{\zeta(\zeta \!+\!2) \sqrt{8 \pi/m}}{\big((\gamma\zeta\!+\!\gamma)/3\big)^{5/6}}
+
O\Big(\frac{1}{m}\Big),
\hspace{-20pt}
&R_2\!-\!R_1
=
\sqrt[3]{\frac{24/\gamma}{\zeta\!+\!1}}
+
O\Big(\frac{1}{m}\Big),\\
&R_3\!-\!R_2
=
\sqrt[3]{\frac{3}{\gamma}\frac{\zeta^3}{\zeta\!+\!1}}
-
\frac{\zeta(\zeta \!+\!2) \sqrt{8 \pi/m}}{\big((\gamma\zeta\!+\!\gamma)/3\big)^{5/6}}
+
O\Big(\frac{1}{m}\Big),
\hspace{-20pt}
&\frac{R_1\!+\!R_2}2
=
\frac{\sqrt[6]{\gamma  (\zeta\!+\!1)/3}}{2\sqrt{2 \pi/m} }
+
O\Big(\frac{1}{m^{1/2}}\Big).
\end{aligned}
\end{equation*}
\end{proposition}

\begin{proof} [of Proposition \ref{proposition: energy for liposome candidates}]
Similar to \cite[Page 106]{van2008copolymer}, we compute the electrostatic potential $\phi$ and then use \eqref{alternative expression of N(U,V)} to obtain $N(U,V)$. We know that $\phi(\vec x)$ is radially symmetric and can be written as $\phi(r)$ with $r=|\vec x|$, satisfying
\begin{equation*}
-r^{1-n}\frac\dd{\dd r}\big(r^{n-1}\phi'(r)\big)=[R_1\!\leq\!r\!\leq\!R_2]-\frac{[R_0\!\leq\!r\!\leq\!R_1\;\text{or}\;R_2\!\leq\!r\!\leq\!R_3]}\zeta,
\end{equation*}
where the left-hand side is due to Laplacian expressed in spherical coordinates, and $[\,\cdot\,]$ on the right-hand side is the Iverson bracket. Since $\phi$ is a continuously differentiable even function, we have $\phi'(0)=0$. We further require $\phi(\infty)=0$.

\noindent For $n=2$, we obtain
\begin{equation*}
4\zeta\phi(r)=\left\{
\begin{aligned}
&0,&&R_3\!<\!r\!<\!\infty\,,\\
&2\big(R_3^2\ln R_3\big)-R_3^2+r^2-2 R_3^2 \ln r,&&R_2\!<\!r\!<\!R_3\,,\\
&2 \big(R_3^2 \ln R_3\!-\!(\zeta\!+\!1)R_2^2 \ln R_2\big)-R_3^2+(\zeta\!+\!1) R_2^2-\zeta  r^2-2\big(R_3^2-(\zeta\!+\!1)  R_2^2\big)\ln r,&&R_1\!<\!r\!<\!R_2\,,\\
&2\big(R_3^2\ln R_3\!-\!(\zeta\!+\!1) R_2^2 \ln R_2\!+\!(\zeta\!+\!1) R_1^2 \ln R_1\big)-R_0^2+r^2-2 R_0^2 \ln r,&&R_0\!<\!r\!<\!R_1\,,\\
&2\big(R_3^2\ln R_3\!-\!(\zeta\!+\!1) R_2^2 \ln R_2\!+\!(\zeta\!+\!1) R_1^2 \ln R_1\!-\!R_0^2 \ln R_0\big),&&\parbox{\widthof{$R_0$}}{\raggedleft0}\!<\!r\!<\!R_0\,.
\end{aligned}
\right.
\end{equation*}
For $n=3$, we obtain
\begin{equation*}
6\zeta\phi(r)=\left\{
\begin{aligned}
&0,&&R_3\!<\!r\!<\!\infty\,,\\
&-3\big(R_3^2\big)+2R_3^3/r+r^2,&&R_2\!<\!r\!<\!R_3\,,\\
&-3\big(R_3^2\!-\!(\zeta\!+\!1) R_2^2\big)+2\big(R_3^3\!-\!(\zeta\!+\!1) R_2^3\big)/r-\zeta  r^2,&&R_1\!<\!r\!<\!R_2\,,\\
&-3\big(R_3^2\!-\!(\zeta\!+\!1) R_2^2\!+\!(\zeta\!+\!1) R_1^2\big)+2R_0^3/r+r^2,&&R_0\!<\!r\!<\!R_1\,,\\
&-3\big(R_3^2\!-\!(\zeta\!+\!1) R_2^2\!+\!(\zeta\!+\!1) R_1^2\!-\!R_0^2\big),&&\parbox{\widthof{$R_0$}}{\raggedleft0}\!<\!r\!<\!R_0\,.
\end{aligned}
\right.
\end{equation*}
Notice that we have used \eqref{volume constraint radial candidates} to simplify the expressions for $\phi$.
\end{proof}

\begin{proof} [of Proposition \ref{Lagrange multiplier equation}]
We minimize $E(U,V)$ given by Proposition \ref{proposition: energy for liposome candidates} with constraints \eqref{volume constraint radial candidates} and \eqref{2d 3d volume constraint radial candidates}.

For $n=2$, we obtain
\begin{equation*}
\begin{aligned}
&R_3^2 \ln R_3-(\zeta\!+\!1) R_2^2 \ln R_2=\lambda=R_0^2 \ln R_0-(\zeta\!+\!1) R_1^2 \ln R_1,\quad\text{and}\\
&\phantom{\;=\;}R_0^2 \Big(\!\ln R_1+\frac12\Big)-R_1^2 \Big((\zeta\!+\!1) \ln R_1+\frac12\Big)+\frac{2 \zeta ^2 R_1^{-1}}{\gamma(\zeta\!+\!1)}=\lambda-\frac{\mu}{\zeta\!+\!1}\\
&=R_3^2 \Big(\!\ln R_2+\frac12\Big)-R_2^2 \Big((\zeta\!+\!1) \ln R_2+\frac12\Big)-\frac{2 \zeta ^2R_2^{-1}}{\gamma(\zeta\!+\!1)},
\end{aligned}
\end{equation*}
where $\lambda$ and $\mu$ are Lagrange multipliers. From the last two equalities, we obtain
\begin{equation*}
\begin{aligned}
\frac{4 \zeta ^2(R_1^{-1}\!+\!R_2^{-1})}{\gamma(\zeta\!+\!1)}&=R_3^2 \big(\!\ln R_2^2+1\big)-R_2^2 \big((\zeta\!+\!1) \ln R_2^2+1\big)+R_1^2 \big((\zeta\!+\!1) \ln R_1^2+1\big)-R_0^2 \big(\!\ln R_1^2+1\big)\\
&=R_3^2-R_0^2+R_1^2-R_2^2+(\zeta\!+\!1)\big(R_1^2 \ln R_1^2-R_2^2 \ln R_2^2\big)+R_3^2\ln R_2^2-R_0^2\ln R_1^2\\
&=\zeta(R_2^2\!-\!R_1^2)+(\zeta\!+\!1)\big(R_1^2 \ln R_1^2-(R_1^2\!+\!m/\pi)\ln R_2^2\big)+\big(R_0^2+(\zeta\!+\!1)m/\pi\big)\ln R_2^2-R_0^2\ln R_1^2\\
&=\zeta(R_2^2\!-\!R_1^2)+\big(R_0^2-(\zeta\!+\!1)R_1^2\big)\ln\big(R_2^2/R_1^2\big).
\end{aligned}
\end{equation*}
where the third equality is due to \eqref{volume constraint radial candidates} and \eqref{2d 3d volume constraint radial candidates}.

For $n=3$, we obtain
\begin{equation*}
\begin{aligned}
(\zeta\!+\!1) R_1^2-R_0^2=\lambda&/3=(\zeta\!+\!1) R_2^2-R_3^2,\quad\text{and}\\
2\frac{R_0^3}{R_1}-(3 \zeta\!+\!2) R_1^2-\frac{12\zeta ^2R_1^{-1}}{(\zeta \!+\!1)\gamma}=\mu&-\lambda=2\frac{R_3^3}{R_2}-(3 \zeta\!+\!2)R_2^2+\frac{12\zeta ^2R_2^{-1}}{(\zeta \!+\!1)\gamma },
\end{aligned}
\end{equation*}
where $\mu$ and $\lambda$ are Lagrange multipliers.
\end{proof}

\begin{proof} [of Theorem \ref{asymptotics of the optimal liposome candidate}]

\noindent\textbf{The 2-D case ($n=2$)}\\
We assume that $\zeta$ and $\gamma$ are fixed. To obtain the asymptotics of $R_i$ as $m\to\infty$, we use the change of variables $r_i=R_i^2\pi/m$ and $\Gamma^{-1}=4\zeta^2(\zeta\!+\!1)^{-1}(\pi/m)^{3/2}/\gamma$, therefore transforming \eqref{volume constraint radial candidates}, \eqref{2d 3d volume constraint radial candidates} and Proposition \ref{Lagrange multiplier equation} into the following
\begin{equation*}
\begin{aligned}
r_3\!-\!r_0=\zeta\!+\!1&=(\zeta\!+\!1)\,(r_2\!-\!r_1),\\
r_3 \ln r_3-r_0 \ln r_0&=(\zeta\!+\!1) \left(r_2 \ln r_2-r_1 \ln r_1\right),\\
\Gamma^{-1}\big(1/\sqrt{r_1}+1/\sqrt{r_2}\big)&=\zeta(r_2\!-\!r_1)+\big(r_0-(\zeta\!+\!1)r_1\big)\ln(r_2/r_1),
\end{aligned}
\end{equation*}
or equivalently,
\begin{equation*}
\begin{aligned}
r_3=r_0\!+\!\zeta\!+\!1,&\quad r_2=r_1\!+\!1,\\
\frac{r_3}{r_1} \ln \frac{r_3}{r_1}-\frac{r_0}{r_1} \ln \frac{r_0}{r_1}&=(\zeta\!+\!1) \Big(\frac{r_2}{r_1} \ln \frac{r_2}{r_1}-\frac{r_1}{r_1} \ln \frac{r_1}{r_1}\Big),\\
\Gamma^{-1}\big(1/\sqrt{r_1}+1/\sqrt{r_2}\big)/r_1&=\zeta(r_2/r_1\!-\!1)+\big(r_0/r_1-(\zeta\!+\!1)\big)\ln(r_2/r_1).
\end{aligned}
\end{equation*}
Using the change of variables $a=1/r_1$ and $b=1-r_0/r_1$, we obtain
\begin{align}
\big(1\!+\!a (\zeta\!+\!1)\!-\!b\big) \ln\big(1\!+\!a (\zeta\!+\!1)\!-\!b\big) -(1\!-\!b) \ln(1\!-\!b)&=(\zeta\!+\!1)\,(1\!+\!a) \ln(1\!+\!a),\label{2d transformed equation in terms of a,b,c,d -- 1}\\
\Gamma^{-1}\Big(\sqrt{a}+\sqrt{a/(1\!+\!a)}\Big)a &= a \zeta - (b\!+\!\zeta)\ln(1\!+\!a).\label{2d transformed equation in terms of a,b,c,d -- 2}
\end{align}
Using the ansatz that $a,b\to0$ as $\Gamma\to\infty$, by Taylor-expanding \eqref{2d transformed equation in terms of a,b,c,d -- 1} around $a,b=0$, we obtain
\begin{equation*}
O(a^7\!+\!b^7)=a(\zeta\!+\!1)\Big(\frac{a \zeta }{2}-b-\zeta\frac{\zeta\!+\!2}6a^2+\frac{\zeta\!+\!1}{2}a b-\frac{b^2}{2}+\cdots\Big).
\end{equation*}
Assuming $b = p_1 a + p_2 a^2 + p_3 a^3 + p_4 a^4 + p_5 a^5 + O(a^6)$ and plugging it into the above equation, we obtain
\begin{equation*}
O(a^7)=a^2(\zeta\!+\!1)\Big(\frac{\zeta }{2}\!-\!p_1\Big)
-
a^3(\zeta\!+\!1)\Big(\zeta\frac{\zeta\!+\!2}{6}-\frac{\zeta\!+\!1}{2} p_1+\frac{p_1^2}{2}+p_2\Big)
+
\cdots,
\end{equation*}
from which we obtain the following solution
\begin{equation}
\label{2d expansion for b}
p_1=\frac{\zeta}{2},\quad p_2=-\zeta\frac{\zeta\!+\!2}{24},\quad p_3=\zeta\frac{\zeta\!+\!2}{48},\quad p_4 = -\zeta\frac{3 \zeta ^2\!+\!6\zeta\!+\!76}{5760/(\zeta\!+\!2)},\quad p_5 = \zeta\frac{\zeta ^2\!+\!2 \zeta\!+\!12}{1280/(\zeta\!+\!2)}.
\end{equation}
By plugging \eqref{2d expansion for b} into \eqref{2d transformed equation in terms of a,b,c,d -- 2} and Taylor-expanding it around $a=0$, we obtain
\begin{equation*}
\Gamma^{-1}=\frac{\zeta ^2}{48}\sqrt{a^3}-\frac{\zeta ^2}{64} \sqrt{a^5}+\frac{ \zeta^2\!+\!4 \zeta\!+\!46}{3840/\zeta ^2}\sqrt{a^7}-7 \frac{ \zeta ^2\!+\!4 \zeta \!+\!21}{15360/\zeta ^2}\sqrt{a^9}+O(\sqrt{a^{11}})\,,
\end{equation*}
from which we obtain
\begin{equation*}
\begin{aligned}
a &= \frac{4}{\big(\Gamma  \zeta ^2/6\big)^{2/3}} + \frac{8}{\big(\Gamma  \zeta ^2/6\big)^{4/3}} - 4\frac{2 \zeta ^2\!+\!8 \zeta\!-\!43}{15\big(\Gamma  \zeta ^2/6\big)^2} - 16\frac{2 \zeta ^2\!+\!8 \zeta\!-\!13}{15\big(\Gamma  \zeta ^2/6\big)^{8/3}} + O\Big(\frac{1}{\Gamma ^{10/3}}\Big)\\
&=
\frac{2^4\pi/m}{ \big(\gamma  (\zeta\!+\!1)/3\big)^{2/3}}
+
\frac{2^7\pi ^2/m^2}{\big(\gamma  (\zeta\!+\!1)/3\big)^{4/3}}
-
\frac{2^8 \pi ^3 (2 \zeta ^2\!+\!8 \zeta\!-\!43)}{15 m^3\big(\gamma (\zeta\!+\!1)/3\big)^2}
-
\frac{2^{12}\pi ^4 (2 \zeta ^2\!+\!8 \zeta\!-\!13)}{15 m^4 \big(\gamma  (\zeta\!+\!1)/3\big)^{8/3}}
+
O\Big(\frac{1}{m^5}\Big),
\end{aligned}
\end{equation*}
where the second equality is due to $\Gamma^{-1}=4\zeta^2(\zeta\!+\!1)^{-1}(\pi/m)^{3/2}/\gamma$. According to the definition of $a$ and $b$, we have
\begin{equation*}
R_0 = \sqrt{m/\pi} \sqrt{(1\!-\!b)/a}, \quad R_1 = \sqrt{m/\pi}\sqrt{1/a}, \quad R_2 = \sqrt{m/\pi}\sqrt{1+1/a}, \quad R_3 = \sqrt{m/\pi} \sqrt{(1\!-\!b)/a\!+\!\zeta\!+\!1},
\end{equation*}
which allow us to compute the asymptotics of $E(U,V)$, $R_{i+1}\!-\!R_i$, $(R_{1}\!+\!R_2)/2$, and $\big(R_3^2\!-\!R_2^2\big)-\big(R_1^2\!-\!R_0^2\big)$.

\noindent\textbf{The 3-D case ($n=3$)}\\
Similar to the 2-D case, we use the rescaling $r_i=R_i\sqrt[3]{4\pi/(3m)}$ and $\Gamma=3m\gamma/(4\pi)$. Using the change of variables $a=r_0^{-3}$ and $b=r_1/r_0\!-\!1$, we transform \eqref{volume constraint radial candidates}, \eqref{2d 3d volume constraint radial candidates} and Proposition \ref{Lagrange multiplier equation} into the following
\begin{align}
d^2\!-\!1&=(\zeta\!+\!1) \big(c^2\!-\!(b\!+\!1)^2\big),\label{transformed equation in terms of a,b,c,d -- 1}\\
12 \zeta ^2\Gamma^{-1}\big((b\!+\!1)^{-1}\!+\!c^{-1}\big)a&=(3 \zeta\!+\!2)(d^2\!-\!1)+2(\zeta\!+\!1)\big((b\!+\!1)^{-1}\!-\!d^3/c\big),\label{transformed equation in terms of a,b,c,d -- 2}
\end{align}
where $c=\sqrt[3]{a+(b\!+\!1)^3}$ and $d=\sqrt[3]{a(\zeta\!+\!1)+1}$. Using the ansatz that $a,b\to0$ as $\Gamma\to\infty$, by Taylor-expanding \eqref{transformed equation in terms of a,b,c,d -- 1} around $a,b=0$, we obtain
\begin{equation*}
O(a^7\!+\!b^7)=a(\zeta\!+\!1)\Big(\frac{2 b}{3}-\frac{a \zeta }{9}-\frac{2 b^2}{3}-\frac{4 a b}{9}+4\zeta\frac{\zeta\!+\!2}{81}a^2+\cdots\Big).
\end{equation*}
Assuming $b = p_1 a + p_2 a^2 + p_3 a^3 + p_4 a^4 + p_5 a^5 + O(a^6)$ and plugging it into the above equation, we obtain
\begin{equation*}
\begin{aligned}
O(a^7)=\;&(\zeta\!+\!1) \frac{a^2}{9} (6 p_1\!-\!\zeta) + (\zeta\!+\!1) \frac{2a^3}{81} \big(2 \zeta  (\zeta\!+\!2)-9 p_1 (3 p_1\!+\!2)+27 p_2\big)+\cdots,
\end{aligned}
\end{equation*}
from which we obtain the following solution
\begin{equation}
\label{expansion for b}
\begin{aligned}
&p_1 = \frac{\zeta }{6},\quad p_2 = -\zeta\frac{5 \zeta+4}{108},\quad p_3 = \zeta\frac{15 \zeta ^2\!+\!26 \zeta\!+\!12}{648},\quad p_4=\frac{167 \zeta ^3\!+\!452 \zeta ^2\!+\!424 \zeta\!+\!136}{-11664/\zeta},\\
&p_5=\frac{693 \zeta ^4\!+\!2561 \zeta ^3\!+\!3644 \zeta ^2\!+\!2348 \zeta \!+\!576}{69984/\zeta}.
\end{aligned}
\end{equation}
By plugging \eqref{expansion for b} into \eqref{transformed equation in terms of a,b,c,d -- 2} and Taylor-expanding it around $a=0$, we obtain
\begin{equation*}
\Gamma^{-1}=\frac{\zeta\!+\!1}{648} a^2 - \frac{(\zeta\!+\!1)^2}{648} a^3+\frac{ 187 \zeta^2\!+\!376\zeta\!+\!196 }{139968/(\zeta\!+\!1)}a^4-\frac{79 \zeta^2\!+\!160\zeta\!+\!88}{69984/(\zeta\!+\!1)^2}a^5+O\big(a^6\big).
\end{equation*}
from which we obtain
\begin{equation*}
\begin{aligned}
a&=\frac{18 \sqrt{2/\Gamma}}{\sqrt{\zeta\!+\!1}}+\frac{324}{\Gamma }+\sqrt{2}\frac{83 \zeta ^2\!+\!164 \zeta\!+\!74}{\big(\Gamma(\zeta\!+\!1)\big)^{3/2}/27}+\frac{29 \zeta ^2\!+\!56 \zeta\!+\!20}{\Gamma ^2 (\zeta\!+\!1)/972}+O\Big(\frac{1}{\Gamma ^{5/2}}\Big)\\
&=\frac{12 \sqrt{6 \pi/m}}{\sqrt{\gamma  (\zeta\!+\!1)}}+\frac{432 \pi }{\gamma  m}+24\sqrt{6}\frac{ 83 \zeta ^2\!+\!164 \zeta\!+\!74}{\big(\gamma  (\zeta\!+\!1) m/\pi\big)^{3/2}}+1728\frac{ 29 \zeta ^2\!+\!56 \zeta\!+\!20}{\gamma ^2 (\zeta\!+\!1) m^2/\pi ^2}+O\Big(\frac{1}{m^{5/2}}\Big),
\end{aligned}
\end{equation*}
where the second equality is due to $\Gamma=3m\gamma/(4\pi)$. According to the definition of $a$ and $b$, we have
\begin{equation*}
R_0=\sqrt[3]{\frac{3m}{4\pi a}}, \quad R_1 = (1\!+\!b)\sqrt[3]{\frac{3m}{4\pi a}}, \quad R_2 = \sqrt[3]{a\!+\!(1\!+\!b)^3}\sqrt[3]{\frac{3m}{4\pi a}}, \quad R_3 = \sqrt[3]{1\!+\!a(\zeta\!+\!1)}\sqrt[3]{\frac{3m}{4\pi a}},
\end{equation*}
which allow us to compute the asymptotics of $E(U,V)$, $R_{i+1}\!-\!R_i$, $(R_{1}\!+\!R_2)/2$, and $\big(R_3^3\!-\!R_2^3\big)-\big(R_1^3\!-\!R_0^3\big)$.
\end{proof}

\begin{proof} [of Proposition \ref{equal volume asymptotics proposition}]
Our proof is similar to \cite[Proof of Theorem 11]{van2008copolymer}. Define $\kappa>0$ via $\kappa^{-n}=(R_1^n+R_2^n)/2$, then according to \eqref{equal volume constraint}, \eqref{volume constraint radial candidates} and \eqref{2d 3d volume constraint radial candidates}, we obtain
\begin{equation*}
\begin{aligned}
(R_0^2\,,\,R_1^2\,,\,R_2^2\,,\,R_3^2)&=\kappa^{-2}+(-\zeta\!-\!1\,,\,-1\,,\,1\,,\,\zeta\!+\!1)\,m/(2 \pi),&&\text{for}\;n=2,\\
(R_0^3\,,\,R_1^3\,,\,R_2^3\,,\,R_3^3)&=\kappa^{-3}+(-\zeta\!-\!1\,,\,-1\,,\,1\,,\,\zeta\!+\!1)\,(3 m)/(8 \pi),&&\text{for}\;n=3.
\end{aligned}
\end{equation*}

\noindent\textbf{The 2-D case ($n=2$)}\\
Using a change of variables $\kappa=2\pi t/m$, we obtain the following asymptotics for $E(U,V)$ given by Proposition \ref{proposition: energy for liposome candidates}:
\begin{equation}
\label{2d energy with respect to thickness t}
\frac{E(U,V)}{m}=\Big(\frac{\zeta\!+\!1}{24}\gamma t^2\!+\!\frac{2}{t}\Big)+\pi ^2 t^3\frac{\gamma(\zeta\!+\!1)(\zeta^2\!+\!4\zeta\!+\!6)\,t^3\!-\!60}{60 m^2}+O\Big(\frac1{m^4}\Big).
\end{equation}
Minimizing $\gamma(\zeta\!+\!1)t^2/24\!+\!2/t$ with respect to $t$ ($t>0$) yields $t=2\sqrt[3]{3/(\gamma\zeta\!+\!\gamma)}$ and
\begin{equation*}
\frac{E(U,V)}{m} = \sqrt[3]{\gamma\frac{\zeta\!+\!1}{8/9} }+24\pi^2\frac{2 \zeta ^2\!+\!8 \zeta\!+\!7}{5 \gamma  (\zeta\!+\!1) m^2}+O\Big(\frac1{m^4}\Big).
\end{equation*}
In order to obtain a more detailed asymptotics of $t$, we differentiate \eqref{2d energy with respect to thickness t} with respect to $t$ and find its root. Using the ansatz $t=2\sqrt[3]{3/(\gamma\zeta\!+\!\gamma)}+Cm^{-2}+O\big(m^{-4}\big)$, where $C$ is independent of $m$, we obtain
\begin{equation*}
\frac{1}{20 m^2}\Big(48 \pi ^2\frac{ 4 \zeta ^2\!+\!16 \zeta\!+\!19}{\big(\gamma (\zeta\!+\!1)/3\big)^{2/3}}+5 \gamma  (\zeta\!+\!1) C\Big)=O\Big(\frac1{m^4}\Big),
\end{equation*}
from which we can solve for $C$,
\begin{equation*}
C=-\frac{16}{5} \pi ^2\frac{ 4 \zeta ^2\!+\!16 \zeta\!+\!19}{\big(\gamma  (\zeta\!+\!1)/3\big)^{5/3}}.
\end{equation*}
We are then able to compute the asymptotics of $R_{i+1}\!-\!R_i$ and $(R_{1}\!+\!R_2)/2$.

\vspace{3pt}
\noindent\textbf{The 3-D case ($n=3$)}\\
Using a change of variables $\kappa= \sqrt{4 \pi t/m}$, we obtain the following asymptotics for $E(U,V)$ given by Proposition \ref{proposition: energy for liposome candidates}:
\begin{equation}
\label{3d energy with respect to thickness t}
\frac{E(U,V)}{m}=\Big(\frac{\zeta\!+\!1}{24} \gamma t^2+\frac{2}{t}\Big)+\pi t^2\frac{7  \gamma  (\zeta \!+\!1)(\zeta ^2\!+\!4 \zeta\!+\!6)\, t^3\!-\!240}{120m}+O\Big(\frac{1}{m^2}\Big).
\end{equation}
Minimizing $\gamma(\zeta\!+\!1)t^2/24\!+\!2/t$ with respect to $t$ ($t>0$) yields $t=2\sqrt[3]{3/(\gamma\zeta\!+\!\gamma)}$ and
\begin{equation*}
\frac{E(U,V)}{m} =  \sqrt[3]{\gamma\frac{\zeta\!+\!1}{8/9} }+\frac{4 \pi}{5 m}\frac{7 \zeta ^2\!+\!28 \zeta\!+\!32}{\big(\gamma  (\zeta\!+\!1)/3\big)^{2/3}}+O\Big(\frac{1}{m^2}\Big).
\end{equation*}
In order to obtain a more detailed asymptotics of $t$, we differentiate \eqref{3d energy with respect to thickness t} with respect to $t$ and find its root. Using the ansatz $t=2\sqrt[3]{3/(\gamma\zeta\!+\!\gamma)}+Cm^{-1}+O\big(m^{-2}\big)$, where $C$ is independent of $m$, we obtain
\begin{equation*}
\frac{1}{4 m}\Big(8 \pi  \frac{7 \zeta ^2\!+\!28 \zeta\!+\!38}{\big(\gamma(\zeta\!+\!1)/3\big)^{1/3}}+\gamma  (\zeta\!+\!1) C\Big)=O\Big(\frac1{m^2}\Big),
\end{equation*}
from which we can solve for $C$,
\begin{equation*}
C=-\frac{8}{3} \pi  \frac{7 \zeta ^2\!+\!28 \zeta\!+\!38}{\big(\gamma(\zeta\!+\!1)/3\big)^{4/3}}.
\end{equation*}
We are then able to compute the asymptotics of $R_{i+1}\!-\!R_i$ and $(R_{1}\!+\!R_2)/2$.

\end{proof}

\section{Asymptotics with 1-Wasserstein distance}
\label{appendix Asymptotics with 1-Wasserstein distance}
In this appendix we make some clarification of Remark \ref{remark on asymptotics of liposome candidates}-\CircleAroundChar{3}. We are interested in $(R_1\!-\!R_0)-(R_3\!-\!R_2)$, which is the difference in the thickness between the inner and outer $V$ layers. For $n=2$, such a difference is $4 \pi\rho^2(\zeta\!+\!2)\zeta/(\zeta m\!+\!m)$ in Corollary \ref{rescaled liposome asymptotics}, and is $12 \pi\rho^2(\zeta\!+\!2)\zeta/(\zeta m\!+\!m)$ in Corollary \ref{rescaled liposome asymptotics under equal mass assumption}. The latter is exactly three times the former. Such a relation is also true for $n=3$. As we explain below, this is still true for $n=2$ in a variant model, where the Coulombic nonlocal term is replaced by the 1-Wasserstein distance (the case of $n=3$ in this variant model is unclear to us since the thickness of $V$ layers is not explicitly provided in \cite{lussardi2014variational}).

For this variant model, a candidate is constructed for the lim-sup inequality in \cite[Section 8.1]{peletier2009partial}. Using the notation in \cite{peletier2009partial}, $\kappa(s)$ in \cite[Figure 8]{peletier2009partial} is negative. The $V$ layers is given by the following \cite[Section 8.1]{peletier2009partial}:
\begin{equation*}
\text{supp}(v_\veps) = \big\{\psi_+(q,t):0\leqslant t\leqslant \mathfrak t_+(q,1)\big\}\cup\big\{\psi_-(q,t):\mathfrak t_-(q,-1)\leqslant t\leqslant 0\big\},
\end{equation*}
where
\begin{equation*}
\psi_+(q,t):=\tilde\gamma_+(q)+t\tilde v_+(q),\quad \psi_-(q,t):=\tilde\gamma_-(q)+t\tilde v_-(q),
\end{equation*}
\begin{equation*}
\mathfrak t_+(q,m)=\Big(1\!-\!\sqrt{1\!-\!2\veps\tilde\kappa_+m}\Big)\Big/\tilde\kappa_+\,,\quad \mathfrak t_-(q,m)=\Big(1\!-\!\sqrt{1\!-\!2\veps\tilde\kappa_-m}\Big)\Big/\tilde\kappa_-\,,
\end{equation*}
\begin{equation*}
\tilde\kappa_+=\kappa/(1\!-\!\veps\kappa),\quad \tilde\kappa_-=\kappa/(1\!+\!\veps\kappa).
\end{equation*}
Therefore, the thickness of the inner $V$ layer is
\begin{equation*}
\big(1/| \kappa|\!-\!\veps\big) \Big(1-\sqrt{3\!-\!2/\big(1\!-\!\veps| \kappa|\big)}\Big)=\veps+\frac{\veps^2 |\kappa|}{2}+O\big(\veps^3\big),
\end{equation*}
and the thickness of the outer $V$ layer is
\begin{equation*}
\big(1/| \kappa|\!+\!\veps\big) \Big(\sqrt{3\!-\!2/\big(1\!+\!\veps  | \kappa|\big)}-1\Big)=\veps-\frac{\veps^2 |\kappa|}{2}+O\big(\veps^3\big).
\end{equation*}
If we require the inner and outer $V$ layers to have the same mass, then their thickness would be $\veps+3|\kappa|\veps^2/2+O\big(\veps^3\big)$ and $\veps-3|\kappa|\veps^2/2+O\big(\veps^3\big)$, respectively, with the difference being exactly three times that of the above lim-sup candidate.

\bibliographystyle{apalike}
\bibliography{main}

\end{document}